\documentclass[11pt]{article}
\usepackage{amssymb}
\usepackage{amsmath,amsfonts,hyperref}
\usepackage{amsmath, latexsym, amsfonts, amssymb, amsthm, amscd, graphics, xcolor,graphicx}
\newtheorem {theoreme} {{\bf Theorem}} [section]
\newtheorem {lemme}[theoreme]{{\bf Lemma}}
\newtheorem {proposition}[theoreme]{{\bf Proposition}}
\newtheorem {remarque} {Remark}
\newtheorem {corollaire}[theoreme]{{\bf Corollary}}

\newcommand{\1}{1\!\!{\sf I}}
\newcommand{\tr}{\operatorname{tr}}
\newcommand{\Tr}{\operatorname{Tr}}
\newcommand{\N}{{\mathbb N}}
\newcommand{\R}{{\mathbb R}}
\newcommand{\C}{{\mathbb C}}
\newcommand{\vers}{\mathop{\longrightarrow}}
\numberwithin{equation}{section}
\title{Spectral properties of polynomials in independent Wigner  and  deterministic matrices
}

\author{S. T. Belinschi\thanks{CNRS, Institut de Math\'ematiques de Toulouse, Universit\'e Paul Sabatier, 118 rte de Narbonne,
  F-31062 Toulouse Cedex 09. 
E-mail: serban.belinschi@math.univ-toulouse.fr } and M. Capitaine\thanks{CNRS, Institut de Math\'ematiques de Toulouse, Universit\'e Paul Sabatier, 118 rte de Narbonne,
  F-31062 Toulouse Cedex 09. 
E-mail: mireille.capitaine@math.univ-toulouse.fr } }

\date{}
\begin{document}
\maketitle
\begin{abstract}  On the one hand, we prove that almost surely, for large dimension,  there is no eigenvalue of a Hermitian polynomial  in independent Wigner and  deterministic matrices, in any  interval lying at  some distance from  the supports  of a sequence of  deterministic probability measures, which is computed with the tools of free probability. On the other hand, we establish the strong asymptotic freeness of independent Wigner matrices and any family of deterministic matrices with strong limiting distribution.

\end{abstract}


\section{Introduction}
Free probability theory was introduced by Voiculescu around 1983 motivated by 
the isomorphism problem of von Neumann algebras of free groups. He developed a noncommutative probability theory, on a noncommutative
probability space, in which a new notion of freeness plays the
role of independence in classical probability. Around 1991, Voiculescu \cite{V}  threw a bridge connecting random matrix theory with free probability 
since he realized that the freeness property is also present for many classes of random matrices, in the
asymptotic regime when the size of the matrices tends to infinity. Since then, random matrices have played a key role in operator  algebra  whereas tools developed in operator algebras and free
probability theory could now be applied to random matrix problems.\\
For the reader's convenience, we recall the following basic definitions from free probability theory. For a thorough introduction to free probability theory, we refer to \cite{VDN}.
\begin{itemize}
\item A ${\cal C}^*$-probability space is a pair $\left({\cal A}, \tau\right)$ consisting of a unital $ {\cal C}^*$-algebra ${\cal A}$ and 
a linear map $\tau: {\cal A}\rightarrow \mathbb{C}$ such that $\tau(1_{\cal A})=1$ and $\tau(aa^*)\geq 0$ for all $a \in {\cal A}$. $\tau$ is a trace if it satisfies $\tau(ab)=\tau(ba)$ for every $(a,b)\in {\cal A}^2$. A trace is said to be faithful if $\tau(aa^*)>0$ whenever $a\neq 0$. 
An element of ${\cal A}$ is called a noncommutative random variable. 
\item The noncommutative distribution of a family $a=(a_1,\ldots,a_k)$ of noncommutative random variables in a ${\cal C}^*$-probability space $\left({\cal A}, \tau\right)$ is defined as the linear functional $\mu_a:P\mapsto \tau(P(a,a^*))$ defined on the set of polynomials in $2k$ noncommutative indeterminates, where $(a,a^*)$ denotes the $2k$-tuple $(a_1,\ldots,a_k,a_1^*,\ldots,a_k^*)$.
For any self-adjoint element $a_1$ in  ${\cal A}$,  there exists a probability  measure $\nu_{a_1}$ on $\mathbb{R}$ such that,   for every polynomial P, we have
$$\mu_{a_1}(P)=\int P(t) \mathrm{d}\nu_{a_1}(t).$$
Then,  we identify $\mu_{a_1}$ and $\nu_{a_1}$. If $\tau$ is faithful then the  support of $\nu_{a_1}$ is the spectrum of $a_1$  and thus  $\|a_1\| = \sup\{|z|, z\in \rm{support} (\nu_{a_1})\}$. 
\item A family of noncommutative random variables $(a_i)_{i\in I}$ in a ${\cal C}^*$-probability space  $\left({\cal A}, \tau\right)$ is free if for all $k\in \mathbb{N}$ and all polynomials $p_1,\ldots,p_k$ in two noncommutative indeterminates, one has 
\begin{equation}\label{freeness}
\tau(p_1(a_{i_1},a_{i_1}^*)\cdots p_k (a_{i_k},a_{i_k}^*))=0
\end{equation}
whenever $i_1\neq i_2, i_2\neq i_3, \ldots, i_{k-1}\neq i_k$ and $\tau(p_l(a_{i_l},a_{i_l}^*))=0$ for $l=1,\ldots,k$.
\item A family $(x_i)_{i\in I}$ of noncommutative random variables in a ${\cal C}^*$-probability space  $\left({\cal A}, \tau\right)$ is a semicircular system if
 $x_i=x_i^*$ for all $i\in I$, $(x_i)_{i\in I}$
is a free family and for any $k\in \mathbb{N}$, $$\tau(x_i^k)= \int t^k d\mu_{sc}(t)$$
where $d\mu_{sc}(t)=
\frac{1}{2\pi} \sqrt{4-t^2}{\1}_{[-2;2]}(t) dt$ is the semicircular standard distribution.
\item Let $k$ be a nonnull integer number. Denote by ${\cal P}$ the set of polynomials in $2k $ noncommutative indeterminates.
A sequence of families of variables $ (a_n)_{n\geq 1}  =
(a_1(n),\ldots, a_k(n))_{n\geq 1}$ in ${\cal C}^* $-probability spaces 
$\left({\cal A}_n, \tau_n\right)$ converge, when n goes to infinity, respectively  in distribution if the map 
$P\in {\cal P} \mapsto
\tau_n(
P(a_n,a_n^*))$ converges pointwise
and strongly in distribution if moreover the map 
$P\in {\cal P} \mapsto  \Vert P(a_n,a_n^*) \Vert$  converges pointwise.
\end{itemize}

Voiculescu  considered random matrices in this noncommutative probability context. Let ${\cal A}_n$ be 
the algebra of $n\times n$ matrices
whose entries are random variables with finite moments and endow this algebra with
the expectation of the normalized trace defined for any $M\in {\cal A}_n$ by 
$\tau_n(M) = \mathbb{E}[\frac{1}{n}\Tr(M)]$. Let us 
consider r  independent $n\times n$ so-called G.U.E matrices, that is to say  random  Hermitian matrices $X_n^{(v)} =  [X^{(v)}_{jk}]_{j,k=1}^n$, $v=1,\ldots,r$,   for which the  random variables $(X^{(v)}_{ii})$,
$(\sqrt{2} Re(X^{(v)}_{ij}))_{i<j}$, $(\sqrt{2} Im(X^{(v)}_{ij}))_{i<j}$ are independent centred gaussian distribution with variance $1$.
Voiculescu  discovered that these independent G.U.E  matrices  are asymptotically free and provide a model for a free semi-circular system since he established that
for any polynomial P in r noncommutative indeterminates,
\begin{equation}\label{sansD}\tau_n\left\{  P\left(\frac{ X_n^{(1)}}{\sqrt{n}},\ldots,\frac{ X_n^{(r)}}{\sqrt{n}}\right)\right\} \rightarrow_{n\rightarrow +\infty} \tau \left( P(x_1,\ldots,x_r)\right)\end{equation}
where $(x_1,\ldots,x_r)$ is a semicircular system in some ${\cal C}^*$-probability space $({\cal A}, \tau)$. 
It turns out that this result still holds on the one hand  almost surely dealing with the normalized trace $\tau_n(M) = \frac{1}{n}\Tr(M)$ instead of the mean normalized trace \cite{T}  and on the other hand
for non-gaussian Wigner matrices \cite{D}.\\

In \cite{HT}, Haagerup and Thorbj{\o}rsen proved the strong asymptotic freeness of independent G.U.E matrices, namely: almost surely,
for any polynomial P in r noncommutative indeterminates,
\begin{equation}\label{strong}\left\| P\left(\frac{ X_n^{(1)}}{\sqrt{n}},\ldots,\frac{ X_n^{(r)}}{\sqrt{n}}\right)\right\| \rightarrow_{n\rightarrow +\infty} \Vert P(x_1,\ldots,x_r) \Vert.\end{equation}
Note that this result led to the proof of the very important result in the theory of operator algebras that $Ext({\cal C}^*_{red}(F_2))$ is not  a group.
\eqref{strong} was proved for Gaussian random matrices with real or symplectic entries 
 by Schultz \cite{Schultz05}, for Wigner matrices with symmetric distribution of the entries satisfying a Poincar\'e inequality  by Capitaine and Donati-Martin \cite{CD07} and for Wigner matrices under i.i.d assumptions and  fourth moment hypotheses by Anderson \cite{A}.\\

The  result \eqref{sansD} of Voiculescu is actually more general since Voiculescu \cite{V,Voiculescu97} proved the asymptotic freeness of independent 
G.U.E matrices with an extra family of deterministic  matrices  with limiting distribution. 
Male \cite{CamilleM} established the strong  asymptotic freeness of independent 
G.U.E matrices with an extra independent family of  matrices  with strong limiting distribution. Note that in \cite{MCo}, Collins and Male established the  strong asymptotic freeness of Haar-distributed random unitary matrices 
and deterministic matrices with strong limiting distribution. \\
In this paper, on the one hand,  we 
prove the strong asymptotic freeness of independent non-Gaussian Wigner matrices and an extra family of deterministic matrices with strong limiting distribution (see Theorem \ref{thprincipal}). \\

On the other hand we prove that a sequence of deterministic measures plays a central role in the study of the
spectrum of a Hermitian  polynomial in independent Wigner matrices and deterministic matrices. Roughly speaking, these measures
are in
some sense obtained by  taking partially the limit when n goes to infinity, only for the Wigner matrices. 
We establish  that almost surely, for large n, each interval lying at  some distance from  the supports of these  deterministic measures contains no eigenvalue of the Hermitian  polynomial (see Theorem \ref{noeigenvalue}). This type of result was first established
by Bai and Silverstein in  \cite{BaiSil98} for sample covariance matrices and in \cite{BaiSil12} for Information-plus-noise type models, and by Capitaine,  Donati-Martin, F\'eral and F\'evrier  in \cite{CDMFF} for additive deformations of a Wigner matrix.
\\

 Here are the matricial models we deal with. Let $t$ and $r$ be fixed nonull integer numbers independent from $n$.\\

\noindent

 
\noindent $\bullet$  $(A_n^{(1)},\ldots,A_n^{(t)})$ is a $t-$tuple of  $n\times n$   deterministic matrices   such that
  for any $u=1,\ldots,t$, \begin{equation}\label{normeAn} \sup_n \Vert A_n^{(u)} \Vert< \infty,\end{equation} where $\Vert \cdot \Vert$ denotes the spectral norm. \\

\noindent $\bullet$  We consider r  independent $n\times n$ random  Hermitian Wigner matrices $X_n^{(v)} =  [X^{(v)                }_{ij}]_{i,j=1}^n$, $v=1,\ldots,r$,  where, for each $v$,
$ [X^{(v)  }_{ij}]_{i\geq1,j\geq 1}$ is an infinite array of  random variables such that $X^{(v)}_{ii}$,
$\sqrt{2}\Re(X^{(v)}_{ij}), i<j$, $\sqrt{2} \Im(X^{(v)}_{ij}), i<j$, are independent,  centred with variance 1 and satisfy:
\begin{enumerate} \item There exists a $K_v$ and a random variable $Z^{(v)}$ with finite fourth moment for which there exists $x_0>0$ and  an  integer number $n_0>0$ such that, for any $x >x_0$ and any integer number $n >n_0$, we have
\begin{equation}\label{condition}\frac{1}{n^2} \sum_{1\leq i,j\leq n}\mathbb{P}\left( \vert X^{(v)}_{ij}\vert >x\right) \leq K_v\mathbb{P}\left(\vert Z^{(v)} \vert>x\right).\end{equation}
\item $$\sup_{1\leq i<j}\mathbb{E}(\vert X^{(v)}_{ij}\vert^3)<+\infty. $$
\end{enumerate}
\begin{remarque}
The independence of the real and imaginary parts of the entries of the Wigner matrices is necessary in our approach since, after conditioning, we use Lemma \ref{lem1} for the real and imaginary parts of each entry.
\end{remarque}
\begin{remarque}
Note that assumption such as \eqref{condition} appears in \cite{CSBD}. It obviously holds if the $X^{(v)}_{ii}$,
$ \sqrt{2}\Re(X^{(v)}_{ij}),{i<j}$, $\sqrt{2} \Im(X^{(v)}_{ij}),{i<j}$, are identically distributed with finite fourth moment.
\end{remarque}
\noindent Here are our main results.
\begin{theoreme}\label{noeigenvalue}
Let $({\cal A},  \tau)$ be a ${\cal C}^*$-probability space
 equipped with a faithful tracial state and 
 $x=(x_1,\ldots,x_r)$ be  a semi-circular system in  $({\cal A},  \tau)$.
Let $a_n=(a_n^{(1)},\ldots,a_n^{(t)})$ be a t-tuple of noncommutative   random variables which is free from  $x$ in $({\cal A},\tau)$ and such that the distribution of  $a_n$  in $({\cal A},\tau)$ coincides with the distribution of $(A_n^{(1)},\ldots, A_n^{(t)})$ in $({ M}_n(\mathbb{C}), \frac{1}{n}\Tr)$. Let $P$ be a self-adjoint  polynomial  in $2t+r$ noncommutative indeterminates $X_1, \ldots, X_r$, where for any $i=1,\ldots,r$, $X_i=X_i^*$, and $X_{r+1},\ldots, X_{r+t}, X_{r+1}^*,\ldots, X_{r+t}^*$. Let $[b,c]$ be a real interval  such  that there exists $\delta>0$ such that, for any large  $n$,   $[b-\delta,c+\delta]$ lies outside the support of the distribution of  the noncommutative random variable  $ P\left(x_1,\ldots,x_r, a_n^{(1)},\ldots,a_n^{(t)},(a_n^{(1)})^*,\ldots,(a_n^{(t)})^*\right)$ in $({\cal A},\tau)$.
  Then, 
almost surely,  for all large $n$, there is no eigenvalue of  the $n\times n$  matrix $ P\left(\frac{X_n^{(1)}}{\sqrt{n}},\ldots,\frac{X_n^{(r)}}{\sqrt{n}}, A_n^{(1)},\ldots,A_n^{(t)}, (A_n^{(1)})^*,\ldots,(A_n^{(t)})^*\right)$ in $[b,c].$
\end{theoreme}
\begin{remarque}
When $r=t=1$,  $A_n^{(1)}=(A_n^{(1)})^*$ and $P(X_1,X_2,X_2^*)=X_1+\frac{X_2+X_2^*}{2}$, the distribution of $P(x_1,a_n^{(1)},(a_n^{(1)})^*)$ is the so-called free convolution $\mu_{sc}\boxplus \mu_{A_n^{(1)}}$ where  $\mu_{A_n^{(1)}}=\frac{1}{n} \sum_{i=1}^n \lambda_i(A_n^{(1)})$, denoting by 
$ \lambda_i(A_n^{(1)})$, $i=1,\ldots,n$, the eigenvalues of  $A_n^{(1)}$.
\end{remarque}
\begin{theoreme}\label{thprincipal} Let $({\cal A},  \tau)$ be  a ${\cal C}^*$-probability space
 equipped with a faithful tracial state.
Let  $x=(x_1,\ldots,x_r)$ be  a semi-circular system and  $a=(a_1,\ldots,a_t)$ be  a t-tuple of noncommutative   random variables which is free from $x$ in $({\cal A},\tau$).\\
Assume moreover that  $(A_n^{(1)},\ldots,A_n^{(t)}, (A_n^{(1)})^*,\ldots,(A_n^{(t)})^*)$  converges strongly towards $a=(a_1,\ldots,a_t, a_1^*,\ldots,a_t^*)$ in
 $({\cal A},  \tau)$, that is
for any polynomial P in 2t noncommutative indeterminates,\\

$\frac{1}{n} \Tr P\left( A_n^{(1)},\ldots,A_n^{(t)}, (A_n^{(1)})^*,\ldots,(A_n^{(t)})^*\right)$ $$ \rightarrow_{n\rightarrow +\infty} \tau \left( P(a_1,\ldots,a_t,a_1^*,\ldots,a_t^*\right)$$ and \\

$\left\|P\left( A_n^{(1)},\ldots,A_n^{(t)}, (A_n^{(1)})^*,\ldots,(A_n^{(t)})^*\right)\right\| $ $$\rightarrow_{n\rightarrow +\infty} \left\| P(a_1,\ldots,a_t,a_1^*,\ldots,a_t^*)\right\|_{\cal A}.$$
 Then, almost surely, for any  polynomial $P$ \hspace*{-0.1cm}in $r+2t $ \hspace*{-0.16cm} noncommutative variables, 
$$ \lim_{n\rightarrow +\infty} \frac{1}{n}\Tr P\left(\frac{X_n^{(1)}}{\sqrt{n}},\ldots,\frac{X_n^{(r)}}{\sqrt{n}}, A_n^{(1)},\ldots,A_n^{(t)}, (A_n^{(1)})^*,\ldots,(A_n^{(t)})^*\right) $$\begin{equation}\label{af} =\tau \left( P\left(x_1,\ldots,x_r, a_1,\ldots,a_t, a_1^*,\ldots,a_t^*\right) \right)\end{equation} and 
$$\lim_{n\rightarrow +\infty} \left\| P\left(\frac{X_n^{(1)}}{\sqrt{n}},\ldots,\frac{X_n^{(r)}}{\sqrt{n}}, A_n^{(1)},\ldots,A_n^{(t)},(A_n^{(1)})^*,\ldots,(A_n^{(t)})^*\right) \right\|$$\begin{equation} \label{saf} =\left\| P\left(x_1,\ldots,x_r, a_1,\ldots,a_t, a_1^*,\ldots,a_t^*\right) \right\|_{\cal A}.\end{equation}
\end{theoreme}
\begin{remarque}{ Note that it is sufficient to prove Theorem \ref{noeigenvalue} and  Theorem \ref{thprincipal} for Hermitian matrices $A_n^{(1)},\ldots,A_n^{(t)}$ by considering their Hermitian and anti-Hermitian parts, so that throughout the paper we assume that the $A_n^{(i)}$'s are Hermitian.}\end{remarque}

\noindent We adopt the   strategy  from  \cite{HT} and \cite{Schultz05} based on a linearization trick and  sharp estimates on matricial Stieltjes transforms. More precisely,  both proofs of Theorem \ref{thprincipal} and Theorem \ref{noeigenvalue} are based on  the following key Lemma \ref{inclu2}.  First, note that the algebra $M_m(\C)\otimes {\cal A}$ formed by the $m\times m $ matrices with coefficients in ${\cal A}$, inherits the structure of ${\cal C}^*$-probability space with trace $\frac{1}{m} \Tr_m \otimes \tau$ and norm 
$$\Vert b \Vert = \lim_{k\rightarrow +\infty} \left( \frac{1}{m} \Tr_m \otimes \tau \left[(b^*b)^k\right]\right)^{\frac{1}{2k}}, \; \forall b \in M_m(\C)\otimes {\cal A} .$$

\begin{lemme}  \label{inclu2} Let $({\cal A},  \tau)$ be a ${\cal C}^*$-probability space
 equipped with a faithful tracial state and 
 $x=(x_1,\ldots,x_r)$ be  a semi-circular system in  $({\cal A},  \tau)$. Let $a_n=(a_n^{(1)},\ldots,a_n^{(t)})$ be a t-tuple of noncommutative self-adjoint  random variables which is free from $x$ in $({\cal A},\tau$), such that the distribution of  $a_n$ coincides with the distribution of $(A_n^{(1)},\ldots, A_n^{(t)})$ in $({ M}_n(\mathbb{C}), \frac{1}{n}\Tr_n)$. Then, for all $m \in \N$, all self-adjoint matrices $\gamma, \alpha_1, \ldots,\alpha_r, \beta_1, \ldots, \beta_t$ of size $m\times m$ and
all $\epsilon >0$, almost surely, for all large $n$, we have\\

\noindent $
sp(\gamma \otimes I_n + \sum_{v=1}^r \alpha_v \otimes \frac{X_n^{(v)}}{\sqrt{n}}+  \sum_{u=1}^t \beta_u \otimes A_n^{(u)})$ \begin{equation} \label{spectre3} \subset
sp(\gamma \otimes 1_{\cal A} + \sum_{v=1}^r \alpha_v \otimes x_v + \sum_{u=1}^t \beta_u \otimes a_n^{(u)}) + ]-\epsilon, \epsilon[.
\end{equation}
 Here, $sp(T)$ denotes the spectrum of the operator $T$, $I_n$ the identity matrix and $1_{\cal A}$ denotes the unit of ${\cal A}$.
\end{lemme}

\begin{remarque}\label{remarqueinversible}
By a density argument, it is sufficient to prove Lemma \ref{inclu2} for invertible  self-adjoint matrices $\gamma, \alpha_1, \ldots,\alpha_r, \beta_1, \ldots, \beta_t$. This invertibility will be used in the proof of Lemma \ref{inversion}.
\end{remarque}
 The proof of \eqref{spectre3} requires sharp estimates of  $g_n(z)-\tilde g_n(z)$
where for $z\in \mathbb{C}\setminus \mathbb{R}$, $$g_n(z) =\mathbb{E} \frac{1}{m} \Tr_m \otimes \frac{1}{n} \Tr_n [ (zI_m \otimes I_n - \gamma \otimes I_n - \sum_{v=1}^r \alpha_v \otimes \frac{X_n^{(v)}}{\sqrt{n}}(\omega)-  \sum_{u=1}^t \beta_u \otimes A_n^{(u)})^{-1}]$$ and $$\tilde g_n(z) = \frac{1}{m} \Tr_m \otimes \tau [ (zI_m \otimes 1_{\cal A} - \gamma \otimes 1_{\cal A} - \sum_{v=1}^r \alpha_v \otimes x_v - \sum_{u=1}^t \beta_u \otimes a_n^{(u)})^{-1}].$$
 More precisely we are going to establish that 
 there exists a polynomial $Q$ with nonnegative coefficients such that, for $z \in \C \setminus \mathbb{R}$, 
\begin{equation} \label{estimdiffeqno}
\left|g_n(z)-\tilde g_n(z)+{\tilde{E}_n(z)}\right|\leq  \frac{Q(\vert \Im  z\vert^{-1})}{n\sqrt{n}},
\end{equation}
where $\tilde{E}_n$ is the Stieltjes transform of a compactly supported distribution $\nabla_n$ on $\mathbb{R}$ whose support is
included in the spectrum of $\gamma \otimes 1_{\cal A} +\sum_{v=1}^r \alpha_v \otimes x_v +\sum_{u=1}^t\beta_u \otimes a_n^{(u)}$ and such that $\nabla_n(1)=0$.\\
 But this required sharp estimate  makes necessary a technical piece of work and  a fit  use of free operator-valued subordination maps (see Section \ref{free}). In particular, we need an explicit development of the Stieltjes transform up to the order $\frac{1}{n\sqrt{n}}$ but the stability under perturbation argument used in \cite{CamilleM} does not provide  this development from the approximate matricial subordination equation. Therefore we use a strategy based on an invertibility property of   matricial subordination maps related to semi-circular system (see Lemma \ref{inversion}).\\
 
 Theorem \ref{thprincipal} can be deduced from Lemma \ref{inclu2} by  following the proofs in  \cite{HT} and \cite{Schultz05}. 
Given a noncommutative polynomial $P$, choosing in Lemma \ref{inclu2} the $\gamma$, $(\alpha_v)_{v=1,\ldots,r}$, $(\beta_u)_{u=1,\ldots,t}$ corresponding to a self-adjoint linearization of $P$ as defined in \cite{A} allows to deduce  Theorem \ref{noeigenvalue}. \\

In  Section \ref{troncation}, we explain why, using a truncation and Gaussian convolution procedure, it is sufficient  to prove  Theorem \ref{thprincipal}  and Theorem \ref{noeigenvalue}
 when we assume that the $X_{ij}^{(v)}$'s satisfy:
\begin{itemize} 
\item[(H)] the  variables $\sqrt{2}\Re  X_{ij}^{(v)}$, $\sqrt{2}\Im   X_{ij}^{(v)}$, $1\leq i<j$, $X_{ii}^{(v)}$, $i\geq1$, $v=1,\ldots,r$, are independent, centred with variance 1 and satisfy a Poincar\'e inequality with common constant  $C_{PI}$.
\end{itemize}

\noindent Note that, according to Corollary 3.2 in \cite{L},  (H) implies that  for any $p\in \mathbb{N}$,  \begin{equation}\label{moments}\max_{v=1,\ldots,r} \sup_{i\geq 1, j\geq 1} \mathbb{E}\left(\vert  X_{ij}^{(v)}\vert^p\right) <+\infty.\end{equation}

\noindent We also explain how the proof of Theorem \ref{thprincipal} 
 may be  reduced to prove that
for any  polynomial $P$ in $r+t $ noncommutative variables, almost surely, 
\begin{equation} \label{safbis}\lim_{n\rightarrow +\infty} \left\| P\left(\frac{X_n^{(1)}}{\sqrt{n}},\ldots,\frac{X_n^{(r)}}{\sqrt{n}}, A_n^{(1)},\ldots,A_n^{(t)}\right) \right\| =\left\| P\left(x_1,\ldots,x_r, a_1,\ldots,a_t\right) \right\|_{\cal A}.\end{equation}
Such a truncation and Gaussian convolution procedure also allows us  to relax the assumption of uniform boundness of moments of entries of Wigner matrices in the almost sure asymptotic freeness result of  Theorem 5.4.5 in \cite{AGZ} (see Proposition \ref{sansmoment} below). 
Note that Proposition \ref{sansmoment} does not need  the independence  of  real and imaginary parts of the entries of the Wigner matrices or  the strong convergence of $\{A_n^{(1)},\ldots, A_n^{(t)}\}$.

In Section \ref{Notations}, we introduce numerous notations and objects  that will be used in the proofs. Section \ref{free} provides required preliminaries on free  operator-valued subordination  properties.  The main section is Section \ref{lemmefonda} where 
we establish Lemma \ref{inclu2}.
     In Sections \ref{sectionnoeigenvalue} and \ref{strategie}, we show how to deduce Theorem \ref{noeigenvalue} and Theorem \ref{thprincipal} respectively from Lemma \ref{inclu2}.
 We end with an Appendix which gathers several useful basic algebraic  lemmas and variance estimates.
\section{Truncation and Gaussian convolution}\label{troncation}
We start with a slightly modified version of Bai-Yin's theorem  (see \cite{BaiYin} for the symmetric case and Theorem 5.1 in \cite{BaiSil06} for the Hermitian case).
\begin{lemme}\label{baiyin}

Let $Z_n =  [Z_{ij}]_{i,j=1}^n$ be a Hermitian $n\times n$ matrix such that 
$( Z_{ij})_{i\leq  j}$,  are independent, centred random variables  such that 
$$\sup_{1\leq i<j}\mathbb{E}(\vert Z_{ij}\vert^3)<+\infty $$  
and there exists a real number $K>0$ and a nonnegative random variable $Y$ with finite fourth moment for which there exists $x_0>0$ and an  integer number $n_0>0$ such that, for any $x >x_0$ and any integer number $n >n_0$, we have \begin{equation}\label{majquatreZ}\frac{1}{n^2} \sum_{1\leq i\leq n,1\leq  j\leq n}\mathbb{P}\left( \vert Z_{ij}\vert >x\right) \leq K \mathbb{P}\left( Y >x\right).\end{equation}
Then, setting $\sigma^*=\{\sup_{1\leq i<j}\mathbb{E}(\vert Z_{ij}\vert^2)\}^{1/2},$
we have that almost surely $$\limsup_{n\rightarrow+\infty} \left\| \frac{Z_n}{\sqrt{n}}\right\| \leq 2\sigma^*.$$
\end{lemme}
\begin{proof} 
Noting that, for any nonnegative increasing sequence $ (u_l)_{l\geq 1}$, we have 
 $$\mathbb{E}\left( \left(\frac{Y}{\sigma^*}\right)^4\right)=4 \int_{0}^{+\infty} t^3 \mathbb{P} \left( \frac{Y}{\sigma^*}>t\right) dt \geq \sum_{l=1}^{+\infty} u_l^3 (u_{l+1}-u_l) \mathbb{P} \left( \frac{Y}{\sigma^*}>u_{l+1}\right), $$
 it readily follows  that for any $\delta>0$, choosing $u_l= \delta 2^{(l-1)/2}$, we have
 $$\sum_{l=1}^\infty 2^{2l} \mathbb{P} \left( \frac{Y}{\sigma^*} >\delta 2^{l/2}\right) <\infty.$$
 In particular, there exists an increasing sequence $(N_k)_{k\geq 1}$ of integer numbers such that 
 for any $k \geq 1$, $$\sum_{l=N_k+1}^\infty 2^{2l} \mathbb{P} \left( \frac{Y}{\sigma^*} >\frac{1}{2^{\frac{k}{8}}} 2^{l/2}\right) \leq \frac{1}{2^k}.$$
 Set for any $l \in ]0,N_{1}]$, $\epsilon_l={1}$ and for any $l \in   ]N_k,N_{k+1}]$, $\epsilon_l=\frac{1}{2^{\frac{k}{8}}}$. Then, 
 $$\sum_{l=N_1 +1}^\infty 2^{2l} \mathbb{P} \left( \frac{Y}{\sigma^*} >\epsilon_l 2^{l/2}\right)= \sum_{k=1}^{+\infty}\sum_{l=N_k+1}^{N_{k+1}} 2^{2l} \mathbb{P} \left( \frac{Y}{\sigma^*} >\frac{1}{2^{\frac{k}{8}} }2^{l/2}\right)\leq \sum_{k=1}^{+\infty}\frac{1}{2^k} <\infty.$$
 Define $\delta_n= \sqrt{2}\epsilon_l$ for $2^{l-1} < n \leq 2^{l}$. 
 Thus, we exhibited a sequence of nonnegative numbers such that $\delta_n \downarrow 0$,
   $\delta_n^2 \sqrt{n} \rightarrow +\infty$ and  $$\sum_{n=1}^\infty 2^{2n} \mathbb{P} \left( \frac{Y}{\sigma^*} >\delta_{2^n} 2^{(n-1)/2}\right) <\infty.$$
Let us consider $X=Z_n/\sigma^*$.
Define for any $i\geq 1, j\geq 1$,  $\check X_{ij}= X_{ij}{\1}_{\vert X_{ij} \vert \leq \sqrt{n} \delta_n} $.
We have for any $k$ large enough, 
 \begin{eqnarray*}
\mathbb{P}\left(X\neq \check X \; i.o \right) &\leq & \sum_{l=k}^\infty \mathbb{P} \left( \bigcup_{2^{l-1}< n \leq  2^{l}} \bigcup_{1\leq i,j\leq n} \left\{\vert X_{ij} \vert > \sqrt{n}\delta_n\right\} \right)\\
 &\leq & \sum_{l=k}^\infty \mathbb{P} \left( \bigcup_{2^{l-1}< n \leq  2^{l}} \bigcup_{1\leq i,j\leq n} \left\{\vert X_{ij} \vert > 2^{\frac{l-1}{2}}\delta_{2^l}\right\} \right)\\
 &\leq & \sum_{l=k}^\infty \sum_{1\leq i,j\leq 2^{l}}  \mathbb{P} \left( \vert X_{ij} \vert > 2^{\frac{l-1}{2}}\delta_{2^l} \right)\\
  &\leq & K\sum_{l=k}^\infty 2^{2l} \mathbb{P} \left( Y > 2^{\frac{l-1}{2}} \delta_{2^l } \sigma^*\right) \rightarrow_{k \rightarrow +\infty} 0.
\end{eqnarray*}

\noindent Define for $i\neq j$, $1\leq i,j\leq n,$ $ \hat X_{ij}= X_{ij}{\1}_{\vert X_{ij} \vert \leq \sqrt{n} \delta_n} $ and for any $i$, $\hat X_{ii}=0$. Then, we have $\left\| \frac{\check X}{\sqrt{n}}- \frac{\hat X}{\sqrt{n}}\right\| \leq \delta_n\rightarrow_{n\rightarrow +\infty} 0.$
\noindent Finally define for $i\neq j$, $1\leq i,j\leq n$, $\tilde X_{ij}= X_{ij}{\1}_{\vert X_{ij} \vert \leq \sqrt{n} \delta_n} -\mathbb{E}\left( X_{ij}{\1}_{\vert X_{ij} \vert \leq \sqrt{n} \delta_n}\right)$ and for any $i$, $\tilde X_{ii}=0$.
We have for large $n$ (denoting by $\Vert\cdot\Vert_2$ the Hilbert-Schmidt norm)\begin{eqnarray*} 
\left\| \frac{\hat X}{\sqrt{n}}- \frac{\tilde X}{\sqrt{n}}\right\|&\leq  &\left\| \frac{\hat X}{\sqrt{n}}- \frac{\tilde X}{\sqrt{n}}\right\|_{2}
\\&=&\left( \frac{1}{n}\sum_{i,j=1;i\neq j }^n \left|\mathbb{E}\left(  X_{ij}{\1}_{\vert X_{ij} \vert > \sqrt{n} \delta_n} \right)\right|^2\right)^{1/2}\\
&\leq&\left( \frac{1}{n^3 \delta_n^4}\sum_{i,j=1}^n \mathbb{E}\left( \left| X_{ij}^2\right| \right)\mathbb{E}\left( \left| X_{ij}\right|^4{\1}_{\vert X_{ij} \vert > \sqrt{n} \delta_n} \right)\right)^{1/2}\\
&\leq &\left( \frac{K}{n \delta_n^4}\mathbb{E}\left( \left| \frac{Y}{\sigma^*} \right|^4{\1}_{\vert \frac{Y}{\sigma^*} \vert > \sqrt{n} \delta_n} \right)\right)^{1/2} \rightarrow_{n\rightarrow +\infty} 0.
\end{eqnarray*}
Thus, we have that almost surely \begin{equation}\label{normetilde}\left\| \frac{ X}{\sqrt{n}}- \frac{\tilde X}{\sqrt{n}}\right\|\rightarrow_{n\rightarrow +\infty} 0.\end{equation}
Note that the entries of $\tilde X$ satisfy 
\begin{itemize}
\item $\tilde X_{ii}=0$;
\item $\tilde X_{ij}$, $i<j$, are independent;
\item $\mathbb{E}\left( \tilde X_{ij}\right) =0, \mathbb{E}\left( \vert \tilde X_{ij}\vert^2\right) \leq 1, \mbox{\; for \;} i\neq j$;
\item $\vert \tilde X_{ij}\vert\leq 2\delta_n \sqrt{n}, \mbox{\; for \;} i\neq j$;
\item $ \mathbb{E}\left( \vert \tilde X_{ij}\vert^l\right) \leq b \left(2\delta_n \sqrt{n}\right)^{l-3}, \mbox{for some constant } b>0 \mbox{\;and all \;} i\neq j, \; l\geq~3$.
\end{itemize}
Then, sticking to the end of  the proof of Theorem 5.1  pages 87-93    in \cite{BaiSil06}, one can prove that for any even integer $k$, one has 
$$\mathbb{E}\left(\Tr \left(\frac{\tilde X}{\sqrt{n}}\right)^k \right)\leq n^2\left[2+ (10(2\delta_n)^{1/3} k/\log n)^3\right]^k.$$
Chebychev's inequality yields that for any $\eta >2$, 
\begin{equation}\label{chebychev}\mathbb{P}\left(\left\|\frac{\tilde X}{\sqrt{n}}\right\| >\eta\right)\leq \frac{1}{\eta^k}\mathbb{E}\left(\Tr \left(\frac{\tilde X}{\sqrt{n}}\right)^k \right)
\leq 
n^2\left[\frac{2}{\eta}+ \frac{(10 (2\delta_n)^{1/3} k/\log n)^3}{\eta}\right]^k.\end{equation}
Selecting the sequence of even integers $k_n=2 \left\lfloor \frac{\log n }{\delta_n^{1/6}}\right\rfloor$ with the properties $k_n/\log n \rightarrow +\infty$ and $k_n \delta_n^{1/3}/\log n \rightarrow 0$, we obtain that the right hand side of \eqref{chebychev} is summable.
Using Borel-Cantelli's Lemma, we easily deduce  that almost surely $$\limsup_{n\rightarrow +\infty} \left\|\frac{\tilde X}{\sqrt{n}}\right\| \leq 2$$ and then, using \eqref{normetilde}, that almost surely $$\limsup_{n\rightarrow +\infty} \left\|\frac{ X}{\sqrt{n}}\right\| \leq 2.$$
\end{proof}
According to Lemma \ref{baiyin}, for any $v=1,\ldots,r$, almost surely, 
$\sup_n \left\| \frac{X_n^{(v)}}{\sqrt{n}} \right\| < +\infty.$ Therefore by a simple approximation argument using polynomials with coefficients in $\mathbb{Q} +i \mathbb{Q}$,  to establish Theorem \ref{thprincipal}, it is sufficient to prove that for any polynomial, almost surely \eqref{af} and \eqref{saf} hold. Now, we are going to show that the proof of Theorem \ref{thprincipal}  can be reduced to the proof of \eqref{saf} in the  case where the $X_{ij}^{(v)}$'s satisfy $(H)$.\\
Let $X =  [X_{jk}]_{j,k=1}^n$ be a Hermitian $n\times n$ matrix such  that the  random variables $X_{ii}$,
$\sqrt{2} \Re (X_{ij})$, $\sqrt{2} \Im (X_{ij}), {i<j},$ are independent, centred  with variance 1 such that 
$$\theta^*=\sup_{1\leq i<j}\mathbb{E}(\vert X_{ij}\vert^3)<+\infty, $$
and there exists a $K$ and a random variable $Z$ with finite fourth moment for which there exists $x_0>0$ and  an  integer number $n_0>0$ such that, for
any $x>x_0$ and  any integer number $n>n_0$, we have \begin{equation}\label{majquatre}\frac{1}{n^2} \sum_{i\leq n, j\leq n}\mathbb{P}\left( \vert  X_{ij}\vert >x\right) \leq K\mathbb{P}\left(\vert Z \vert>x\right).\end{equation}

\noindent Define for any  $C>0$, for any $1\leq i , j \leq n$,
\begin{eqnarray}Y_{ij}^C &=&\Re  X_{ij}\1_{\vert  \Re X_{ij} \vert \leq C} - \mathbb{E}\left( \Re  X_{ij}\1_{\vert \Re  X_{ij} \vert \leq C} \right) \nonumber \\&&+ \sqrt{-1} \left\{
\Im  X_{ij}\1_{\vert  \Im  X_{ij} \vert \leq C} - \mathbb{E}\left( \Im  X_{ij}\1_{\vert \Im  X_{ij} \vert \leq C} \right) \right\}.\label{ycdef}\end{eqnarray}
We have 
\begin{eqnarray*}\mathbb{E} \left( \vert   X_{ij}-Y_{ij}^C\vert^2 \right)&=&\mathbb{E} \left( \vert \Re   X_{ij}\vert^2 \1_{\vert \Re  X_{ij} \vert > C} \right)
+ \mathbb{E} \left( \vert \Im  X_{ij}\vert^2 \1_{\vert \Im  X_{ij} \vert > C} \right) \\&&
-\left\{ \mathbb{E} \left( \Re  X_{ij} \1_{\vert \Re  X_{ij} \vert > C} \right)\right\}^2-\left\{ \mathbb{E} \left( \Im   X_{ij} \1_{\vert \Im  X_{ij} \vert > C} \right)\right\}^2\\& \leq & \frac{\mathbb{E} \left(\vert \Re  X_{ij} \vert^3\right)+\mathbb{E} \left(\vert \Im  X_{ij} \vert^3\right)}{C}
\end{eqnarray*}
so that $$\sup_{i\geq 1,j\geq 1}\mathbb{E} \left( \vert  X_{ij}-Y_{ij}^C\vert^2 \right) \leq \frac{2\theta^*}{C}.$$
According to Lemma \ref{baiyin}, we have almost surely \begin{equation}\label{centragebis}\limsup_{n\rightarrow+\infty} \left\| \frac{X}{\sqrt{n}}-\frac{Y^C}{\sqrt{n}}\right\| \leq 2\frac{\sqrt{2\theta^*}}{\sqrt{C}}.\end{equation}
Note that 
\begin{eqnarray*}1 - 2 \mathbb{E} \left( \vert \Re  Y_{ij}^C\vert^2 \right) &=& 1-2 \mathbb{E}   \left\{\left(\Re   X_{ij} \1_{\vert \Re X_{ij} \vert \leq  C} -  \mathbb{E} \left(  \Re    X_{ij} \1_{\vert \Re X_{ij} \vert \leq  C}\right)\right)^2\right\}
\\&=& 2\left[ \frac{1}{2} -\mathbb{E} \left( \vert  \Re  X_{ij}\vert^2 \1_{\vert \Re X_{ij} \vert \leq  C} \right)\right] 
+2\left\{ \mathbb{E} \left( \Re X_{ij} \1_{\vert \Re  X_{ij} \vert \leq  C} \right)\right\}^2\\&=& 
2 \mathbb{E} \left( \vert \Re   X_{ij}\vert^2 \1_{\vert\Re  X_{ij} \vert > C} \right)+ 2\left\{\mathbb{E} \left(  \Re  X_{ij} \1_{\vert \Re  X_{ij} \vert > C}\right) \right\}^2. \end{eqnarray*}
so that $$\sup_{i\geq 1, j\geq 1}\vert 1 - 2\mathbb{E} \left( \vert \Re  Y_{ij}^C\vert^2 \right) \vert \leq \frac{4\theta^*}{C}.$$
Similarly $$\sup_{i\geq 1,j\geq 1}\vert 1 - 2\mathbb{E} \left( \vert \Im  Y_{ij}^C\vert^2 \right) \vert \leq \frac{4\theta^*}{C}.$$
Let us assume that $C>8\theta^*.$ Then, we have 
$$\mathbb{E} \left( \vert \Re  Y_{ij}^C\vert^2 \right)> \frac{1}{4} \; \mbox{and}\; \mathbb{E} \left( \vert \Im  Y_{ij}^C\vert^2 \right)> \frac{1}{4}.$$
Now define \begin{equation}\label{defxc}{X}_{ij}^C =\frac{\Re  Y_{ij}^C}{\sqrt{2\mathbb{E} \left( \vert \Re Y_{ij}^C\vert^2 \right)}} +\sqrt{-1}  \frac{\Im  Y_{ij}^C}{\sqrt{2\mathbb{E} \left( \vert \Im  Y_{ij}^C\vert^2 \right)}}.\end{equation}
Note that \begin{eqnarray*} {X}_{ij}^C-Y_{ij}^C&=& \Re  X_{ij}^C \left( 1-\sqrt{2}  \mathbb{E} \left( \vert \Re  Y_{ij}^C\vert^2 \right)^{1/2}\right)
+\sqrt{-1} \Im  X_{ij}^C \left( 1-\sqrt{2}  \mathbb{E} \left( \vert \Im  Y_{ij}^C\vert^2 \right)^{1/2}\right)
\\&=&\Re  X_{ij}^C\frac{1 - 2\mathbb{E} \left( \vert \Re  Y_{ij}^C\vert^2 \right) }{1 + \sqrt{2}\mathbb{E} \left( \vert \Re  Y_{ij}^C\vert^2 \right)^{1/2}}+
\sqrt{-1} \Im  X_{ij}^C\frac{1 - 2\mathbb{E} \left( \vert \Im  Y_{ij}^C\vert^2 \right) }{1 + \sqrt{2}\mathbb{E} \left( \vert \Im  Y_{ij}^C\vert^2 \right)^{1/2}}.\end{eqnarray*} 
\noindent Thus, according to Lemma \ref{baiyin}, we have almost surely \begin{equation}\label{normal}\limsup_{n\rightarrow+\infty} \left\| \frac{X^C}{\sqrt{n}}-\frac{Y^C}{\sqrt{n}}\right\| \leq \frac{{8\theta^*}}{{C}}.\end{equation}
Thus, by \eqref{centragebis} and \eqref{normal}, we obtain that almost surely \begin{equation}\label{xc}\limsup_{n\rightarrow+\infty} \left\| \frac{X}{\sqrt{n}}-\frac{X^C}{\sqrt{n}}\right\| \leq 2\frac{\sqrt{2\theta^*}}{\sqrt{C}}+ \frac{{8\theta^*}}{{C}}.\end{equation}
Let $[{\cal G}_{ij}]_{i\geq 1, j\geq 1}$ be an infinite array which is  independent of the $X_{ij}'s$ and such that $\sqrt{2} \Re {\cal G}_{ij}$, $ \sqrt{2} \Im {\cal G}_{ij}$, $i<j$, 
and  ${\cal G}_{ii}$ are independent centred standard real Gaussian variables and for any $i\geq 1, j\geq 1$,  ${\cal G}_{ji}= \overline{{\cal G}_{ij}}$ .
Thus,  ${\cal G} = [{\cal G}_{ij}]_{ i,j =1}^n $ is  a random G.U.E matrix being independent of $X$. Define 
for any $\delta>0$,
\begin{equation}\label{xcdelta}X^{C,\delta}= \frac{ X^C +\delta {\cal G}}{\sqrt{1+\delta^2}}.\end{equation}
Note that  the random variables $\sqrt{2}\Re  (X^{(v)})^{C,\delta}_{ij}$, $\sqrt{2}\Im   (X^{(v)})^{C,\delta}_{ij}$, $ i<j$, $(X^{(v)})_{ii}^{C,\delta}$ are independent, centred with variance 1.\\
From Bai-Yin's theorem  (Theorem 5.8 in \cite{BaiSil06}),
we have that almost surely, \begin{equation}\label{by}\left\| \frac{{\cal G}}{\sqrt{n}} \right\|=2+o(1)\end{equation} and by  Lemma  \ref{baiyin},
 for any $C > 8\theta^*$, we have almost surely
\begin{equation}\label{bybis} \limsup_{n\rightarrow+\infty}\left\| \frac{{X^C}}{\sqrt{n}} \right\|\leq 2.\end{equation}
Now, \eqref{xc}, \eqref{by}  and \eqref{bybis} yield that for any $C > 8 \theta^*$, any $\delta >0$, almost surely   $ \limsup_{n\rightarrow +\infty}\left\|\frac{X -{X^{C,\delta}}}{\sqrt{n}} \right\| \leq u_C +v_\delta$ where $u_C$ and $v_\delta$ are deterministic positive functions tending to zero when respectively $C$ goes to infinity and $\delta$ goes to zero. 
Hence, it is easy to see that for any $0<\epsilon<1$, there exists $C_\epsilon$ and $\delta_\epsilon$ such that almost surely for all large $n$, 
$$\left\|\frac{X -{X^{C_\epsilon,\delta_\epsilon}} }{\sqrt{n}}\right\|\leq \epsilon.$$
We can deduce that for any  polynomial $P$ in $r+t$ noncommutative variables, there exists some constant $L>0$ such that the following holds: for any $0<\epsilon<1$, there exists $C_\epsilon$ and $\delta_\epsilon$ such that almost surely for all large $n$, \\

\noindent $\left\| P\left(\frac{X_n^{(1)}}{\sqrt{n}},\ldots,\frac{X_n^{(r)}}{\sqrt{n}}, A_n^{(1)},\ldots,A_n^{(t)}\right)\right. $ \begin{equation}\label{approxnorme} \left.- P\left(\frac{(X_n^{(1)})^{C_\epsilon,\delta_\epsilon}}{\sqrt{n}},\ldots,\frac{(X_n^{(r)})^{C_\epsilon,\delta_\epsilon}}{\sqrt{n}}, A_n^{(1)},\ldots,A_n^{(t)}\right)\right\|\leq L \epsilon.\end{equation} 
Then, it is clear that it is sufficient to establish Theorem \ref{thprincipal}  and Theorem \ref{noeigenvalue} for the $(X_n^{(v)})^{C_\epsilon,\delta_\epsilon}$'s. \\Moreover, we obviously have that for any $\epsilon$ and  for any $p\in \mathbb{N}$,  \begin{equation}\label{moments}\max_{v=1,\ldots,r} \sup_{i\geq 1, j\geq 1} \mathbb{E}\left(\vert  (X_{ij}^{(v)})^{C_\epsilon,\delta_\epsilon}\vert^p\right) <+\infty.\end{equation}
 Then, \eqref{af}  is a consequence of Theorem 5.4.5 in \cite{AGZ}.\\

 Moreover, note that, by definition,   the distributions of the random variables
$\sqrt{2}\Re  (X^{(v)})_{ij}^{C_\epsilon,\delta_\epsilon}, \sqrt{2}\Im  (X^{(v)})_{ij}^{C_\epsilon,\delta_\epsilon},  i<j,  ( X^{(v)})_{ii}^{C_\epsilon,\delta_\epsilon}$,  $v=1\ldots,r$, are all    a convolution of a centred   Gaussian distribution with variance $\delta_\epsilon^2/(1+\delta_\epsilon^2)$,  with  some  law with bounded support
in the ball of radius $2C_\epsilon$; thus, according to Lemma \ref{zitt},    
they satisfy a  Poincar\'e inequality with some common  constant $ C_{PI}(C_\epsilon,\delta_\epsilon)$.\\

{\it Hence in the following we will focus on establishing Theorem \ref{noeigenvalue} and \eqref{safbis} when  the $X_{ij}^{(v)}$'s satisfy  (H).}
\\

Note that a similar truncation and Gaussian convolution procedure also allows to establish the following asymptotic freeness result.
\begin{proposition}\label{sansmoment}
Assume that $(A_n^{(1)},\ldots,A_n^{(t)})$ is a $t-$tuple of  $n\times n$   Hermitian deterministic matrices   such that
 $\max_{u\in\{1,\ldots,t\}} \sup_n \Vert A_n^{(u)} \Vert< \infty$ and  
 $(A_n^{(1)},\ldots,A_n^{(t)})$ converges  in distribution in $(M_n(\C), \frac{1}{n} \Tr_n )$ towards some t-tuple of noncommutative random variables $(a_1,\ldots a_t)$ in some noncommutative probability space $({\cal A},\tau)$.
  Let  $X_n^{(v)} =  [X^{(v) }_{ij}]_{i,j=1}^n$, $v=1,\ldots,r$, be  r  independent $n\times n$ random  Hermitian matrices such that, for each $v$,
$ [X^{(v)  }_{ij}]_{i\geq1,j\geq 1}$ is an infinite array of  random variables such that $X^{(v)}_{ii}$,
$X^{(v)}_{ij}, i<j$,  are independent, centred and $\mathbb{E}(\vert X_{ij}\vert^2)=1$. Let  $(x_1,\ldots,x_r)$ be  a semi-circular system in  $({\cal A},  \tau)$ which is free from $(a_1,\ldots a_t)$.
Then, $\left( X^{(1)},\ldots, X^{(r)}, A_n^{(1)},\ldots,A_n^{(t)}\right)$ converges in distribution in $(M_n(\C), \frac{1}{n} \Tr_n )$ towards
$(x_1,\ldots,x_r, a_1,\ldots a_t)$ \end{proposition}
\begin{proof}
In the arguments above the present proposition, one  can consider $$\tilde Y_{ij}^C =X_{ij}\1_{\vert X_{ij} \vert \leq C} - \mathbb{E}\left( X_{ij}\1_{\vert X_{ij} \vert \leq C} \right),\mbox{~~and~~}\tilde{X}_{ij}^C =\frac{\tilde Y_{ij}^C}{\sqrt{\mathbb{E} \left( \vert \tilde Y_{ij}^C\vert^2 \right)}},$$
instead of $Y^C$ in \eqref{ycdef} and $X^C$ in \eqref{defxc} and similarly  obtain that almost surely \begin{equation}\limsup_{n\rightarrow+\infty} \left\| \frac{X}{\sqrt{n}}-\frac{\tilde X^C}{\sqrt{n}}\right\|\leq 2\frac{\sqrt{\theta^*}}{\sqrt{C}}+ \frac{{4\theta^*}}{{C}}.\end{equation}
Then,  one can similarly prove that \eqref{approxnorme} holds with 
$$\tilde X^{C,\delta}=\frac{ \tilde X^C +\delta {\cal G}}{\sqrt{1+\delta^2}}$$ instead of $X^{C,\delta}$ in \eqref{xcdelta}, so that the result  follows from Theorem 5.4.5 in \cite{AGZ}.
\end{proof}
\section{Notations }\label{Notations} 
The main part of the following  consists in proving Lemma \ref{inclu2}. For that purpose, we fix some integer number $m$ and Hermitian $ m\times m$ matrices $\{\alpha_v\}_{v=1,\ldots,r}, \{\beta_u\}_{u=1,\ldots,t}, \gamma$.
To begin with, we introduce
some notations on the set of matrices.
\begin{itemize}
\item $M_p(\C)$ is the set of $p\times p$ matrices  with complex entries,  $M_p(\C)_{sa}$ the subset of self-adjoint elements of $M_p(\C)$ and $I_p$ the identity matrix. In the following, we shall consider two sets of matrices with $p=m$ ($m$ fixed) and $p=n$ with $n\vers \infty$.
\item $\Tr_p$ denotes the trace and $\tr_p = \frac{1}{p} \Tr_p$ the normalized trace on $M_p(\C)$.
\item $|| . ||$ denotes the spectral norm on $M_p(\C)$ and $||M||_2 = (\Tr_p (M^*M))^{1/2}$ the Hilbert-Schmidt norm.
\item ${\rm id}_p$ denotes the identity operator from $M_p(\C)$ to $M_p(\C)$.
 \item  Let $(E_{ij})_{i,j =1}^n$ be the canonical basis of $M_n(\C)$ and define a basis of the real vector space of the self-adjoint matrices $M_n(\C)_{sa}$ by:
 \begin{eqnarray*}
 e_{jj} &= &E_{jj}, 1\leq j \leq n \\
 e_{jk}& = &\frac{1}{\sqrt{2}}(E_{jk} + E_{kj}), 1 \leq j < k \leq n \\
 f_{jk} &=& \frac{\sqrt{-1}}{\sqrt{2}}(E_{jk} - E_{kj}) , 1 \leq j < k \leq n
 \end{eqnarray*}
 \item[-]  $(\hat{E}_{ij})_{i,j =1}^m$ is the canonical basis of $M_m(\C)$.
 \item[-] For a matrix $M$ in $M_m(\C) \otimes M_n(\C)$, we set
 $$M_{ij} : = ({\rm id}_m \otimes \Tr_n)(M (1_m \otimes E_{ji})) \in M_m(\C), \ 1\leq i,j \leq n.$$
\item[-] For any matrix $M$ (or, more generally, any operator $M$),
we define its real part to be $\Re  M=(M+M^*)/2$ and its imaginary part to be $\Im  M=(M-M^*)/2i$.
We write $M>0$ if the matrix $M$ is positive definite and $M\ge0$ if it is nonnegative definite.
In general $M>P$ means that $M-P$ is positive definite.
  \end{itemize}
\noindent We now define   the random variables of interest. 
Let $({\cal A},  \tau)$ be a ${\cal C}^*$-probability space with unit $1_{\cal A}$,
 equipped with a faithful tracial state and 
 $x=(x_1,\ldots,x_r)$ be  a semi-circular system in  $({\cal A},  \tau)$.
 Let $a_n=(a_n^{(1)},\ldots,a_n^{(t)})$ be a t-tuple of noncommutative self-adjoint  random variables which is free from $x$  in $({\cal A},\tau$) and  such that the distribution of  $a_n$ in $({\cal A},\tau)$  coincides with the distribution of $(A_n^{(1)},\ldots, A_n^{(t)})$ in $({ M}_n(\mathbb{C}),\tr_n)$.
\begin{itemize}

 \item[-] 
 We define the random variable $S_n$ with values in $M_m(\C) \otimes M_n(\C)$ by:
 \begin{equation} \label{defSn}
 S_n = \gamma \otimes I_n + \sum_{v=1}^r \alpha_v  \otimes \frac{X^{(v)}_n}{\sqrt{n}}+ \sum_{u=1}^t \beta_u \otimes A_n^{(u)}
 \end{equation}
and $s_n\in M_m(\C) \otimes {\cal A}$ by 
$$s_n=\gamma \otimes 1_{\cal A} + \sum_{v=1}^r\alpha_v \otimes x_v+ \sum_{u=1}^t\beta_u \otimes a_n^{(u)}.$$
 \item[-] For any matrix $\lambda$ in $ M_m(\C) $ such that $ \Im(\lambda)$ is positive definite,
  we define the $M_m(\C)\otimes M_n(\C)$-valued respectively $M_m(\C) \otimes {\cal A}$-valued  random variables:
 \begin{equation}\label{rn}
	R_n(\lambda)=(\lambda \otimes I_n - S_n)^{-1},
	\end{equation}
	$$r_{n}(\lambda)= \left(\lambda \otimes 1_{\cal A} -s_n \right)^{-1},$$
	and 
	the $M_m(\C)$-valued random variables:
 \begin{equation} \label{defSt}
 H_n(\lambda) = ({\rm id}_m \otimes \tr_n) [ (\lambda \otimes I_n - S_n)^{-1}],
 \end{equation}
  \begin{equation} \label{espSt}
G_n(\lambda) =\mathbb{ E}[ H_n(\lambda) ]
\end{equation}
and
\begin{equation}\label{defGntilde}\tilde G_{n}(\lambda)={\rm id}_m \otimes \tau \left(r_n (\lambda ) \right).\end{equation}
Since $\sum_{v=1}^r\alpha_v \otimes x_v$ is an $ M_m(\mathbb C)$-valued semicircular of variance $\eta\colon b\mapsto\sum_{v=1}^r\alpha_vb\alpha_v$ which is free over $ M_m(\mathbb C)$ from $\gamma\otimes 1_{\cal A}+\sum_{u=1}^t\beta_u \otimes a_n^{(u)}$, we know from \cite{ABFN} (see proof of Theorem 8.3) that $\tilde G_{n}$ satisfies (see Section \ref{free} Theorem \ref{resusub} for $p=1$) \begin{equation}\label{subor}\tilde G_{n}(\lambda)= G_{\sum_{u=1}^t\beta_s \otimes a_n^{(u)}}(\omega_n(\lambda))
\end{equation}
where \begin{equation}\label{omegan}
\omega_n(\lambda)=\lambda-\gamma -\sum_{v=1}^r\alpha_v \tilde  G_n(\lambda)\alpha_v
\end{equation} and $$G_{\sum_{u=1}^t\beta_u \otimes a_n^{(u)}}(\lambda)={\rm id}_m \otimes \tau \left(\lambda\otimes 1_{\cal A} -\sum_{u=1}^t\beta_u \otimes a_n^{(u)} \right)^{-1}.$$
\noindent For $z\in \C \setminus  \R$, we also define
 \begin{equation}\label{defpetitg}g_n(z) = \tr_m(G_n(zI_m))\end{equation}and \begin{equation}\label{defpetitgtilde}\tilde g_n(z) = \tr_m(\tilde G_n(z I_m)).\end{equation}


 
 
  \end{itemize}
 In the sequel, we will say that a random  term in some $M_p(\C)$, depending on $n$, $\lambda \in M_m(\mathbb{C})  $ such that $\Im  \lambda$ is positive definite,  the $\{\alpha_v\}_{v=1,\ldots,r}$, $\{\beta_u\}_{u=1,\ldots,t}$ and  $\gamma$, is  $O\left(\frac{1}{n^k}\right)$ if its  operator  norm  is smaller than
 $\frac{ Q\left(\Vert (\Im \lambda)^{-1} \Vert\right)}{n^k}$ for some deterministic  polynomial $Q$
 whose coefficients are nonnegative real numbers  and can depend on $m$,  $\{\alpha_v\}_{v=1,\ldots,r}$, $\{\beta_u\}_{u=1,\ldots,t}$, $\gamma$.\\
 For a family of random terms $I_{pq}$, $(p,q) \in \{1,\ldots,n\}^2$, we will set $I_{pq}=O_{p,q}^{(u)} \left(\frac{1}{n^k}\right)$ if for each $(p,q)$, $I_{pq}=O \left(\frac{1}{n^k}\right)$ and moreover one can find a bound of 
the norm of each $I_{p,q}$  as above involving a common polynomial $Q$.\\

Throughout the paper, $K$, $C$ denote some positive constants and  $Q$  denotes some deterministic  polynomial in one variable
 whose coefficients are nonnegative real numbers; they can depend on $m$,   $\{\alpha_v\}_{v=1,\ldots,r}$, $\{\beta_u\}_{u=1,\ldots,t}$ and $\gamma$ and they  may  vary from line to line.

\section{Operator-valued subordination}\label{free}


In this section we introduce one of the main tools used in describing joint distributions of
random variables that do not necessarily commute. It was a crucial insight of Voiculescu
\cite{V2000, FreeMarkov, V1} that J. L. Taylor's theory of free noncommutative functions \cite{taylor} 
provides the appropriate analogue of the classical Stieltjes transform for encoding
operator-valued distributions \cite{V1995}, and hence joint distributions of $q$-tuples of
noncommuting random variables. We shall only present below the case of relevance to our
present work, and refer the reader to \cite{V1995,V2000,V1} for the general case and for the
proofs of the main results.

Given a tracial $\mathcal C^*$-probability space $(\mathcal A,\tau)$ and a trace and order preserving
unital $\mathcal C^*$ inclusion $M_m(\mathbb C)\subseteq\mathcal A$ (i.e. such that $I_m\in 
M_m(\mathbb C)$ identifies with the unit of $\mathcal A$ and ${\rm tr}_m(b)=\tau(b)$ for all $b\in 
M_m(\mathbb C)$), there exists a conditional expectation $E\colon\mathcal A\to M_m(\mathbb C)$, i.e. 
a linear map sending the unit to itself and such that $E(b_1yb_2)=b_1E(y)b_2$ for all $b_1,b_2\in M_m(\mathbb C),y
\in\mathcal A$ - see \cite[Section II.6.10.13]{Bruce}. We will only be concerned with the trivial case
of the canonical inclusion $M_m(\mathbb C)\subseteq M_m(\mathbb C)\otimes\mathcal A$ 
given by $b\mapsto b\otimes 1_{\cal A}$, when the conditional expectation is the partial trace:
$E(b\otimes y)=({\rm id}_m\otimes\tau)(b\otimes y)=\tau(y)b$.
The $M_m(\mathbb C)$-valued distribution of an element $y\in\mathcal A$ with respect to $E$ is by 
definition the family of multilinear maps 
$\mu_y=
\{\Psi_q\colon\underbrace{M_m(\mathbb C)\times\cdots\times M_m(\mathbb C)}_{q-1\text{ times}}
\to M_m(\mathbb C)\colon \Psi_q(b_1,\dots,b_{q-1})=E[yb_1y\cdots b_{q-1}y],q\in\mathbb N\}$. 
By convention, $\Psi_0=1\in M_m(\mathbb C)$, $\Psi_1=E[y]$.

For a given $y=y^*\in\mathcal A$ with distribution $\mu_y$, define its noncommutative 
Stieltjes transform to be the {\em countable family} of maps
$G_{\mu_y,p}(b)=(E\otimes{\rm id}_{p})\left[(b-y\otimes I_p)^{-1}\right], p\in \mathbb{N}\setminus\{0\}.$
Thus, $G_{\mu_y,1}(b)=E\left[(b-y)^{-1}\right],b\in M_m(\mathbb C)$,
$G_{\mu_y,2}(b)=(E\otimes{\rm id}_{2})\left[\left(\begin{bmatrix}
b_{11} & b_{12}\\
b_{21} & b_{22}
\end{bmatrix}-\begin{bmatrix}
y & 0\\
0 & y
\end{bmatrix}\right)^{-1}\right],$ $b_{11}, b_{12},
b_{21}, b_{22}\in M_m(\mathbb C)$ etc. These maps are clearly analytic on the open sets $\{b\in 
M_m(\mathbb C)\otimes M_p(\mathbb C)\colon b-y\otimes I_p\text{ invertible in }\mathcal A\otimes M_p(\mathbb C)\}$. Two 
such sets will be important in this paper: the noncommutative upper half-plane $H^+_p(M_m(\mathbb 
C))=\{b\in M_m(\mathbb C)\otimes M_p(\mathbb C)\colon \Im  b:=(b-b^*)/2i>0\}$, $p\in \mathbb{N}\setminus\{0\}$, 
and the ``ball around infinity'' $\{b\in M_m(\mathbb C)\otimes M_p(\mathbb C) \colon b\text{ invertible, 
}\|b^{-1}\|<\|y\|^{-1}\}$ (actually only for $p=1$). The maps $b\mapsto G_{\mu_y,p}(b^{-1})$ have thus an analytic extension 
around zero, which maps zero to zero and has the identity as first (Frechet) derivative. While 
$G_{\mu_y,p}$ does not map these ``balls around infinity'' into themselves, it does map 
$H^+_p(M_m(\mathbb C))$ into $-H^+_p(M_m(\mathbb C))$, and moreover $G_{\mu_y,p}(b)^{-1}$ maps 
$H^+_p(M_m(\mathbb C))$ into itself (see \cite[Section 3.6]{V2000}). (In addition, the family of maps 
$\{G_{\mu_y,p}\}_{p\in \mathbb{N}\setminus\{0\}}$ satisfy certain compatibility conditions that make them into free 
noncommutative maps - see \cite{KVV}. It is known \cite{W} that there is a bijection between such families 
of maps $G$ that send for any $p \in \mathbb{N}\setminus\{0\}$, $H^+_p(M_m(\mathbb C))$ into $-H^+_p(M_m(\mathbb C))$ and have the 
above-described behavior on ``balls around infinity'' and $M_m(\mathbb C)$-valued distributions of 
self-adjoint elements; however, since we will not make use of this correspondence, we chose to only 
mention it in order to illustrate the parallel to the case of classical Stieltjes transforms, and direct the 
interested reader to \cite{W} for details.)

As in scalar-valued free probability, one defines \cite{V1995} {\em freeness with amalgamation}
over $M_m(\mathbb C)$ via an algebraic relation similar to \eqref{freeness}, but involving $E$
instead of $\tau$ and noncommutative polynomials with coefficients in $M_m(\mathbb C)$.
Since it is not important for us here, we refer the interested reader to \cite{V1995} for 
more details. The essential result of Voiculescu that we will need in this paper is the following
analytic subordination result:
\begin{theoreme}\label{resusub}
With the above notations, assume that $y_1=y_1^*,y_2=y_2^*\in\mathcal A$ are free with amalgamation
over $M_m(\mathbb C)$. For any $p\in\mathbb N$ there exist analytic maps $\omega_{1,p},
\omega_{2,p}\colon H^+_p(M_m(\mathbb C))\to H^+_p(M_m(\mathbb C))$ such that:
\begin{enumerate}
\item For all $b\in H^+_p(M_m(\mathbb C)),$ $\Im \omega_{j,n}(b)\ge\Im  b$, $j=1,2$;
\item For all $b\in H^+_p(M_m(\mathbb C))$
$$G_{\mu_{y_1+y_2},p}(b)=G_{\mu_{y_1},p}(\omega_{1,p}(b))=G_{\mu_{y_2},p}(\omega_{2,p}(b))
=\left[\omega_{1,p}(b)+\omega_{2,p}(b)-b\right]^{-1}$$ 
\item $\omega_{1,p},
\omega_{2,p}$ are noncommutative maps in the sense of $\cite{taylor}$ $($see $\cite{KVV})$.
\end{enumerate}
\end{theoreme}
The result as phrased here is a combination of parts of \cite[Theorem 3.8]{V2000} and
\cite[Theorem 2.7]{BMS}. We shall use this theorem in the particular case when $y_2=s$ is a centred 
{\em operator-valued semicircular} random variable (and actually only for $p=1$). As for the scalar-valued semicircular
centred random variables, it is  uniquely determined by its variance $\eta\colon b\mapsto E(sbs)$,
which is a completely positive self-map of $M_m(\mathbb C)$. A characterization in terms of moments
and cumulants via $\eta$ is provided by Speicher in \cite{SMem}. Given the context of our paper,
we find it more useful to provide a characterization in terms of the noncommutative 
Stieltjes transform \cite{HRS}: the functions $G_{\mu_s,p}$ are  the unique solutions mapping
$H^+_p(M_m(\mathbb C))$ into $-H^+_p(M_m(\mathbb C))$ of the functional equations
$$
G_p(b)^{-1}=b-(\eta\otimes{\rm id}_p)(G_p(b)),\quad b\in H^+_p(M_m(\mathbb C)),
p\in\mathbb N.
$$
Starting from this equation, it can be shown (see \cite{ABFN}) that the subordination function associated 
to a semicircular operator-valued random variable is particularly nice: if $y_1$ and $s$ are free with
amalgamation over $M_m(\mathbb C)$, then 
\begin{equation}
\omega_{1,p}(b)=b-(\eta\otimes{\rm id}_p)(G_{\mu_{y_1},p}(\omega_{1,p}(b)))=b-
(\eta\otimes{\rm id}_p)(G_{\mu_{y_1+s},p}(b)),
\end{equation}
for $b\in H^+_p(M_m(\mathbb C)),p\in\mathbb N,$ or, equivalently,
\begin{equation}
G_{\mu_{y_1},p}(b-(\eta\otimes{\rm id}_p)(G_{\mu_{y_1+s},p}(\omega_{1,p}(b)))=G_{\mu_{y_1+s},p}(b),
\quad b\in H^+_p(M_m(\mathbb C)),p\in\mathbb N.
\end{equation}
This indicates that the $\omega_{1,p}$'s are injective maps on $H^+_p(M_m(\mathbb C)),p\in\mathbb N.
$ Their left inverses are defined as
\begin{equation}\label{tauto}
{\Lambda}_{1,p}(w)=w+(\eta\otimes{\rm id}_p)(G_{\mu_{y_1},p}(w)),
\quad w\in H^+_p(M_m(\mathbb C)),p\in\mathbb N.
\end{equation}

Let us  explain how all this relates to the joint distributions of free
random variables. It turns out (see, for example, \cite{NSS}, but it can be easily verified directly)
that if $\{x_1=x_1^*,\dots,x_r=x_r^*\}, \{y_1=y_1^*,\dots,y_t=y_t^*\}\subset\mathcal A$ are free over 
$\mathbb C$ and for $v=1,\ldots,r$, $\alpha_v=\alpha_v^*$, for $u=1,\ldots,t$, $\beta_u=\beta_u^*\in M_m(\mathbb C)$,
then $\{\alpha_1\otimes x_1,\dots,\alpha_r\otimes x_r\}$ and $\{\beta_1\otimes y_1,\dots,\beta_t\otimes y_t\}$ are free 
with amalgamation over $M_m(\mathbb C)$. Thus, they can be treated with the tools described above.
Moreover, if $x_1,\dots,x_r$ are free $\mathbb C$-valued semicircular centred random variables 
of variance one and $\alpha_1,\dots,\alpha_r$ are self-adjoint $m\times m$ complex matrices, then
$\alpha_1\otimes x_1+\cdots+\alpha_r\otimes x_r$ is a centred $M_m(\mathbb C)$-valued semicircular of 
variance $b\mapsto\sum_{j=1}^r \alpha_jb\alpha_j$. These simple facts together with a linearization trick
(see Section \ref{linearisation} and  Step 1 of Section \ref{strategie}) will allow us in principle to treat, from the point of view of
the Stieltjes transform, an $r$-tuple of Wigner matrices and deterministic  matrices as we would treat a 
single Wigner matrix together with a single deterministic  matrix. \\
~~

\noindent Let us conclude this section with the following  invertibility property of   matricial subordination maps related to semi-circular system that will fundamental in our approach.

\begin{lemme}\label{inversion}
Using the notations of Section \ref{Notations}, define  for any $\rho$ in $M_m(\mathbb{C})$ such that $\Im  \rho>0$,
\begin{equation}\Lambda_n(\rho)= \gamma +\rho + \sum_{v=1}^r \alpha_v G_{\sum_{u=1}^t \beta_u \otimes a_n^{(u)}}(\rho) \alpha_v.\end{equation}
With $\omega_n$  defined in \eqref{omegan},
when $\Im \rho>0$ and $\Im \Lambda_n(\rho)>0$ we have 
$$\omega_n(\Lambda_n(\rho))=\rho.$$
\end{lemme}
\begin{proof}
The equality $\Lambda_n(\omega_n(\rho))=\rho$ holds tautologically for all $\rho$ with $\Im  \rho>0$ (see \eqref{tauto}).
Let us first show that the equality $\omega_n(\Lambda_n(\rho))=\rho$ holds when $\rho$ has 
a small enough inverse. 
The map $\Lambda_n$ has a power series expansion
$$
\Lambda_n(\rho)=\rho+\gamma+\sum_{v=1}^r\alpha_v\left(
\sum_{k=0}^\infty({\rm id}_m\otimes\tau)
\left(\rho^{-1}\left[\sum_{u=1}^t(\beta_u\otimes a_n^{(u)})\rho^{-1}\right]^k\right)\right)\alpha_v,
$$
convergent when $\|\rho^{-1}\|<\left\|\sum_{u=1}^t\beta_u\otimes a_n^{(u)}\right\|^{-1}$. For simplicity
we let $h(\lambda)=\sum_{v=1}^r\alpha_v\left(
\sum_{k=0}^\infty({\rm id}_m\otimes\tau)
\left(\lambda\left[\sum_{u=1}^t(\beta_u\otimes a_n^{(u)})\lambda\right]^k\right)\right)\alpha_v,$ norm
convergent on a ball of radius $\left\|\sum_{u=1}^t\beta_u\otimes a_n^{(u)}\right\|^{-1}$ and fixing zero. 
Performing the change of variable $\lambda=\rho^{-1}$, we obtain $\Lambda_n(\rho)=
\Lambda_n(\lambda^{-1})=\lambda^{-1}+\gamma+h(\lambda)$. Then
$(\Lambda_n(\lambda^{-1}))^{-1}=(\lambda^{-1}+\gamma+h(\lambda))^{-1}=
\lambda(1+(\gamma+h(\lambda))\lambda)^{-1}$, which is analytic on the set of all $\lambda\in
 M_m(\mathbb C)$ such that $\|\lambda\|<\left\|\sum_{u=1}^t\beta_u\otimes a_n^{(u)}\right\|^{-1}$
and $\|\gamma+h(\lambda)\|<\|\lambda\|^{-1}.$ 

Define ${\check\Lambda_n}(\rho)=(\Lambda_n(\rho^{-1}))^{-1}$ and ${\check\omega_n}
(\rho)=(\omega_n(\rho^{-1}))^{-1}$. We have established above that ${\check\Lambda_n}$ is analytic
on a neighbourhood of zero, and a direct computation shows that ${\check\Lambda_n}(0)=0,
{\check\Lambda_n}'(0)={\rm id}$. The inverse function theorem for analytic maps allows us to
conclude that there exists a neighbourhood of zero on which ${\check\Lambda_n}$ has a unique inverse
which fixes zero and whose derivative at zero is equal to the identity. 
The map ${\check\omega_n}$ is shown precisely the same way to satisfy the same properties as
${\check\Lambda_n}$. In particular, for $\|\rho\|$ small enough, 
${\check\Lambda_n}({\check\omega_n}(\rho))=(\Lambda_n({\check\omega_n}(\rho)^{-1}))^{-1}=
(\Lambda_n(((\omega_n(\rho^{-1}))^{-1})^{-1}))^{-1}=(\Lambda_n(\omega_n(\rho^{-1}))^{-1}=
(\rho^{-1})^{-1}=\rho$ for any $\rho$ with strictly positive imaginary part. Since zero is in the
closure of $\{\rho\in M_m(\mathbb C)\colon\Im  \rho>0\}$, it follows that ${\check\omega_n}$
and ${\check\Lambda_n}$ are compositional inverses to each other on a small enough neighbourhood of
zero. We conclude that for all $\rho$ such that the lower bound of the spectrum of $\Im  \rho$ is 
sufficiently large, $\omega_n(\Lambda_n(\rho))=\rho$.

Let now $\rho$ be fixed in $ M_m(\mathbb{C})$ such that $\Im \rho>0$ and $\Im \Lambda_n(\rho)>0$.
Let $\phi$ be  a positive linear functional on $ M_m(\mathbb{C})$ such that
$\phi(1)=1$ (i.e. a state). Define $\varphi_\rho(\cdot )=\phi(\cdot )/\phi(\Im \rho)$. It is linear and 
positive (well defined because $\Im  \rho\ge\frac{1}{\|(\Im  \rho)^{-1}\|}1$, so that $\phi(\Im  \rho)\ge
\frac{1}{\| (\Im  \rho)^{-1}\|}>0$). Define 
$$
f_\rho(z)=\varphi_\rho\left(\Lambda_n(\Re \rho+z\Im \rho)\right),\quad z\in\mathbb C^+.
$$
Note that $$f_\rho(z)=z+ \varphi_\rho(\gamma+\Re \rho))+F(z)$$
where $$F(z)=\frac{\phi\left[ \sum_{v=1}^r \alpha_v {\rm id}_m\otimes \tau \left\{\left( (\Re  \rho +z \Im  \rho)\otimes 1_{\cal A} -\sum_{u=1}^t \beta_u \otimes a_n^{(u)} \right)^{-1}\right\}\alpha_v\right] }{\phi(\Im \rho)}.$$
$F(z)$ is analytic on $\mathbb{C}\setminus \mathbb{R}$ and satisfies $\overline{F(z)}=F(\bar{z}).$
Let $z\in \mathbb{C}^+$.
We have \\

\noindent $\Im  F(z)$ $$= \frac{\phi\left[ \sum_{v=1}^r \alpha_v{\rm id}_m\otimes \tau \left\{ \Im  \left\{\left( (\Re  \rho +z \Im  \rho)\otimes 1_{\cal A} -\sum_{u=1}^t \beta_u \otimes a_n^{(u)} \right)^{-1}\right\} \right\}\alpha_v\right] }{\phi(\Im \rho)}$$
where \\

\noindent $\Im  \left\{\left( (\Re  \rho +z \Im  \rho)\otimes 1_{\cal A} -\sum_{u=1}^t \beta_u \otimes a_n^{(u)} \right)^{-1}\right\}$ \begin{eqnarray*}&=&
-\Im  z \left(( \Re  \rho +z \Im  \rho)\otimes 1_{\cal A} -\sum_{u=1}^t \beta_u \otimes a_n^{(u)} \right)^{-1} (\Im  \rho \otimes 1_{\cal A})\\&&~~~~~~~~~~~~~~~~~~~~~~~~~~~\times \left( (\Re  \rho +\bar{z} \Im  \rho)\otimes 1_{\cal A} -\sum_{u=1}^t \beta_u \otimes a_n^{(u)} \right)^{-1}<0.
\end{eqnarray*}
It follows by the complete positivity of the trace $\tau$ that $${\rm id}_m\otimes \tau\left\{ \Im  \left\{\left(( \Re  \rho +z \Im  \rho)\otimes 1_{\cal A} -\sum_{u=1}^t \beta_u \otimes a_n^{(u)} \right)^{-1}\right\}\right\}<0.$$
Now, according to Remark \ref{remarqueinversible}, we can assume the $\alpha_v$'s invertible so that 
$\sum_{v=1}^r \alpha_v{\rm id}_m\otimes \tau \left\{\Im  \left\{\left( (\Re  \rho +z \Im  \rho) \otimes 1_{\cal A} -\sum_{u=1}^t \beta_u \otimes a_n^{(u)} \right)^{-1}\right\}\right\}\alpha_v <0$ and then $\Im  F(z) <0$.
Thus for any $z\in \mathbb{C}\setminus \mathbb{R}$, we have $\Im z \Im  F(z)<0.$
Finally 
$$\lim_{y \rightarrow +\infty} iy F(iy)= \varphi_\rho \left( \sum_{v=1}^r\alpha_v (\Im  \rho )^{-1} \alpha_v \right):=c_\rho >0. $$
Thus, by Akhiezer-Krein's Theorem (\cite{AK} page 93), there exists a probability measure $\mu$ on $\mathbb {R}$ such that
$$
F(z)= c_\rho \int_\mathbb R\frac{d\mu(s)}{z-s}
,\quad z\in\mathbb C^+.
$$
Then
$$
f_{\rho}(z)=z+\varphi_\rho(\gamma+\Re \rho)+c_\rho \int_\mathbb R\frac{d\mu(s)}{z-s}
,\quad z\in\mathbb C^+.
$$
Thus $\Im f_{\rho}(u+iv)=v\left(1-c_\rho\int_\mathbb R\frac{d\mu(s)}{(u-s)^2+v^2}\right)$. We observe 
that what's under parenthesis is strictly increasing in $v$. Since by hypothesis, we have $\Im  \Lambda_n(\rho)>0$ and thus  $\Im f_{\rho}(i)=\Im \varphi_\rho\left(\Lambda_n(\Re \rho+i\Im \rho)\right)
>0$, we obtain 
immediately that $\Im f_{\rho}(iv)>0$ for all $v\ge1$. This means that $\Im 
\phi(\Lambda_n(\Re \rho+iv\Im \rho))>0$ for all $v\in[1,+\infty)$ and all states $\phi$, so that \begin{equation}\label{ray}\Im \Lambda_n (\Re \rho+iv\Im \rho)
>0, {\rm~ for~ all~} v\ge1.\end{equation}
Now it is clear that $$\Omega=\{z \in \C^+, \Im \Lambda_n(\Re  \rho + z\Im  \rho) >0\}$$
is an open set which contains $d=\{iv, v\geq 1\}$.
Let $\Omega_d$ be the connected component of $\Omega$ which contains $d$. Note that $\Omega_d $ is an open set.

As we have shown at the beginning of our proof, for given $\rho,\Im   \rho>0$, there exists an $M>0$ (possibly depending on $\rho$) such that $\omega_n(\Lambda_n(\Re  \rho + iv\Im  \rho))=\Re  \rho + iv\Im  \rho$ for all $v>M.$
By the identity principle for analytic functions, we immediately obtain that 
$ \omega_n(\Lambda_n(\Re \rho +z \Im \rho))=\Re  \rho +z \Im  \rho$ for all $z\in \Omega_d$ and in particular for $z=i$. The 
proof of Lemma \ref{inversion} is complete.
\end{proof}

\section{Proof of Lemma \ref{inclu2}}\label{lemmefonda} \subsection{ Sharp estimates of Stieltjes transforms} The proof of \eqref{spectre3} requires the  sharp estimate \eqref{estimdiffeqno} we are going to prove here.

\noindent According to Section \ref{troncation}, from now on, we assume that the $X_{ij}^{(v)}$'s satisfy (H).
Note that this assumption  implies that for any  $v\in \{1,\ldots,r\}$, 
$$\forall i\geq 1, \forall j \geq  1, \;\kappa_1^{i,j,v}=0, \;\kappa_2^{i,j,v}=1,$$
 $$\forall i\geq 1, \forall j \geq  1,\;,  i\neq j, \; \tilde \kappa_1^{i,j,v}=0,\; \tilde \kappa_2^{i,j,v}=1$$ and  
 for any $p\in \mathbb{N}\setminus\{0\}$, \begin{equation}\label{cumulants}\max_{v=1,\ldots,r} \sup_{i\geq 1, j \geq 1} \vert \kappa_p^{i,j,v}\vert<+\infty, \; \max_{v=1,\ldots,r} \sup_{i\geq 1, j\geq 1} \vert \tilde  \kappa_p^{i,j,v}\vert<+\infty, \end{equation}
where for $i\neq j$,  $(\kappa_p^{i,j,v})_{p\geq1}$ and $(\tilde \kappa_p^{i,j,v})_{p\geq 1}$ denote  the classical cumulants of $\sqrt{2}\Re  X_{ij}^{(v)}$ and $\sqrt{2}\Im   X_{ij}^{(v)}$ respectively and $(\kappa_p^{i,i,v})_{p\geq 1}$ denotes  the classical cumulants of $ X_{ii}^{(v)}$ (we set  $(\tilde \kappa_p^{i,i,v})_{p\geq 1}\equiv 0$).\\

\noindent Now, we present our  main technical tool (see \cite{KKP}):
 \begin{lemme} \label{lem1}
Let $\xi$ be a real-valued random variable such that  $\mathbb{E}(\vert \xi
\vert^{p+2})<\infty$. Let  $\phi$ be a function from  $\R$ to $\C$
such that the first $p+1$ derivatives are  continuous and bounded. Then,
\begin{equation}\label{IPP}\mathbb E(\xi \phi(\xi)) = \sum_{a=0}^p
\frac{\kappa_{a+1}}{a!}\mathbb{E}(\phi^{(a)}(\xi)) + \epsilon\end{equation}
where  $\kappa_{a}$ are the classical  cumulants of  $\xi$, $\epsilon \leq C
\sup_t \vert \phi^{(p+1)}(t)\vert \mathbb{E}(\vert \xi \vert^{p+2})$, $C$
depends  on $p$ only.
\end{lemme}
In the following, we shall apply this identity with a function
$\phi(\xi)$ given by the entries  of the resolvent of $S_n$. It
follows from  Lemma \ref{lem2} and (\ref{resolvente}) below
that the conditions of Lemma \ref{lem1} (bounded derivatives) are
fulfilled.
We first need the following preliminary lemma.
 
\begin{lemme}\label{inverseY}
For any $\lambda \in M_m(\mathbb{C})  $ such that $\Im  \lambda$ is positive definite, $(\lambda -\gamma -\sum_{v=1}^r\alpha_v G_n(\lambda)\alpha_v)\otimes I_n -  \sum_{u=1}^t \beta_u \otimes A_n^{(u)}$ and $(\lambda -\gamma -\sum_{v=1}^r\alpha_v \tilde G_n(\lambda)\alpha_v)\otimes I_n -  \sum_{u=1}^t \beta_u \otimes A_n^{(u)}$ are invertible.
Set  \begin{equation}\label{y}Y_n(\lambda)=\left((\lambda -\gamma -\sum_{v=1}^r\alpha_v G_n(\lambda)\alpha_v)\otimes I_n -  \sum_{u=1}^t \beta_u \otimes A_n^{(u)} \right)^{-1}\end{equation}
and  \begin{equation}\label{ytilde}\tilde Y_n(\lambda)=\left((\lambda -\gamma -\sum_{v=1}^r\alpha_v \tilde G_n(\lambda)\alpha_v)\otimes I_n -  \sum_{u=1}^t \beta_u \otimes A_n^{(u)} \right)^{-1}.\end{equation}
We have \begin{equation} \label{Y}\left\| Y_n(\lambda) \right\| \leq \Vert ( \Im  \lambda)^{-1} \Vert\end{equation}
and \begin{equation} \label{Y2}\left\| \tilde Y_n(\lambda) \right\| \leq \Vert ( \Im  \lambda)^{-1} \Vert.\end{equation}

\end{lemme}
\begin{proof}
We only present the proof for $Y_n(\lambda)$ since the proof is similar for $\tilde Y_n(\lambda)$.
Note that 
\begin{eqnarray*}
\Im  \left[ \left( \lambda\otimes I_n -S_n\right)^{-1}\right]
&=& \frac{1}{2i}\left[\left( \lambda\otimes I_n -S_n\right)^{-1}-\left( \lambda^*\otimes I_n -S_n\right)^{-1}\right]\\
&=& -\left( \lambda\otimes I_n -S_n\right)^{-1}\left( \Im  \lambda \otimes I_n\right) \left( \lambda^*\otimes I_n -S_n\right)^{-1}.
\end{eqnarray*}
This yields that  $-\Im  R_n(\lambda) $ is positive definite. Since the map ${\rm id }_m\otimes \tr_n$ is positive we can deduce that $-\Im  H_n(\lambda)$ is positive and then that $-\Im  G_n(\lambda)$ is positive. It readily follows that
\begin{equation}\label{image} \Im  \left[\lambda -\gamma -\sum_{v=1}^r\alpha_v G_n(\lambda)\alpha_v \right] \geq \Im  \lambda\end{equation} and then
$$\Im  \left[\left(\lambda -\gamma -\sum_{v=1}^r\alpha_v G_n(\lambda)\alpha_v\right)\otimes I_n -  \sum_{u=1}^t \beta_u \otimes A_n^{(u)}\right] \geq \Im  \lambda\otimes I_n.$$ Hence Lemma \ref{inverseY} follows  by lemma 3.1 in \cite{HT}.
\end{proof}
\begin{theoreme}\label{resolvante}
For any $\lambda \in M_m(\mathbb{C})  $ such that $\Im \lambda$ is positive definite, we have 
\begin{equation}\label{mast}\mathbb{E} \left(R_n(\lambda)\right)=Y_n(\lambda)+Y_n(\lambda)\Xi(\lambda)\end{equation}
where $Y_n(\lambda)$ is defined in Lemma \ref{inverseY}
\noindent and  $\Xi(\lambda)=\sum_{l,j}\Xi_{lj}(\lambda)\otimes E_{lj}$ satisfies that   for all $l,j\in \{1, \ldots,n\}$, $$ \Xi_{lj}(\lambda)=\Psi_{lj}(\lambda) +O_{lj}^{(u)}( \frac{1}{n^2})$$
where
 \begin{eqnarray}\Psi_{lj}(\lambda)& =&\sum_{v=1}^r\bigg\{\mathbb{E}\left[ \alpha_v [H_n(\lambda) -\mathbb{E}(H_n(\lambda))] \alpha_v    [R_n(\lambda ) ]_{lj} \right]\nonumber
\\&&+ \frac{1}{2\sqrt{2}n\sqrt{n}}  \sum_{i=1}^n  \left\{1-\delta_{il}\left(1-\frac{1}{\sqrt{2}}\right)\right\}M^{(3)}(v,i,l,j)\nonumber\\&&+  \frac{1}{4n^2} \sum_{i=1}^n \left(1-\frac{1}{2}\delta_{il}\right) M^{(4)}(v,i,l,j)\nonumber\\&&\left.
+  \frac{1}{4\sqrt{2}n^2\sqrt{n}}\sum_{i=1}^n  \left[1-\delta_{il}\left(1-\frac{1}{2\sqrt{2}}\right)\right]M^{(5)}(v,i,l,j)\right\},\label{psi}\end{eqnarray}
with \\

\noindent $M^{(3)}(v,i,l,j)$
\begin{eqnarray}=& \mathbb{E} \{(\kappa_3^{i,l,v}+\tilde \kappa_3^{i,l,v}\sqrt{-1})\alpha_v(R_n(\lambda))_{ii}\alpha_v(R_n(\lambda))_{li}\alpha_v 
 (R_n(\lambda))_{lj}\nonumber \\
 &+(\kappa_3^{i,l,v}-\tilde \kappa_3^{i,l,v}\sqrt{-1})\alpha_v (R_n(\lambda))_{ii}\alpha_v (R_n(\lambda))_{ll} \alpha_v  (R_n(\lambda))_{ij}\label{except}
  \\
&+(\kappa_3^{i,l,v}-\tilde \kappa_3^{i,l,v}\sqrt{-1})\alpha_v (R_n(\lambda))_{il}\alpha_v (R_n(\lambda))_{ii} \alpha_v  (R_n(\lambda))_{lj} \nonumber \\
&+(\kappa_3^{i,l,v}+\tilde \kappa_3^{i,l,v}\sqrt{-1})\alpha_v (R_n(\lambda))_{il}\alpha_v (R_n(\lambda))_{il} \alpha_v  (R_n(\lambda))_{ij}\}, \nonumber \end{eqnarray}

\noindent $M^{(4)}(v,i,l,j)$
\begin{eqnarray}=&(\kappa_4^{i,l,v}+\tilde \kappa_4^{i,l,v}) \mathbb{E} \{\alpha_v(R_n(\lambda))_{il}\alpha_v(R_n(\lambda))_{il}\alpha_v 
(R_n(\lambda))_{il} \alpha_v (R_n(\lambda))_{ij}\label{premierk4}
\\&+\alpha_v (R_n(\lambda))_{il}\alpha_v (R_n(\lambda))_{ii} \alpha_v (R_n(\lambda))_{li}\alpha_v (R_n(\lambda))_{lj}\label{deuxcas4} \\&+\alpha_v (R_n(\lambda))_{ii}\alpha_v (R_n(\lambda))_{ll} \alpha_v(R_n(\lambda))_{ii}\alpha_v (R_n(\lambda))_{lj}\label{troiscas4}\\& + \alpha_v (R_n(\lambda))_{ii} \alpha_v(R_n(\lambda))_{li}\alpha_v (R_n(\lambda))_{ll}\alpha_v (R_n(\lambda))_{ij}\} \label{quatrecas4}\\
&\hspace*{-0.4cm}+(\kappa_4^{i,l,v}~\hspace*{-0.4cm}-\tilde \kappa_4^{i,l,v}) \mathbb{E} \{\alpha_v(R_n(\lambda))_{ii}\alpha_v(R_n(\lambda))_{li}\alpha_v 
(R_n(\lambda))_{li} \alpha_v (R_n(\lambda))_{lj} \label{tilde1}
\\&+\alpha_v (R_n(\lambda))_{il}\alpha_v (R_n(\lambda))_{ii} \alpha_v (R_n(\lambda))_{ll}\alpha_v (R_n(\lambda))_{ij} \label{tilde2} \\&+\alpha_v (R_n(\lambda))_{il}\alpha_v (R_n(\lambda))_{il} \alpha_v(R_n(\lambda))_{ii}\alpha_v (R_n(\lambda))_{lj} \label{tilde3}\\& + \alpha_v (R_n(\lambda))_{ii} \alpha_v(R_n(\lambda))_{ll}\alpha_v (R_n(\lambda))_{il}\alpha_v (R_n(\lambda))_{ij}\}, \label{tilde4}
\end{eqnarray}

\noindent $M^{(5)}(v,i,l,j)$
\begin{eqnarray*}=& \mathbb{E} \{(\kappa_5^{i,l,v}+\tilde\kappa_5^{i,l,v}\sqrt{-1}) ~~~~~~~~~~~~~~~~~~~~~~~~~~~~~~~~~~~~~~~~~~~~~~~~~~~~~~~~~~~~~~~~~~~~~~
\\& \times \left[\alpha_v(R_n(\lambda))_{ii}\alpha_v(R_n(\lambda))_{li}\alpha_v 
 (R_n(\lambda))_{li}\alpha_v (R_n(\lambda))_{ll} \alpha_v  (R_n(\lambda))_{ij}\right.\\
 &+\alpha_v (R_n(\lambda))_{ii}\alpha_v (R_n(\lambda))_{li} \alpha_v  (R_n(\lambda))_{ll} \alpha_v (R_n(\lambda))_{ii} \alpha_v (R_n(\lambda))_{lj} \\
&+\alpha_v (R_n(\lambda))_{ii}\alpha_v (R_n(\lambda))_{ll} \alpha_v  (R_n(\lambda))_{ii}\alpha_v(R_n(\lambda))_{li}\alpha_v 
 (R_n(\lambda))_{lj} \\
&+\alpha_v (R_n(\lambda))_{ii}\alpha_v (R_n(\lambda))_{ll} \alpha_v  (R_n(\lambda))_{il} \alpha_v (R_n(\lambda))_{il} \alpha_v (R_n(\lambda))_{ij} \\&+\alpha_v (R_n(\lambda))_{il}\alpha_v (R_n(\lambda))_{ii} \alpha_v  (R_n(\lambda))_{li}  \alpha_v  (R_n(\lambda))_{li}\alpha_v  (R_n(\lambda))_{lj} \\
&+  \alpha_v  (R_n(\lambda))_{il}\alpha_v (R_n(\lambda))_{ii}\alpha_v (R_n(\lambda))_{ll} \alpha_v  (R_n(\lambda))_{il} \alpha_v  (R_n(\lambda))_{ij} \\
&+ \alpha_v  (R_n(\lambda))_{il}\alpha_v (R_n(\lambda))_{il}\alpha_v (R_n(\lambda))_{ii} \alpha_v  (R_n(\lambda))_{ll} \alpha_v  (R_n(\lambda))_{ij} \\
&+\left. \alpha_v  (R_n(\lambda))_{il}\alpha_v (R_n(\lambda))_{il}\alpha_v (R_n(\lambda))_{il} \alpha_v  (R_n(\lambda))_{ii} \alpha_v  (R_n(\lambda))_{lj} \right]\\
&+(\kappa_5^{i,l,v}-\tilde\kappa_5^{i,l,v}\sqrt{-1}) ~~~~~~~~~~~~~~~~~~~~~~~~~~~~~~~~~~~~~~~~~~~~~~~~~~~~~~~~~~~~~~~~~~~~~~\\&\left[
\alpha_v(R_n(\lambda))_{ii}\alpha_v(R_n(\lambda))_{li}\alpha_v 
 (R_n(\lambda))_{li}\alpha_v (R_n(\lambda))_{li} \alpha_v  (R_n(\lambda))_{lj}\right.\\
&+\alpha_v (R_n(\lambda))_{ii}\alpha_v (R_n(\lambda))_{li} \alpha_v  (R_n(\lambda))_{ll} \alpha_v (R_n(\lambda))_{il} \alpha_v  (R_n(\lambda))_{ij} \\
&+\alpha_v (R_n(\lambda))_{ii}\alpha_v (R_n(\lambda))_{ll} \alpha_v  (R_n(\lambda))_{ii}\alpha_v(R_n(\lambda))_{ll}\alpha_v 
 (R_n(\lambda))_{ij} \\
&+\left. \alpha_v  (R_n(\lambda))_{ii}\alpha_v (R_n(\lambda))_{ll}\alpha_v (R_n(\lambda))_{il} \alpha_v  (R_n(\lambda))_{ii} \alpha_v  (R_n(\lambda))_{lj} \right]\\
&+\alpha_v (R_n(\lambda))_{il}\alpha_v (R_n(\lambda))_{ii} \alpha_v  (R_n(\lambda))_{li} \alpha_v (R_n(\lambda))_{ll} \alpha_v (R_n(\lambda))_{ij} \\
&+\alpha_v (R_n(\lambda))_{il}\alpha_v (R_n(\lambda))_{ii} \alpha_v  (R_n(\lambda))_{ll}  \alpha_v  (R_n(\lambda))_{ii}\alpha_v  (R_n(\lambda))_{lj} \\
&+  \alpha_v  (R_n(\lambda))_{il}\alpha_v (R_n(\lambda))_{il}\alpha_v (R_n(\lambda))_{ii} \alpha_v  (R_n(\lambda))_{li} \alpha_v  (R_n(\lambda))_{lj} \\
&+\left. \alpha_v  (R_n(\lambda))_{il}\alpha_v (R_n(\lambda))_{il}\alpha_v (R_n(\lambda))_{il} \alpha_v  (R_n(\lambda))_{il} \alpha_v  (R_n(\lambda))_{ij}\right]\}.
\end{eqnarray*}
\end{theoreme}

 \begin{proof}
 We shall apply formula \eqref{IPP} to the ${ M}_m(\C)$-valued function $\phi(\xi) = (R_n(\lambda))_{ij}$ for $1 \leq i,j \leq n$ and $\xi$ is one of the variable $\frac{X^{(v)}_{kk}}{\sqrt{n}}$,
$\sqrt{2} \frac{Re(X^{(v)}_{kl})}{\sqrt{n}}$, $\sqrt{2} \frac{Im(X^{(v)}_{kl})}{\sqrt{n}}$ for $1 \leq k<l \leq n$ and $1\leq v \leq r$. \\
We notice that
$$\frac{\partial \phi}{\partial Re(\frac{X^{(v)}_{kk}}{\sqrt{n}})}  = \phi'_{\frac{X_n^{(v)}}{\sqrt{n}}} . e_{kk}, \;\frac{X^{(v)}_{kk}}{\sqrt{n}} = \Tr(\frac{X_n^{(v)}}{\sqrt{n}} e_{kk}), \ 1 \leq k \leq n,$$
$$\frac{\partial \phi}{\partial \sqrt{2} Re(\frac{X^{(v)}_{kl}}{\sqrt{n}})}  = \phi'_{\frac{X_n^{(v)}}{\sqrt{n}}} . e_{kl}, \; \sqrt{2} Re(\frac{X^{(v)}_{kl}}{\sqrt{n}}) = \Tr(\frac{X_n^{(v)}}{\sqrt{n}} e_{kl}), \ 1 \leq k <l \leq n, $$
$$\frac{\partial \phi}{\partial \sqrt{2} Im(\frac{X^{(v)}_{kl}}{\sqrt{n}})}  = \phi'_{\frac{X_n^{(v)}}{\sqrt{n}}} . f_{kl} , \;\sqrt{2} Im(\frac{X^{(v)}_{kl}}{\sqrt{n}}) = \Tr(\frac{X_n^{(v)}}{\sqrt{n}} f_{kl}), \ 1 \leq k <l \leq n. $$

\noindent Let  $1\leq k < l \leq n$  be fixed.
For simplicity, we write $\phi'$, $\phi''$, $\phi'''$ and  $\phi''''$ for the
first derivatives of $\phi$ with respect to $\sqrt{2}
Re(\frac{X^{(v)}_{kl}}{\sqrt{n}})$. Then, according to (\ref{resolvente}),
\begin{eqnarray*}
\phi' &=& \left[ R_n(\lambda) \alpha_v \otimes e_{kl} R_n(\lambda) \right]_{ij} \\
\phi''&=&2 \left[R_n (\lambda) \alpha_v \otimes e_{kl}R_n (\lambda)  \alpha_v \otimes e_{kl}R_n (\lambda)\right]_{ij} \\
\phi'''&=&6 \left[ R_n (\lambda) \alpha_v \otimes e_{kl}R_n (\lambda)  \alpha_v \otimes e_{kl} R_n (\lambda)  \alpha_v \otimes e_{kl} R_n(\lambda)\right]_{ij}\\
\phi''''&=&24 \left[R_n (\lambda) \alpha_v \otimes e_{kl}R_n (\lambda)  \alpha_v \otimes e_{kl} R_n (\lambda)  \alpha_v \otimes e_{kl} R_n(\lambda)\alpha_v \otimes e_{kl} R_n (\lambda)\right]_{ij}
\end{eqnarray*}
Writing \eqref{IPP} in this setting gives
$$\mathbb{E}[ \Tr(\frac{X_n^{(v)}}{\sqrt{n}}e_{kl}) (R_n (\lambda))_{ij} ] = \frac{1}{n} \mathbb{E}[ \phi'] +
\frac{\kappa_3^{k,l,v}}{2n\sqrt{n}} \mathbb{E}[\phi'']  +
\frac{\kappa_4^{k,l,v}}{6n^2} \mathbb{E}[\phi''']+ \frac{\kappa_5^{k,l,v}}{24 n^2 \sqrt{n}}\mathbb{E}[\phi'''']$$
\begin{equation} \label{1}+O_{k,l,i,j,v}\left(\frac{1}{n^3}\right)
\end{equation}
where there exists $C>0$ such that for every $k,l,i,j \in\{1,\ldots,n\}$ and every $v\in \{1,\ldots,r\}$, \begin{equation}\label{majuniv}\left\|  O_{k,l,i,j}\left(\frac{1}{n^3}\right)\right\|\ \leq \frac{C\Vert \alpha_v \Vert^5 \Vert ( \Im \lambda)^{-1} \Vert^6}{n^3},\end{equation}
with the analogous equations with $f_{kl}$  and $e_{kk}$ replacing  the $\kappa_i^{k,l,v}$'s by the  $\tilde \kappa_i^{k,l,v}$'s and $\kappa_i^{k,k,v}$'s respectively. \\

\noindent
Noticing that for $ k<l$, $$E_{kl} =\frac{e_{kl} -\sqrt{-1}f_{kl}}{\sqrt{2}}\;
\mbox{and }\; E_{lk} =\frac{e_{kl} +\sqrt{-1}f_{kl}}{\sqrt{2}},$$
  we deduce that :\\

$\mathbb{E}[ \Tr(\frac{X_n^{(v)}}{\sqrt{n}} E_{kl}) (R_n(\lambda))_{ij} ] $
\begin{eqnarray*} &=& \frac{1}{n}\mathbb{E}\left[ R_n(\lambda) \alpha_v \otimes E_{kl} R_n(\lambda)\right]_{ij}\\
&&+ \frac{\kappa_3^{k,l,v}}{  \sqrt{2}n\sqrt{n}}\mathbb{E} \left\{\left[ R_n(\lambda) \alpha_v \otimes e_{kl} R_n(\lambda) \alpha_v \otimes e_{kl}  R_n(\lambda) \right]_{ij}\right\}\\ && -\sqrt{-1} \frac{\tilde \kappa_3^{k,l,v}}{  \sqrt{2}n\sqrt{n}}\mathbb{E} \left\{ \left[ R_n(\lambda) \alpha_v \otimes f_{kl} R_n(\lambda) \alpha_v \otimes f_{kl} R_n(\lambda) )\right]_{ij}\right\}\\
&&+ \frac{\kappa_4^{k,l,v}}{  \sqrt{2}n^2}\mathbb{E} \left\{\left[ R_n(\lambda) \alpha_v \otimes e_{kl} R_n(\lambda) \alpha_v \otimes e_{kl}  R_n(\lambda) \alpha_v \otimes e_{kl}R_n(\lambda)\right]_{ij}\right\}\\ && -\sqrt{-1}  \frac{\tilde \kappa_4^{k,l,v}}{  \sqrt{2}n^2}\mathbb{E} \left\{\left[ R_n(\lambda) \alpha_v \otimes f_{kl} R_n(\lambda) \alpha_v \otimes f_{kl} R_n(\lambda) \alpha_v \otimes f_{kl} R_n(\lambda)\right]_{ij}\right\}\\
&&+ \frac{\kappa_5^{k,l,v}}{  \sqrt{2}n^2\sqrt{n}} \\&&~~~~\times \mathbb{E} \left\{\left[ R_n(\lambda) \alpha_v \otimes e_{kl} R_n(\lambda) \alpha_v \otimes e_{kl}  R_n(\lambda) \alpha_v \otimes e_{kl}R_n(\lambda) \alpha_v \otimes e_{kl}R_n(\lambda)\right]_{ij}\right\}\\ && -\sqrt{-1} \frac{\tilde \kappa_5^{k,l,v}}{  \sqrt{2}n^2\sqrt{n}}  \\&&~~~ \times \mathbb{E} \left\{ \left[ R_n(\lambda) \alpha_v \otimes f_{kl} R_n(\lambda) \alpha_v \otimes f_{kl} R_n(\lambda) \alpha_v \otimes f_{kl} R_n(\lambda) \alpha_v \otimes f_{kl} R_n(\lambda)\right]_{ij}\right\}\\
&&+ O_{l,k,i,j,v}\left(\frac{1}{n^3}\right),
 \end{eqnarray*}

$\mathbb{E}[ \Tr(\frac{X^{(v)}_n }{\sqrt{n}}E_{lk}) (R_n(\lambda))_{ij} ]$
\begin{eqnarray*}  &=& \frac{1}{n}\mathbb{E}\left[ R_n(\lambda) \alpha_v \otimes E_{lk} R_n(\lambda)\right]_{ij} \\
&&+ \frac{\kappa_3^{k,l,v}}{  \sqrt{2}n\sqrt{n}}\mathbb{E} \left\{\left[ R_n(\lambda) \alpha_v \otimes e_{kl} R_n(\lambda) \alpha_v \otimes e_{kl}  R_n(\lambda) \right]_{ij}\right\}\\ &&+\sqrt{-1}\frac{\tilde \kappa_3^{k,l,v}}{  \sqrt{2}n\sqrt{n}}\mathbb{E} \left\{ \left[ R_n(\lambda) \alpha_v \otimes f_{kl} R_n(\lambda) \alpha_v \otimes f_{kl} R_n(\lambda) )\right]_{ij}\right\}\\
&& + \frac{\kappa_4^{k,l,v}}{  \sqrt{2}n^2}\mathbb{E} \left\{\left[ R_n(\lambda) \alpha_v \otimes e_{kl} R_n(\lambda) \alpha \otimes e_{kl}  R_n(\lambda) \alpha_v \otimes e_{kl}R_n(\lambda)\right]_{ij}\right\}\\ && +\sqrt{-1}  \frac{\tilde\kappa_4^{k,l,v}}{  \sqrt{2}n^2}\mathbb{E} \left\{ \left[ R_n(\lambda) \alpha_v \otimes f_{kl} R_n(\lambda) \alpha_v \otimes f_{kl} R_n(\lambda) \alpha_v \otimes f_{kl} R_n(\lambda)\right]_{ij}\right\}\\
&&+ \frac{\kappa_5^{k,l,v}}{  \sqrt{2}n^2\sqrt{n}}\\ &&~~\times \mathbb{E} \left\{\left[ R_n(\lambda) \alpha_v \otimes e_{kl} R_n(\lambda) \alpha_v \otimes e_{kl}  R_n(\lambda) \alpha_v \otimes e_{kl}R_n(\lambda) \alpha_v \otimes e_{kl}R_n(\lambda)\right]_{ij}\right\}\\ &&+\sqrt{-1}
\frac{\tilde \kappa_5^{k,l,v}}{  \sqrt{2}n^2\sqrt{n}} \\ &&~~ \times \mathbb{E} \left\{ \left[ R_n(\lambda) \alpha_v \otimes f_{kl} R_n(\lambda) \alpha_v \otimes f_{kl} R_n(\lambda) \alpha_v \otimes f_{kl} R_n(\lambda) \alpha_v \otimes f_{kl} R_n(\lambda)\right]_{ij}\right\}\\
&&+ O_{k,l,i,j,v}\left(\frac{1}{n^3}\right)
 \end{eqnarray*}
and \\

\noindent $\mathbb{E}[ \Tr(\frac{X^{(v)}_n}{\sqrt{n}} E_{kk}) (R_n(\lambda))_{ij} ]$
\begin{eqnarray*}  &=& \frac{1}{n}\mathbb{E}\left[ R_n(\lambda) \alpha_v \otimes E_{kk} R_n(\lambda)\right]_{ij}\\&&  + \frac{\kappa_3^{k,k,v}}{  n\sqrt{n}} \mathbb{E} \left\{\left[  R_n(\lambda) \alpha_v \otimes e_{kk}  R_n(\lambda) \alpha_v \otimes e_{kk}R_n(\lambda)\right]_{ij}\right\}
\\&&  + \frac{\kappa_4^{k,k,v}}{  n^2} \mathbb{E} \left\{\left[ R_n(\lambda) \alpha_v \otimes e_{kk} R_n(\lambda) \alpha_v \otimes e_{kk}  R_n(\lambda) \alpha_v \otimes e_{kk}R_n(\lambda)\right]_{ij}\right\}\\
&&  + \frac{\kappa_5^{k,k,v}}{  n^2 \sqrt{n}} \mathbb{E} \left\{\left[ R_n(\lambda) \alpha_v \otimes e_{kk} R_n(\lambda) \alpha_v \otimes e_{kk}  R_n(\lambda) \alpha_v \otimes e_{kk}R_n(\lambda)e_{kk}R_n(\lambda)\right]_{ij}\right\}\\
&&+  O_{k,k,i,j,v}\left(\frac{1}{n^3}\right)
\end{eqnarray*}
where the $O_{k,l,i,j}\left(\frac{1}{n^3}\right)$'s satisfy \eqref{majuniv}.
We then obtain:
  \begin{eqnarray}\mathbb{E}[(\alpha_v \otimes \frac{X^{(v)}_n}{\sqrt{n}}) R_n(\lambda ) ]_{lj}& =& \mathbb{E}\left[ \alpha_v H_n(\lambda) \alpha_v    [R_n(\lambda ) ]_{lj}  \right]\nonumber
\\&&+ \frac{1}{2\sqrt{2}n\sqrt{n}}  \sum_{i=1}^n  \left\{1-\delta_{il}\left(1-\frac{1}{\sqrt{2}}\right)\right\}M^{(3)}(v,i,l,j)\nonumber \\&&+  \frac{1}{4n^2} \sum_{i=1}^n \left(1-\frac{1}{2}\delta_{il}\right) M^{(4)}(v,i,l,j)\nonumber\\&&
+  \frac{1}{4\sqrt{2}n^2\sqrt{n}}\sum_{i=1}^n  \left\{1-\delta_{il}\left(1-\frac{1}{2\sqrt{2}}\right)\right\}M^{(5)}(v,i,l,j)\nonumber\\&&+O_{l,j,v}(1/n^2),\label{4}\end{eqnarray}

\noindent  where the $M^{(q)}(v,i,l,j)$ for $ q=3,4,5$ are defined in Theorem \ref{resolvante} and 
 there exists $C>0$ such that for every $l,j \in\{1,\ldots,n\}$ and any $v \in \{1,\ldots,r\}$, \begin{equation}\label{majuniv2}\left\|  O_{l,j,v}\left(\frac{1}{n^2}\right)\right\|\ \leq \frac{C\Vert \alpha_v \Vert^5 \Vert ( \Im \lambda)^{-1} \Vert^6}{n^2}.\end{equation}

\noindent Now,
\begin{eqnarray*}
\sum_{v=1}^r(\alpha_v   \otimes X_n^{(v)})R_n (\lambda)
 &=&  (S_n - \sum_{u=1}^t \beta_u \otimes A_n^{(u)}-\gamma \otimes I_n) (\lambda \otimes I_n - S_n)^{-1}  \\
 &=& -I_m\otimes I_n  +\left[ (\lambda - \gamma) \otimes I_n -\sum_{u=1}^t \beta_u \otimes A_n^{(u)} \right] R_n(\lambda)
 \end{eqnarray*}
 implying
\begin{eqnarray}\sum_{v=1}^r \mathbb{E}[(\alpha_v \otimes \frac{X_n^{(v)}}{\sqrt{n}}) R_n(\lambda ) ]_{lj} &=& -\delta_{jl} I_m +(\lambda-\gamma) \mathbb{E} (R_n(\lambda))_{lj}\nonumber
\\&&- \left[\sum_{u=1}^t \beta_u \otimes A_n^{(u)}  \mathbb{E}( R_n(\lambda))\right]_{lj}. \label{identiteresol}\end{eqnarray}
On the other hand, we have 
\begin{eqnarray}\mathbb{E}\left[ \alpha_v H_n(\lambda) \alpha_v    [R_n(\lambda ) ]_{lj}  \right]&= &  \mathbb{E}\left[ \alpha_v [H_n(\lambda) -\mathbb{E}(H_n(\lambda))] \alpha_v    [R_n(\lambda ) ]_{lj} \right]\nonumber\\ &&
+
\alpha_v G_n(\lambda) \alpha_v \mathbb{E}\left[ [R_n(\lambda ) ]_{lj} \right]
.\label{centrage}\end{eqnarray}
Hence \eqref{4}, \eqref{identiteresol} and \eqref{centrage} yield \\

 $ -\delta_{jl}I_m +(\lambda-\gamma) \mathbb{E} (R_n(\lambda))_{lj}
- \left[\sum_{u=1}^t \beta_u \otimes A_n^{(u)}  \mathbb{E}( R_n(\lambda))\right]_{lj}$ \begin{equation}\label{me} = \sum_{v=1}^r\alpha_v G_n(\lambda) \alpha_v \mathbb{E}\left[ [R_n(\lambda ) ]_{lj} \right] +
\Xi_{lj}(\lambda) \end{equation}
where $$\Xi_{lj}(\lambda)=\Psi_{lj}(\lambda)
+O^{(u)}_{l,
j}(1/n^2)$$ and  $\Psi_{lj}$ is defined in Theorem \ref{resolvante}.
Thus, we have 
$$ \left[\left(\lambda-\gamma -\sum_{v=1}^r\alpha_v G_n(\lambda) \alpha_v\right)\otimes I_n - \sum_{u=1}^t \beta_u \otimes A_n^{(u)} \right]  \mathbb{E} (R_n(\lambda))
 =  I_n\otimes I_m +
\Xi(\lambda).$$
\eqref{mast} readily follows. \end{proof}

\begin{proposition}\label{55}
For any $p,q\in \{1,\ldots,n\}^2$, for any $mn\times mn $ deterministic matrix $F_n(\lambda)$ such that  
$  F_n(\lambda)=O(1)$, setting $\Psi(\lambda)=\sum_{l,j}\Psi_{lj}(\lambda)\otimes E_{lj}$ where $\Psi_{lj}$ is defined by \eqref{psi}, we have\\

$\left\{Y_n(\lambda)\Psi(\lambda)F_n(\lambda)\right\}_{pq}$
\begin{eqnarray} &=&\frac{1}{2\sqrt{2}n\sqrt{n}}  \sum_{v=1}^r\sum_{i,l=1}^n (\kappa_3^{i,l,v} -\tilde \kappa_3^{i,l,v}\sqrt{-1})(Y_n)_{pl}\nonumber \\&& ~~~~~~~~~~~~~~~~~~~~~~~~~\times \mathbb{E} \{\alpha_v(R_n(\lambda))_{ii}\alpha_v(R_n(\lambda))_{ll}\alpha_v 
 (R_n(\lambda)F_n(\lambda))_{iq}\} \nonumber\\&&+O^{(u)}_{p,q}(\frac{1}{{n}}). \label{YPSIF}\end{eqnarray}
\end{proposition}
\begin{proof}
Let us fix $v \in \{1,\ldots,r\}$. Using \eqref{cumulants}, \eqref{Y} and  \eqref{norme}, one can easily deduce from \eqref{Oderacine} (respectively from \eqref{Oden}) that  all  the terms  in $\left\{Y_n(\lambda)\Psi(\lambda)F_n(\lambda)\right\}_{pq}$ corresponding to the $M^{(3)}(v,i,l,j)$'s  in \eqref{psi} excluding $$(\kappa_3^{i,l,v}-\tilde \kappa_3^{i,l,v}\sqrt{-1})\alpha_v (R_n(\lambda))_{ii}\alpha_v (R_n(\lambda))_{ll} \alpha_v  (R_n(\lambda))_{ij} $$  (respectively all  the terms corresponding to the $M^{(4)}(v,i,l,j)$'s)  are equal to $O(1/n)$.



Let  $C$ be   some constant such that $\sup_{i,l,v} \{\vert \kappa_5^{i,l,v}\vert +\vert \tilde \kappa_5^{i,l,v} \vert\}\leq C$.
For the terms in $\left\{Y_n(\lambda)\Psi(\lambda)F_n(\lambda)\right\}_{pq}$ corresponding to $M^{(5)}(v,i,l,j)$'s, note that using \eqref{norme} they can be all obviously bounded by 
\\

\noindent $
\frac{C\Vert F_n(\lambda) \Vert \Vert \alpha_v\Vert^5  \Vert (Im(\lambda))^{-1}\Vert^5}{n\sqrt{n}}\sum_{l=1}^n \left\| (Y_n(\lambda))_{pl}\right\|$
\begin{eqnarray*}&\leq &
\frac{C\Vert F_n (\lambda)\Vert \Vert \alpha_v\Vert^5  \Vert (Im(\lambda))^{-1}\Vert^5}{n}\left\{\sum_{l=1}^n \left\| (Y_n(\lambda))_{pl}\right\|^2\right\}^{\frac{1}{2}}\\&\leq & \frac{\sqrt{m}C\Vert F_n(\lambda) \Vert \Vert \alpha_v\Vert^5  \Vert (Im(\lambda))^{-1}\Vert^6}{n}\\&=&O(1/n)\end{eqnarray*}
where we used \eqref{l} and \eqref{Y}.\\

\noindent Finally define
$$\hat {\cal R}= \left[ \alpha_v [H_n(\lambda) -\mathbb{E}(H_n(\lambda))] \alpha_v\right] \otimes I_n $$
so that there exists some constant $C>0$ such that 
\\

\noindent $\sum_{j,l=1}^n (Y_n(\lambda))_{pl} \mathbb{E}\left[ \alpha_v [H_n(\lambda) -\mathbb{E}(H_n(\lambda))] \alpha_v    [R_n(\lambda ) ]_{lj}(F_n(\lambda)_{jq} \right]
$\begin{eqnarray*}&=& \mathbb{E}\left[  [Y_n (\lambda)\hat {\cal R} R_n(\lambda ) F_n(\lambda)]_{pq} \right]\\
&\leq&\Vert F_n(\lambda) \Vert  \Vert (Im(\lambda))^{-1}\Vert^2 \mathbb{E}\left[  \Vert \hat  {\cal R} \Vert \right]\\
&
\leq&\frac{C \sqrt{m} \Vert F_n (\lambda)\Vert \Vert \alpha_v\Vert^2}{n}
 \Vert (Im(\lambda))^{-1}\Vert^4\\&=&O(1/n).\end{eqnarray*}
 where we used \eqref{Y}, \eqref{norme} and \eqref{varhn} in the last lines.\\
\noindent It is moreover clear that one can find a common polynomial to bound  the involved $nO_{p,q}(1/n)$.  \eqref{YPSIF} follows.
\end{proof}

\begin{corollaire} \label{estimenunsurn}
For any  $mn\times mn$ deterministic  matrix  $F_n(\lambda)$ such that $F_n(\lambda)=O(1)$, we have  \\

\noindent 
$\mathbb{E}\left[  [ R_n(\lambda ) F_n(\lambda)]_{pq} \right]$ \begin{eqnarray*}&=& \left(Y_n(\lambda)F_n(\lambda)\right)_{pq} 
\\&&+\ \sum_{v=1}^r \sum_{i,l=1}^n \frac{(\kappa_3^{i,l,v} -\tilde \kappa_3^{i,l,v}\sqrt{-1})}{2\sqrt{2}n\sqrt{n}} (Y_n(\lambda))_{pl}  \alpha_v(Y_n(\lambda))_{ii}\alpha_v (Y_n(\lambda))_{ll}\alpha_v\\~&&~~~~~~~~~~~~~~~~~~~~~~~~~~~~~~~~~~~~~~~~~~~~~~~~\times \mathbb{E}\left[ (R_n(\lambda)F_n(\lambda))_{iq} \right]\\&& +O^{(u)}_{p,q}(1/n).\end{eqnarray*}
\end{corollaire}
\begin{proof} 
Noticing that

\noindent $\left\|
\sum_{l,j=1}^n \left(Y_n(\lambda)\right)_{pl} O^{(u)}_{l,j}\left(1/n^2\right) \left(F_n(\lambda)\right)_{jq} \right\|$
 \begin{eqnarray*}&\leq & 
 O\left(\frac{1}{n}\right)\left\{ \sum_{l=1}^n \Vert \left(Y_n(\lambda)\right)_{pl} \Vert^2 \right\}^{1/2} \left\{ \sum_{j=1}^n \Vert \left(F_n(\lambda)\right)_{jq} \Vert^2 \right\}^{1/2}\\
 &=&O_{p,q}^{(u)}(1/n) ,
\end{eqnarray*}
(using Lemma \ref{majcarre} and \eqref{Y} in the last line)
it readily follows from Theorem \ref{resolvante} and Proposition \ref{55} that \\

\noindent $\mathbb{E}\left[  [ R_n(\lambda ) F_n(\lambda)]_{pq} \right]$ \begin{eqnarray*}&\hspace*{-0.5cm}=& \left(Y_n(\lambda)F_n(\lambda)\right)_{pq} 
\\&&+ \sum_{v=1}^r  \sum_{i,l=1}^n \frac{(\kappa_3^{i,l,v} -\tilde \kappa_3^{i,l,v}\sqrt{-1})}{2\sqrt{2}n\sqrt{n}} (Y_n(\lambda))_{pl}    \alpha_v\\&&~~~~~~~~~~~~~~~~~~~~~~~~~~~\times \mathbb{E}\left[(R_n(\lambda))_{ii}\alpha_v(R_n(\lambda))_{ll}\alpha_v 
 (R_n(\lambda)F_n(\lambda))_{iq} \right]\\&& +O^{(u)}_{p,q}(1/n).\end{eqnarray*}

To simplify the writing let us set $U_i=\alpha_v (R_n(\lambda))_{ii}$, $V_l=\alpha_v(R_n(\lambda))_{ll}$ and $W_i= \alpha_v 
 (R_n(\lambda)F_n(\lambda))_{iq}$.
We have \\

$\mathbb{E}\left[U_i V_l W_i\right]$
\begin{eqnarray}&=&\mathbb{E}\left[ (U_i-\mathbb{E}(U_i)) (V_l-\mathbb{E}(V_l))W_i\right]  \nonumber\\&&+ \mathbb{E}\left[ (U_i-\mathbb{E}(U_i))\mathbb{E}(V_l) (W_i-\mathbb{E}(W_i))\right] \nonumber
\\&&+ \mathbb{E}(U_i) \mathbb{E}\left[ (V_l-\mathbb{E}(V_l))(W_i-\mathbb{E}(W_i))\right]+\mathbb{E}\left[U_i\right]
\mathbb{E}\left[V_l\right]\mathbb{E}\left[W_i\right]. \label{decomposition}
\end{eqnarray}

\noindent Now,\\

\noindent $\left\|\sum_{i,l=1}^n (\kappa_3^{i,l,v} -\tilde \kappa_3^{i,l,v}\sqrt{-1})(Y_n(\lambda))_{pl}    \mathbb{E}\left[ (U_i-\mathbb{E}(U_i)) (V_l-\mathbb{E}(V_l))W_i\right]\right\|$ 

\begin{eqnarray*}&\leq& C\Vert \alpha_v \Vert \Vert ( \Im \lambda)^{-1} \Vert \Vert F_n(\lambda)\Vert
\\&&~~~~~~\times  \sum_{i,l=1}^n \left\| 
(Y_n(\lambda))_{pl} \right\|\left\{
\mathbb{E} \left(  \left\| U_{i} -\mathbb{E}\left( U_i\right)\right\|^2\right)\right\}^{1/2}
\left\{\mathbb{E} \left(  \left\| V_l -\mathbb{E}\left( V_l\right)\right\|^2\right)\right\}^{1/2}\\
&\leq &\sqrt{n} C\Vert \alpha_v \Vert \Vert {( \Im \lambda)}^{-1} \Vert \Vert F_n(\lambda)\Vert \left\{\sum_{l=1}^n \left\| (Y_n(\lambda))_{pl} \right\|^2\right\}^{1/2}
\\&&~~~~~~~~~~~~~~~~~\times  
\left\{\sum_{l=1}^n \mathbb{E} \left(  \left\|V_l -\mathbb{E}\left( V_l\right)\right\|^2\right)\right\}^{1/2}\left\{\sum_{i=1}^n \mathbb{E} \left(  \left\|U_i -\mathbb{E}\left( U_i\right)\right\|^2\right)\right\}^{1/2}.\end{eqnarray*}

\noindent Moreover,
$$\left\|\sum_{i,l=1}^n (\kappa_3^{i,l,v} -\tilde \kappa_3^{i,l,v}\sqrt{-1})(Y_n(\lambda))_{pl}    \mathbb{E}\left[ (U_i-\mathbb{E}(U_i))\mathbb{E}(V_l) (W_i-\mathbb{E}(W_i))\right]\right\|$$
\begin{eqnarray*}&\leq& C\Vert \alpha_v \Vert \Vert ( \Im \lambda)^{-1} \Vert  \sum_{i,l=1}^n \left\| (Y_n(\lambda))_{pl} \right\|\\&&~~~~~~~~~~~~~~~~~\times  \left\{
\mathbb{E} \left(  \left\| U_{i} -\mathbb{E}\left( U_i\right)\right\|^2\right)\right\}^{1/2}
\left\{\mathbb{E} \left(  \left\| W_i -\mathbb{E}\left( W_i\right)\right\|^2\right)\right\}^{1/2}\\
&\leq &\sqrt{n} C\Vert \alpha_v \Vert \Vert {( \Im \lambda)}^{-1} \Vert 
 \left\{\sum_{l=1}^n \left\| (Y_n(\lambda))_{pl} \right\|^2\right\}^{1/2}
\\&& \times \left\{\sum_{l=1}^n \mathbb{E} \left(  \left\|W_i -\mathbb{E}\left( W_i\right)\right\|^2\right)\right\}^{1/2} \left\{\sum_{i=1}^n \mathbb{E} \left(  \left\|U_i -\mathbb{E}\left( U_i\right)\right\|^2\right)\right\}^{1/2}.\end{eqnarray*}
Finally
$$\left\|\sum_{i,l=1}^n (\kappa_3^{i,l,v} -\tilde \kappa_3^{i,l,v}\sqrt{-1})(Y_n(\lambda))_{pl}    \mathbb{E}\left[ U_i\right]\mathbb{E}\left[ (V_l-\mathbb{E}(V_l)) (W_i-\mathbb{E}(W_i))\right]\right\|$$
\begin{eqnarray*}&\leq& C\Vert \alpha_v \Vert \Vert ( \Im \lambda)^{-1} \Vert  \sum_{i,l=1}^n \left\| (Y_n(\lambda))_{pl} \right\|
\\&& \times \left\{
\mathbb{E} \left(  \left\| V_{l} -\mathbb{E}\left( V_l\right)\right\|^2\right)\right\}^{1/2}
\left\{\mathbb{E} \left(  \left\| W_i -\mathbb{E}\left( W_i\right)\right\|^2\right)\right\}^{1/2}\\
&\leq &\sqrt{n} C\Vert \alpha_v \Vert \Vert {( \Im \lambda)}^{-1} \Vert 
 \left\{\sum_{l=1}^n \left\| (Y_n(\lambda))_{pl} \right\|^2\right\}^{1/2}
\\&& \times \left\{\sum_{l=1}^n \mathbb{E} \left(  \left\|V_l -\mathbb{E}\left( V_l\right)\right\|^2\right)\right\}^{1/2}\left\{\sum_{i=1}^n \mathbb{E} \left(  \left\|W_i -\mathbb{E}\left( W_i\right)\right\|^2\right)\right\}^{1/2}.\end{eqnarray*}

 \noindent Using Lemma \ref{var},   (\ref{l}) and \eqref{Y}, we readily deduce that \\

\noindent $\mathbb{E}\left[  [ R_n(\lambda ) F_n(\lambda)]_{pq} \right]$ \begin{eqnarray}&=& \left(Y_n(\lambda)F_n(\lambda)\right)_{pq} \nonumber
\\&&+\frac{1}{2\sqrt{2}n\sqrt{n}}\sum_{v=1}^r  \sum_{i,l=1}^n  (\kappa_3^{i,l,v} -\tilde \kappa_3^{i,l,v}\sqrt{-1})(Y_n(\lambda)_{pl}  \alpha_v\\&&~~~~~~~~~~~~~~~~~~~~~~~\times\mathbb{E}\left[(R_n(\lambda))_{ii}\right]
\alpha_v\mathbb{E}\left[(R_n(\lambda))_{ll}\right]\alpha_v \mathbb{E} \left[ (R_n(\lambda)F_n(\lambda))_{iq} \right] \nonumber \\&& +O^{(u)}_{p,q}(1/{n}). \label{eq}\end{eqnarray}
\noindent Now, define 
$${\cal R}= \sum_{v=1}^r\sum_{i,l=1}^n  (\kappa_3^{i,l,v} -\tilde \kappa_3^{i,l,v}\sqrt{-1})    \alpha_v\mathbb{E}\left[(R_n(\lambda))_{ii}\right]\alpha_v\mathbb{E}
\left[(R_n(\lambda))_{ll}\right]\alpha_v\otimes E_{li},$$
It is easy to see that  $\Vert {\cal R}\Vert\leq C \Vert (\Im \lambda)^{-1}\Vert^2 n.$
We have \\

\noindent $ \sum_{v=1}^r \sum_{i,l=1}^n(\kappa_3^{i,l,v} -\tilde \kappa_3^{i,l,v}\sqrt{-1})  (Y_n(\lambda))_{pl}   \alpha_v\mathbb{E}\left[(R_n(\lambda))_{ii}\right]  $ $$ ~~~~~~~~~~\times \alpha_v\mathbb{E}\left[(R_n(\lambda))_{ll}\right]\alpha_v\mathbb{E} \left[ (R_n(\lambda)F_n(\lambda))_{iq} \right]= \left[ Y_n(\lambda) {\cal R}  \mathbb{E} (R_n(\lambda) F_n(\lambda))\right]_{pq}.$$
So that if we define $$T=\sum_{p,q=1}^n T_{pq}\otimes E_{pq}$$ where 
$$T_{pq}
= \frac{1}{2\sqrt{2}n\sqrt{n}}\sum_{v=1}^r \sum_{l,i=1}^n (\kappa_3^{i,l,v} -\tilde \kappa_3^{i,l,v}\sqrt{-1})    (Y_n(\lambda))_{pl}    \alpha_v\mathbb{E}\left[(R_n(\lambda))_{ii}\right]$$ $$~~~~~~~~~~~~~~~~~~~~~~~~~~~~\times\alpha_v\mathbb{E}\left[(R_n(\lambda))_{ll}\right]
\alpha_v\mathbb{E} \left[ (R_n(\lambda)F_n(\lambda))_{iq} \right],$$
we have \begin{equation}\label{Q}\Vert T\Vert=\left\| \frac{1}{2\sqrt{2}n\sqrt{n}} Y_n(\lambda) {\cal R}  \mathbb{E} (R_n(\lambda) F_n(\lambda)) \right\|=O(1/\sqrt{n}).\end{equation}
Hence, in particular we have \begin{equation}\label{estimenracinen} \mathbb{E}\left[  [ R_n(\lambda ) F_n]_{pq} \right]= \left(Y_n(\lambda)F_n(\lambda)\right)_{pq} +O^{(u)}_{p,q}(1/\sqrt{n}). \end{equation}

\noindent Now, 
we have 
$$\left\| \sum_{i,l=1}^n \frac{ (\kappa_3^{i,l,v} -\tilde \kappa_3^{i,l,v}\sqrt{-1})}{2\sqrt{2}n\sqrt{n}}  (Y_n(\lambda))_{pl}  \alpha_vO^{(u)}_{i,i}(\frac{1}{\sqrt{n}})\alpha_v(Y_n(\lambda))_{ll}\alpha_v\mathbb{E} \left[ (R_n(\lambda)F_n(\lambda))_{iq} \right] \right\|$$
\begin{eqnarray}& \leq& \frac{C \Vert(\Im \lambda)^{-1} \Vert }{n} \left(  \sum_{l=1}^n \Vert (Y_n(\lambda))_{pl} \Vert^2 \right)^{1/2} \left\{\sum_{i=1}^n 
\left\| \mathbb{E} \left[(R_n(\lambda)F_n(\lambda))_{iq}\right] \right\|^2 \right\}^{1/2}\nonumber \\&&~~~~~~~~~~~~~~~~~~~~~~~~~\times  \left\{\sum_{i=1}^n \Vert O^{(u)}_{i,i}(1/{\sqrt{n}})\Vert^2\right\}^{1/2} \nonumber \\ & =& O^{(u)}_{p,q}(1/n) \label{eq2} 
\end{eqnarray}
where we used (\ref{l}) twice and \eqref{Y} and \eqref{norme} in the last line.
Similarly $$
\left\|  \sum_{i,l=1}^n \frac{(\kappa_3^{i,l,v} -\tilde \kappa_3^{i,l,v}\sqrt{-1})}{2\sqrt{2}n\sqrt{n}} (Y_n(\lambda))_{pl}  \alpha_v(Y_n(\lambda))_{ii}\alpha_vO_{l,l}(\frac{1}{\sqrt{n}})\alpha_v\mathbb{E} \left[ (R_n(\lambda)F_n(\lambda))_{iq} \right] \right\|$$ \begin{equation}\label{eq3}=O^{(u)}_{p,q}(\frac{1}{{n} }), \end{equation} 
$$
\left\|  \sum_{i,l=1}^n \frac{ (\kappa_3^{i,l,v} -\tilde \kappa_3^{i,l,v}\sqrt{-1})}{2\sqrt{2}n\sqrt{n}}(Y_n(\lambda))_{pl}  \alpha_vO_{i,i}(\frac{1}{\sqrt{n}})\alpha_vO_{l,l}(1/\sqrt{n})\alpha_v\mathbb{E} \left[ (R_n(\lambda)F_n(\lambda))_{iq} \right] \right\|$$ \begin{equation}\label{eq4}=O^{(u)}_{p,q}(\frac{1}{{n \sqrt{n}} }) \end{equation}

(\ref{eq}), (\ref{estimenracinen}) and (\ref{eq2}), (\ref{eq3}), \eqref{eq4} readily yields Corollary \ref{estimenunsurn}.
\end{proof}
\begin{corollaire} \label{ME} With the notations of Section \ref{Notations},
 
 \begin{eqnarray*}G_n(\lambda) &=&\mathbb{E} \left({\rm id}_m\otimes tr_n  R_n(\lambda)\right)\\
&=&G_{\sum_{u=1}^t\beta_u \otimes a_n^{(u)}}\left(\lambda -\gamma -\sum_{v=1}^r \alpha_v G_n(\lambda)\alpha_v \right)
\\&&+L_n(\lambda)
+\epsilon_n(\lambda)
 \end{eqnarray*}
 where 
$$L_n(\lambda)=\frac{1}{n}\sum_{p=1}^n  \left[Y_n(\lambda)
\Psi(\lambda) \right]_{pp},$$
(with $\Psi(\lambda)$ defined in Theorem \ref{resolvante} and $Y_n(\lambda)$ defined in Lemma \ref{inverseY})
 and $$ \epsilon_n(\lambda) =O \left( \frac{1}{n\sqrt{n}}\right).$$
 Moreover \begin{equation} \label{L}   L_n(\lambda) =O \left( \frac{1}{\sqrt{n}}\right).\end{equation}

 \end{corollaire}
\begin{proof}
First note that, since  the distribution of  $a_n=(a_n^{(1)},\ldots,a_n^{(t)})$  in $({\cal A},\tau)$  coincides with the distribution of $(A_n^{(1)},\ldots, A_n^{(t)})$ in $({ M}_n(\mathbb{C}),\tr_n)$, we have \begin{eqnarray*}{\rm id}_m\otimes tr_n Y_n(\lambda)&= &{\rm id}_m\otimes \tau \left((\lambda -\gamma -\sum_{v=1}^r \alpha_v G_n(\lambda)\alpha_v)\otimes 1_{\cal A}-\sum_{u=1}^t\beta_u \otimes a_n^{(u)} \right)^{-1}
\\&=& G_{\sum_{u=1}^t\beta_u \otimes a_n^{(u)}}\left(\lambda -\gamma -\sum_{v=1}^r \alpha_v G_n(\lambda)\alpha_v \right).\end{eqnarray*}
Then, the corollary readily follows from Theorem \ref{resolvante} and Proposition \ref{55} by noting that 
\begin{itemize}
\item 

\noindent  
$\frac{1}{n}\Vert \sum_{p,l=1}^n (Y_n(\lambda))_{pl}O_{lp}^{(u)}(1/n^2)\Vert$  \begin{eqnarray*} & \leq &\frac{1}{n}\left(\sum_{p,l=1}^n \Vert (Y_n(\lambda))_{pl}\Vert ^2 \right)^{1/2}\left( \sum_{p,l=1}^n \Vert O^{(u)}_{lp}(1/n^2)\Vert ^2 \right)^{1/2}\\&=&O(\frac{1}{n\sqrt{n}})\end{eqnarray*}
where we used Lemma \ref{majcarre} and \eqref{Y} in the last line.
\item $$ \Vert \sum_{i,l,p=1}^n \frac{(\kappa_3^{i,l,v} -\tilde \kappa_3^{i,l,v}\sqrt{-1})}{n^2\sqrt{n}}(Y_n(\lambda))_{pl}    \mathbb{E} \{\alpha_v(R_n(\lambda))_{ii}\alpha_v(R_n(\lambda))_{ll}\alpha_v 
 (R_n(\lambda))_{ip} \Vert $$
$$\leq   \frac{C\Vert \alpha_v \Vert^3 \Vert (\Im  \lambda) ^{-1}\Vert^2}{n^2} \sum_{l,p=1}^n \Vert (Y_n(\lambda))_{pl}   \Vert \left(\sum_{i=1}^n  \Vert(R_n(\lambda))_{ip} \Vert^2\right)^{1/2}$$
$$ \leq   \frac{ C\sqrt{m}\Vert \alpha_v \Vert^3 \Vert (\Im \lambda )^{-1}\Vert^3}{n}  \left( \sum_{p,l=1}^n \Vert (Y_n(\lambda))_{pl}\Vert ^2 \right)^{1/2} =O(1/\sqrt{n})$$
where we used Lemma \ref{majcarre}, \eqref{norme} and \eqref{Y} in the two last lines.
\end{itemize}
\end{proof}


\begin{theoreme}\label{difftilde}
Let $\lambda$ be in $M_m(\mathbb{C})$ such that 
 $\Im  \lambda>0$, and $\tilde G_n(\lambda)$ as defined in \eqref{defGntilde}. We have 
\begin{equation} \label{prediff}
G_n(\lambda)-\tilde G_n(\lambda)+{E_n(\lambda)}= O(\frac{1}{n\sqrt{n}}),
\end{equation}
where $E_n(\lambda)$ is given by\\

\noindent $E_n(\lambda) =$
\begin{equation}
 \sum_{v=1}^r \tilde G_n'(\lambda) \cdot \alpha_v L_n(\lambda) \alpha_v -\frac{1}{2} \tilde G_n''(\lambda) \cdot\left( \sum_{v=1}^r \alpha_v L_n(\lambda) \alpha_v,  \sum_{v=1}^r \alpha_v L_n(\lambda) \alpha_v\right) -L_n(\lambda) 
\end{equation}
with $L_n(\lambda)$ defined in Corollary \ref{ME}.
\end{theoreme}
\begin{proof}
Let $\lambda$ be in $M_m(\mathbb{C})$ such that 
 $\Im  \lambda>0$. Note that according to \eqref{image}, we have $\Im \left( \lambda-\gamma - \sum_{v=1}^r \alpha_v G_n(\lambda)\alpha_v\right)>0.$
Define  (using the notations of Section \ref{Notations})
\begin{equation}\label{lambdan}\Lambda_n(\lambda)= \gamma +\lambda + \sum_{v=1}^r \alpha_v G_{\sum_{u=1}^t \beta_u \otimes a_n^{(u)}}(\lambda) \alpha_v\end{equation}
and  $\lambda'=\Lambda_n(\lambda-\gamma - \sum_{v=1}^r \alpha_v G_n(\lambda)\alpha_v).$
Using Corollary \ref{ME}, we have 
\begin{eqnarray}\lambda'-\lambda &= & - \sum_{v=1}^r \alpha_v G_n(\lambda) \alpha_v + \sum_{v=1}^r \alpha_v G_{\sum_{u=1}^t \beta_u \otimes a_n^{(u)}}(\lambda -\gamma -\sum_{v=1}^r \alpha_v G_n(\lambda ) \alpha_v) \alpha_v \nonumber \\&=& -\sum_{v=1}^r \alpha_v L_n(\lambda) \alpha_v +O(\frac{1}{n\sqrt{n}})\label{lambdaprimemoinslambdaavant}\\&=& O(1/\sqrt{n}). \label{lambdaprimemoinslambda}
\end{eqnarray}

\noindent Thus there exists  a polynomial $Q$ with nonnegative coefficients 
such that $$\left\|\lambda'-\lambda\right\|\leq \frac{Q(\Vert (\Im \lambda )^{-1}\Vert)}{\sqrt{n}}.$$
-On the one hand, if $$\frac{Q(\Vert (\Im  \lambda )^{-1}\Vert)}{\sqrt{n}}\geq \frac{1}{2\Vert (\Im \lambda )^{-1}\Vert},$$ 
or equivalently 
\begin{equation} \label{1=O(1/n)}
1\leq \frac{2\Vert (\Im \lambda )^{-1}\Vert Q(\Vert (\Im \lambda )^{-1}\Vert)}{\sqrt{n}},
\end{equation}
to prove \eqref{prediff}
it is enough to prove that 
\begin{equation} \label{O(1)}
G_n(\lambda)-\tilde G_n(\lambda)+E_n(\lambda) = O(1).
\end{equation}
Indeed, if we assume that \eqref{1=O(1/n)} and \eqref{O(1)} hold, 
then there exists a polynomial $\tilde Q$ with nonnegative coefficients 
such that 
\begin{eqnarray*}
\left\|G_n(\lambda)-\tilde G_n(\lambda)+E_n(\lambda) \right\|&\leq &\tilde Q(\Vert (\Im  \lambda )^{-1}\Vert)\\
&\leq &\tilde Q(\Vert ( \Im \lambda )^{-1}\Vert)\frac{2\Vert ( \Im \lambda )^{-1}\Vert Q(\Vert( \Im \lambda )^{-1}\Vert)}{\sqrt{n}}\\
&\leq &\tilde Q(\Vert (\Im  \lambda )^{-1}\Vert)(\frac{2\Vert (\Im \lambda) ^{-1}\Vert Q(\Vert (\Im \lambda )^{-1}\Vert)}{\sqrt{n}})^4.
\end{eqnarray*}
Hence, $$G_n(\lambda)-\tilde G_n(\lambda)+E_n(\lambda) = O(\frac{1}{n^2})$$
so that \eqref{prediff} holds.
To prove \eqref{O(1)}, one can notice that, using \eqref{norme} and \eqref{normeG},
 both $G_n(\lambda)$ and $\tilde G_n(\lambda)$ 
are bounded by $\Vert (\Im  \lambda) ^{-1}\Vert$, and that \\

\noindent 
$\left\|E_n(\lambda)\right\|$ $$\leq \left\{r \max_{v=1}^r\Vert \alpha_v\Vert^2 \Vert (\Im  \lambda )^{-1}\Vert^2 +1\right\} \left\|L_n(\lambda)\right\|+r^2 \max_{v=1}^r\Vert \alpha_v\Vert^4 \Vert (\Im \lambda )^{-1}\Vert^3 \left\|L_n(\lambda)\right\|^2 ,$$
where $L_n(\lambda)=O(1/\sqrt{n})$ according to \eqref{L}.\\
~~

\noindent -On the other hand, if $$\frac{Q(\Vert (\Im \lambda )^{-1}\Vert)}{\sqrt{n}}\leq \frac{1}{2\Vert (\Im \lambda )^{-1}\Vert},$$ 
one has : 
\begin{equation}\label{lambdaprime}\left\|\Im \lambda'-\Im \lambda\right\|\leq \left\|\lambda'-\lambda \right\|\leq \frac{1}{2\Vert ( \Im \lambda) ^{-1}\Vert}\end{equation}
Denoting for any Hermitian matrix $H$ by $l_1(H)$ the smallest eigenvalue of $H$, we readily deduce from \eqref{lambdaprime} that 
 $l_1(\Im \lambda')\geq \frac{l_1(\Im \lambda)}{2}$ and 
therefore \begin{equation}\label{lambdaprimepositif}\Im \lambda' >0. \end{equation}
Then, it makes sense to consider $\tilde G_{n}(\lambda')$ which satisfies according to \eqref{subor}
\begin{eqnarray}\tilde G_{n}(\lambda')&= &G_{\sum_{u=1}^t \beta_u \otimes a_n^{(u)}}(\lambda'-\gamma -\sum_{v=1}^r\alpha_v \tilde  G_n(\lambda')\alpha_v) \nonumber\\&=&
G_{\sum_{u=1}^t \beta_u \otimes a_n^{(u)}}\left(\omega_n(\lambda')\right) \nonumber\\& =&G_{\sum_{u=1}^t \beta_u \otimes a_n^{(u)}}\left(\omega_n(\Lambda_n(\lambda -\gamma -\sum_{v=1}^r \alpha_v G_n(\lambda)\alpha_v ))\right). \label{subordin}\end{eqnarray}


Applying Lemma \ref{inversion} to $\rho=\lambda -\gamma -\sum_{v=1}^r\alpha_v G_n(\lambda)\alpha_v $ ( using \eqref{image} and \eqref{lambdaprimepositif}) we obtain that 
 since $\Im \lambda'=\Im  \Lambda_n(\rho)>0$ we have $\omega_n(\Lambda_n(\rho))=\rho$ and according to \eqref{subordin}
$$ \tilde G_{n}\left(\lambda'\right)
=G_{\sum_{u=1}^t \beta_u \otimes a_n^{(u)}}\left(\lambda -\gamma -\sum_{v=1}^r \alpha_v G_n(\lambda)\alpha_v \right). $$
Hence, by Corollary \ref{ME}, we have
\begin{equation} \label{termone}
G_n(\lambda)-\tilde G_n(\lambda')-{L_n(\lambda)}=O(\frac{1}{n\sqrt{n}}).
\end{equation}
Now, we have \\

\noindent $\tilde G_n(\lambda')-\tilde G_n(\lambda)$
\begin{eqnarray*}
&=& {\rm id}_m \otimes \tau \left\{ r_n(\lambda')\left[ (\lambda-\lambda') \otimes 1_{\cal A}\right] r_n(\lambda)\right\}\\
&=& {\rm id}_m\otimes \tau \left\{ r_n(\lambda)\left[(\lambda-\lambda') \otimes 1_{\cal A}\right] r_n(\lambda)\right\} \\&&+ 
{\rm id}_m \otimes \tau \left\{ r_n(\lambda)\left[ (\lambda-\lambda') \otimes1_{\cal A}\right] r_n(\lambda) \left[ (\lambda-\lambda') \otimes 1_{\cal A}\right] r_n(\lambda) \right\}
\\&&+ 
{\rm id}_m \otimes \tau \left\{ r_n(\lambda')\left[ (\lambda-\lambda') \otimes1_{\cal A}\right] r_n(\lambda)\left[ (\lambda-\lambda') \otimes1_{\cal A}\right] r_n(\lambda)  \left[(\lambda-\lambda') \otimes 1_{\cal A}\right] r_n(\lambda) \right\}\\
&=&\sum_{v=1}^r {\rm id}_m \otimes \tau \left\{ r_n(\lambda)\left[\alpha_v L_n(\lambda)\alpha_v \otimes 1_{\cal A}\right] r_n(\lambda)\right\} \\&&+\sum_{v,v'=1}^r {\rm id}_m \otimes \tau \left\{ r_n(\lambda)\left[\alpha_v L_n(\lambda)\alpha_v \otimes 1_{\cal A}\right] r_n(\lambda)\left[\alpha_{v'} L_n(\lambda)\alpha_{v'} \otimes 1_{\cal A}\right] r_n(\lambda)\right\}
\\&& 
+O(\frac{1}{n\sqrt{n}})
\end{eqnarray*}
where we used \eqref{lambdaprimemoinslambdaavant}, \eqref{lambdaprimemoinslambda} and \eqref{normeG} in the last line.
Hence we have 
$$
\tilde G_n(\lambda')-\tilde G_n(\lambda)+\sum_{v=1}^r\tilde G_n'(\lambda) \cdot \alpha_v L_n(\lambda) \alpha_v
-\frac{1}{2}\tilde G_n''(\lambda) \cdot\left(\sum_{v=1}^r \alpha_v L_n(\lambda) \alpha_v, \sum_{v=1}^r \alpha_v L_n(\lambda) \alpha_v\right)$$
\begin{equation} \label{termtwo}~~~~~~~~~~~~~~~~~~~~~~~~~~=O(\frac{1}{n\sqrt{n}}).
\end{equation}
(\ref{prediff}) follows from \eqref{termone} and \eqref{termtwo} since \\

\noindent 
$\left\|G_n(\lambda)-\tilde G_n(\lambda)+{E_n(\lambda)}\right\|\leq 
\left\|G_n(\lambda)-\tilde G_n(\lambda')-{L_n(\lambda)}\right\|$
$$+\left\|\tilde G_n(\lambda')-\tilde G_n(\lambda)+\sum_{p=1}^r\tilde G_n'(\lambda)\cdot \alpha_v{L_n(z)}\alpha_v -\frac{1}{2}\tilde G_n''(\lambda) \cdot\left(\sum_{p=1}^r \alpha_v L_n(\lambda) \alpha_v, \sum_{p=1}^r \alpha_v L_n(\lambda) \alpha_v\right)\right\|.
$$

\end{proof}
\begin{remarque}
\eqref{L} and \eqref{normeG} readily yield that $E_n(\lambda)=O(\frac{1}{\sqrt{n}})$. Thus,  we can deduce from (\ref{prediff}) that \begin{equation}\label{difgngntilde}G_n( \lambda)-\tilde G_n(\lambda) =O(\frac{1}{\sqrt{n}}).\end{equation}
\end{remarque}

\begin{proposition} \label{estimdiff}
 For $z \in \C  \setminus \R$, let $g_n(z)$ and $\tilde g_n(z)$ as defined in \eqref{defpetitg} and \eqref{defpetitgtilde} respectively. We have 
\begin{equation} \label{estimdiffeqn}
g_n(z)-\tilde g_n(z)+{\tilde{E}_n(z)} = O(\frac{1}{n\sqrt{n}}),
\end{equation}
where $\tilde{E}_n(z)$ is given by
\begin{equation}\label{defentilde}
\tilde{E}_n(z) = \sum_{v=1}^r tr_m\left( \tilde G_n'(zI_m)\cdot \alpha_v \tilde{L}_n(z)\alpha_v\right)  - \tr_m \tilde L_n(z)
\end{equation}
{with }\\

$\displaystyle{\tilde{L}_n(z)=\sum_{v=1}^r}$ \begin{eqnarray*}  && \bigg\{
\sum_{i,l,p=1}^n  \frac{(\kappa_4^{i,l,v} +\tilde \kappa_4^{i,l,v})}{4n^3}  \tilde Y_n(zI_m)_{pl} \alpha_v (\tilde Y_n(zI_m))_{ii}   \alpha_v (\tilde Y_n(zI_m))_{ll} \alpha_v (\tilde Y_n(zI_m))_{ii} \alpha_v \left(\tilde  Y_n(zI_m)\right)_{lp}\\
&&+ \sum_{i,l,p=1}^n \frac{(\kappa_3^{i,l,v} +\tilde \kappa_3^{i,l,v}\sqrt{-1}) }{2\sqrt{2} n^2 \sqrt{n} } \tilde Y_n(zI_m)_{pl} \alpha_v  (\tilde Y_n(zI_m))_{ii} \alpha_v  (\tilde Y_n(zI_m))_{li} \alpha_v  \left( \tilde Y_n(zI_m)\right)_{lp} \\
&&+\sum_{i,l,p=1}^n \frac{ (\kappa_3^{i,l,v} -\tilde \kappa_3^{i,l,v}\sqrt{-1}) }{2\sqrt{2} n^2 \sqrt{n} }  \tilde Y_n(zI_m)_{pl} \alpha_v  (\tilde Y_n(zI_m))_{ii} \alpha_v  (\tilde Y_n(zI_m))_{ll} \alpha_v  \left(\tilde Y_n(zI_m)\right)_{ip}\\
&&+\sum_{i,l,p=1}^n  \frac{(\kappa_3^{i,l,v} -\tilde \kappa_3^{i,l,v}\sqrt{-1})}{2\sqrt{2} n^2 \sqrt{n} }   \tilde Y_n(zI_m)_{pl} \alpha_v  (\tilde Y_n(zI_m))_{il}  \alpha_v  (\tilde Y_n(zI_m))_{ii}  \alpha_v  \left (\tilde Y_n(zI_m)\right)_{lp}\bigg\}
\end{eqnarray*}
where  $\tilde Y_n$ and $\tilde G_n$ were defined in \eqref{ytilde} and \eqref{defGntilde} respectively  so that 
 \begin{eqnarray*}\tilde Y_n(zI_m)&= &\left((zI_m -\gamma -\sum_{v=1}^r\alpha_v \tilde  G_n(zI_m)\alpha_v)\otimes I_n -  \sum_{u=1}^t \beta_u \otimes A_n^{(u)} \right)^{-1}\\&=&\left(\omega_n(zI_m)\otimes I_n -  \sum_{u=1}^t \beta_u \otimes A_n^{(u)} \right)^{-1}\end{eqnarray*}
and  \begin{eqnarray*}\tilde G_n(zI_m)&=&{\rm id}_m\otimes \tau\left( (zI_m-\gamma) \otimes 1_{\cal A} - \sum_{v=1}^r \alpha_v \otimes x_v-\sum_{u=1}^t \beta_u \otimes a_n^{(u)}\right)^{-1}\\&=&{\rm id}_m \otimes \tau\left( zI_m\otimes  1_{\cal A} - s_n\right)^{-1} .\end{eqnarray*}
\end{proposition}
\begin{proof}
Let $z\in \mathbb{C}\setminus \mathbb{R}$ such that $\Im  z >0$.
Theorem \ref{difftilde} yields $$ g_n(z)- \tilde g_n(z)+\sum_{v=1}^r tr_m\tilde G_n'(zI_m) \cdot \alpha_v L_n(zI_m) \alpha_v -tr_mL_n(zI_m)
$$ \begin{equation}\label{doubleetoile}-\frac{1}{2} \tr_m \tilde G_n''(zI_m) \cdot\left( \sum_{v=1}^r \alpha_v L_n(zI_m) \alpha_v,  \sum_{v=1}^r \alpha_v L_n(zI_m) \alpha_v\right)=O(\frac{1}{n\sqrt{n}}).\end{equation}
First note that, by Riesz-Fr\'echet's Theorem (and using \eqref{normeG} and \eqref{L}), there exists $B_n^{(1)}(z)$ and $B_n^{(2)}(z)$ in $M_m(\C)$ such that  $$\Vert B_n^{(1)}(z) \Vert_2 \leq \vert \Im  z \vert^{-2} \left( \sum_{v=1}^r \Vert \alpha_v\Vert^2 \right)=O(1), $$  \begin{equation}\label{B2}\Vert B_n^{(2)}(z) \Vert_2=O(1/\sqrt{n}), \end{equation} and 
$$\sum_{v=1}^r \tr_m\tilde G_n'(zI_m) \cdot \alpha_v L_n(zI_m) \alpha_v = \Tr_m \left[ B_n^{(1)}(z) L_n(zI_m) \right],$$
\begin{equation}\label{etoilehat}\tr_m \tilde G_n''(zI_m) \cdot\left( \sum_{v=1}^r \alpha_v L_n(zI_m) \alpha_v,  \sum_{v=1}^r \alpha_v L_n(zI_m) \alpha_v\right)= \Tr_m \left[ B_n^{(2)}(z) L_n(zI_m) \right].\end{equation}

\noindent Recall that for $\lambda \in M_m(\C)$ such that $\Im  \lambda>0$,  $$L_n(\lambda)=\frac{1}{n}\sum_{p=1}^n  \left[Y_n(\lambda)
\Psi(\lambda) \right]_{pp},$$ where $\Psi$ is defined in \eqref{psi}.
First, note that according to \eqref{YPSIF}, we have (setting  $c_{l,i,v}= \frac{(\kappa_3^{i,l,v} -\tilde \kappa_3^{i,l,v}\sqrt{-1})}{2\sqrt{2}} $)\\
~~

\noindent $\Tr_m \left[ B_n^{(2)}(z) L_n(zI_m) \right]$
\begin{eqnarray*} 
\hspace*{-1.6cm}=&\hspace*{-0.3cm}
\displaystyle{ \sum_{v=1}^r \sum_{i,l=1}^n} \frac{c_{l,i,v}}{n^2\sqrt{n}}  \mathbb{E} \Tr_m \left\{\alpha_v(R_n(zI_m))_{ii}\alpha_v(R_n(zI_m))_{ll}\alpha_v 
 \left(R_n(zI_m) (B_n^{(2)}(z)\otimes I_n)Y_n(zI_m)\right)_{il}\right\}
\\&+O(\frac{1}{n\sqrt{n}}).
\end{eqnarray*}
Moreover, we have 
 $$  \sum_{v=1}^r\sum_{i,l=1}^n \frac{ c_{l,i,v}}{n^2\sqrt{n}} \mathbb{E} \Tr_m \left\{\alpha_v(R_n(zI_m))_{ii}\alpha_v(R_n(zI_m))_{ll}\alpha_v 
 \left(R_n(zI_m) (B_n^{(2)}(z)\otimes I_n)Y_n(zI_m)\right)_{il}\right\}$$ 
$$= O(\frac{1}{n\sqrt{n}}),$$
where we used  \eqref{Odenracinepas}, \eqref{norme}, \eqref{Y} and  \eqref{B2}. 
Hence \begin{equation}\label{etoiletilde}\Tr_m \left[ B_n^{(2)}(z) L_n(zI_m) \right]= O(\frac{1}{n\sqrt{n}}).\end{equation}
\eqref{doubleetoile}, \eqref{etoilehat} and \eqref{etoiletilde} yield  \begin{equation}\label{doubleetoilehat}g_n(z)- \tilde g_n(z)+\sum_{v=1}^r tr_m\tilde G_n'(zI_m) \cdot \alpha_v L_n(zI_m) \alpha_v -tr_mL_n(zI_m)=O(\frac{1}{n\sqrt{n}}).\end{equation}
 Thus, in the following we will consider $\frac{1}{n}\sum_{p=1}^n \tr_m B_n(\lambda) \left[Y_n(\lambda)
T(\lambda) \right]_{pp}$  for any $\lambda\in M_m\left( \mathbb{C}\right)$ such that $\Im \lambda >0$,  for each term $T(\lambda)$ involving in \eqref{psi} and any $m\times m$ matrix $B_n(\lambda)=O(1)$ (in the interests of simplifying notations, we deal with any $\lambda\in M_m\left( \mathbb{C}\right)$  such that $\Im \lambda >0$ instead of $zI_m$).
We set   ${\cal B}(\lambda)= B_n(\lambda) \otimes I_{n}$.

\noindent First, for any fixed $v \in \{1,\ldots,r\}$,\\
 
$\left|\frac{1}{n}\sum_{p,l=1}^n \tr_m B_n(\lambda)(Y_n(\lambda))_{pl} \mathbb{E}\left[ \alpha_v [H_n(\lambda) -\mathbb{E}(H_n(\lambda))] \alpha_v    [R_n(\lambda ) ]_{lp} \right]\right|$
\begin{eqnarray*}
&=& \left|\tr_m \mathbb{E}\left[ \alpha_v [H_n(\lambda) -\mathbb{E}(H_n(\lambda))] \alpha_v id\otimes  tr_n  [R_n(\lambda ) {\cal B}(\lambda)Y_n(\lambda)] \right]\right|\\
&=& \left|\tr_m \mathbb{E}\left\{ \alpha_v [H_n(\lambda) -\mathbb{E}(H_n(\lambda))] \alpha_v\right. \right.\\&&\left.\left.~~~~~~\times \left[id\otimes  tr_n  [R_n(\lambda ) {\cal B}(\lambda)Y_n(\lambda)]
-\mathbb{E}(id\otimes  tr_n  [R_n(\lambda ) {\cal B}(\lambda)Y_n(\lambda)]) \right]\right\}\right|\\&
\leq&\frac{\Vert B_n(\lambda) \Vert \Vert \alpha_v\Vert^2 C m}{n^2}
 \Vert (Im(\lambda))^{-1}\Vert^5\\&=& O(1/n^2).\end{eqnarray*}
where we used Cauchy Schwarz's inequality, \eqref{norme}, \eqref{Y} and Lemma \ref{var} in the last line.\\
We also have
$$\left|\sum_{i,p,l=1}^n  \frac{ (\kappa_3^{i,l,v} +\tilde \kappa_3^{i,l,v}\sqrt{-1})}{n^2\sqrt{n}} \tr_m B_n(\lambda)(Y_n(\lambda))_{pl}    \mathbb{E} \{\alpha_v(R_n(\lambda))_{il}\alpha_v(R_n(\lambda))_{il}\alpha_v 
 (R_n(\lambda))_{ip}\}\right|$$
\begin{eqnarray*}&=&\frac{1}{n^2\sqrt{n}} \left| \sum_{i,l=1}^n  (\kappa_3^{i,l,v} +\tilde \kappa_3^{i,l,v}\sqrt{-1})  \mathbb{E}\tr_m \{\alpha_v(R_n(\lambda))_{il}\alpha_v(R_n(\lambda))_{il}\alpha_v 
 (R_n(\lambda){\cal B}(\lambda) Y_n(\lambda))_{il}\}\right|
\\&=&O(\frac{1}{n\sqrt{n}})
\end{eqnarray*}
where we used \eqref{Odenpas}, \eqref{norme} and  \eqref{Y}.

\noindent Now, let us investigate  the terms corresponding to the the  $M^{(4)}(v,i,l,j)$'s in \eqref{psi}. We have 
$$  \frac{1}{4n^3} \sum_{i,p,l=1}^n  (\kappa_4^{i,l,v} +\tilde \kappa_4^{i,l,v})\mathbb{E} \tr_m B_n(\lambda)(Y_n(\lambda))_{pl}  \alpha_v (R_n(\lambda))_{il}  \alpha_v (R_n(\lambda))_{il} \alpha_v (R_n(\lambda))_{il}\alpha_v  (R_n(\lambda))_{ip}$$
\begin{eqnarray*}
&=& \frac{1}{4n^3} \sum_{i,l=1}^n  (\kappa_4^{i,l,v} +\tilde \kappa_4^{i,l,v}) \mathbb{E}\tr_m    \alpha_v(R_n(\lambda))_{il}\alpha_v(R_n(\lambda))_{il}\alpha_v(R_n(\lambda))_{il} \alpha_v  \left(R_n(\lambda){\cal B}(\lambda)Y_n(\lambda)\right)_{il}
\\&=&O(1/n^2) \end{eqnarray*}
where we used \eqref{Odenpas}, \eqref{norme} and \eqref{Y}.
Similarly the terms corresponding to  \eqref{deuxcas4}, \eqref{quatrecas4}, \eqref{tilde1},   \eqref{tilde2}, \eqref{tilde3} and  \eqref{tilde4}
are $O(1/n^2).$\\

\noindent Since moreover each term in \eqref{psi} corresponding to the $M^{(3)}(v,i,i,j)$, $M^{(4)}(v,i,i,j)$ and $M^{(5)}(v,i,l,j)$ leads obviously to a term which is a $O(\frac{1}{n\sqrt{n}})$, it readily follows from \eqref{doubleetoilehat} that $$ g_n(z)- \tilde g_n(z)+ \sum_{v=1}^r tr_m\tilde G_n'(zI_m) \cdot \alpha_v \hat L_n(zI_m) \alpha_v -tr_m \hat L_n(zI_m)=O(\frac{1}{n\sqrt{n}}),$$
where \\

\noindent $\hat L_n (zI_m)=\sum_{v=1}^r \sum_{i,p,l=1}^n $ $$\bigg\{  \frac{(\kappa_4^{i,l,v} +\tilde \kappa_4^{i,l,v})}{4n^3} ( Y_n(zI_m)) _{pl} \alpha_v \mathbb{E} \left\{(R_n(zI_m))_{ii} \alpha_v (R_n(zI_m))_{ll} \alpha_v(R_n(zI_m))_{ii}\alpha_v (R_n(zI_m))_{lp}\right\}$$
$$+\frac{(\kappa_3^{i,l,v}+\tilde \kappa_3^{i,l,v}\sqrt{-1})}{2\sqrt{2} n^2 \sqrt{n} } ( Y_n(zI_m)) _{pl} \alpha_v  \mathbb{E} \left\{(R_n(zI_m))_{ii} \alpha_v (R_n(zI_m))_{li} \alpha_v (R_n(zI_m))_{lp}\right\}$$
$$+ \frac{(\kappa_3^{i,l,v}-\tilde \kappa_3^{i,l,v}\sqrt{-1})}{2\sqrt{2} n^2 \sqrt{n} }  ( Y_n(zI_m)) _{pl} \alpha_v \mathbb{E} \left\{ (R_n(zI_m))_{ii} \alpha_v (R_n(zI_m))_{ll} \alpha_v (R_n(zI_m))_{ip}\right\}$$
$$+ \frac{(\kappa_3^{i,l,v}-\tilde \kappa_3^{i,l,v}\sqrt{-1})}{2\sqrt{2} n^2 \sqrt{n} }  ( Y_n(zI_m)) _{pl} \alpha_v \mathbb{E} \left\{(R_n(zI_m))_{il} \alpha_v (R_n(zI_m))_{ii} \alpha_v (R_n(zI_m))_{lp}\right\}\bigg\}$$
For any $m\times m$ deterministic matrix $B_n(z)$, \\
~~

\noindent $\tr_m B_n (z)\hat L_n (zI_m)=\sum_{v=1}^r\sum_{i,l=1}^n$ $$\bigg\{\frac{(\kappa_4^{i,l,v} +\tilde \kappa_4^{i,l,v}) }{4n^3}  \mathbb{E} \left\{\tr_m  \alpha_v (R_n(zI_m))_{ii} \alpha_v (R_n(zI_m))_{ll}\alpha_v (R_n(zI_m))_{ii}\alpha_v (R_n(zI_m){\cal B}(z)Y_n(zI_m))_{ll}\right\}$$
$$+ \frac{(\kappa_3^{i,l,v}+\tilde \kappa_3^{i,l,v}\sqrt{-1}) }{2\sqrt{2} n^2 \sqrt{n} } \mathbb{E} \tr_m \left\{ \alpha_v (R_n(zI_m))_{ii} \alpha_v (R_n(zI_m))_{li} \alpha_v \left (R_n(zI_m){\cal B}(z)Y_n(zI_m)\right)_{ll}\right\}$$
$$+ \frac{ (\kappa_3^{i,l,v}-\tilde \kappa_3^{i,l,v}\sqrt{-1}) }{2\sqrt{2} n^2 \sqrt{n} } \mathbb{E} \tr_m\left\{ \alpha_v (R_n(zI_m))_{ii} \alpha_v (R_n(zI_m))_{ll} \alpha_v \left (R_n(zI_m){\cal B}(z)Y_n(zI_m)\right)_{il}\right\}$$
$$+\frac{ (\kappa_3^{i,l,v}-\tilde \kappa_3^{i,l,v}\sqrt{-1})}{2\sqrt{2} n^2 \sqrt{n} } \mathbb{E} \tr_m\left\{ \alpha_v (R_n(zI_m))_{il} \alpha_v (R_n(zI_m))_{ii} \alpha_v\left (R_n(zI_m){\cal B}(z)Y_n(zI_m)\right)_{ll}\right\}\bigg\}$$

\noindent Hence (using a decomposition similar to \eqref{decomposition}) Lemma \ref{var} readily yields that \\

\hspace*{-0.8cm} $\tr_m B_n \hat L_n (zI_m) + O(\frac{1}{n\sqrt{n}})= \sum_{v=1}^r \sum_{i,l=1}^n \tr_m$ $$\hspace*{-0.8cm}\bigg\{\frac{(\kappa_4^{i,l,v} +\tilde \kappa_4^{i,l,v}) }{4n^3} \alpha_v \mathbb{E}\left[ (R_n(zI_m))_{ii} \right]  \alpha_v \mathbb{E}\left[(R_n(zI_m))_{ll}\right] \alpha_v\mathbb{E}\left[ (R_n(zI_m))_{ii} \right]\alpha_v [ \mathbb{E}\left[\left (R_n(zI_m){\cal B}(z)Y_n(zI_m)\right)_{ll}\right]$$
$$\hspace*{-0.8cm}+ \frac{(\kappa_3^{i,l,v}+\tilde \kappa_3^{i,l,v}\sqrt{-1})}{2\sqrt{2} n^2 \sqrt{n} }   \alpha_v  \mathbb{E}\left[(R_n(zI_m))_{ii}\right]  \alpha_v  \mathbb{E}\left[(R_n(zI_m))_{li}\right] \alpha_v   \mathbb{E}\left[\left(R_n(zI_m){\cal B}(z)Y_n(zI_m)\right)_{ll}\right] $$
$$\hspace*{-0.8cm}+ \frac{(\kappa_3^{i,l,v}-\tilde \kappa_3^{i,l,v}\sqrt{-1})}{2\sqrt{2} n^2 \sqrt{n} }     \alpha_v  \mathbb{E}\left[(R_n(zI_m))_{ii}\right]  \alpha_v  \mathbb{E}\left[(R_n(zI_m))_{ll}\right]  \alpha_v  \mathbb{E}\left[ \left(R_n(zI_m){\cal B}(z)Y_n(zI_m)\right)_{il}\right]$$
$$\hspace*{-0.8cm}+ \frac{(\kappa_3^{i,l,v}-\tilde \kappa_3^{i,l,v}\sqrt{-1})}{2\sqrt{2} n^2 \sqrt{n} }  \alpha_v  \mathbb{E}\left[(R_n(zI_m))_{il} \right] \alpha_v  \mathbb{E}\left[(R_n(zI_m))_{ii}\right]  \alpha_v  \mathbb{E}\left[\left (R_n(zI_m){\cal B}(z)Y_n(zI_m)\right)_{ll}\right]\bigg\}.$$
Note that, with $Y_n$ and $\tilde Y_n$  defined in \eqref{y} and \eqref{ytilde}, we have  $$Y_n(zI_m)-\tilde Y_n(zI_m) =\sum_{v=1}^r Y_n(zI_m) \left[\alpha_v (G_n(zI_m)-\tilde G_n(zI_m)) \alpha_v \otimes I_n\right] \tilde Y_n(zI_m)$$ so that, using (\ref{difgngntilde}), \eqref{Y} and \eqref{Y2}, we can deduce that \begin{equation}\label{YmoinsYtilde} \Vert Y_n(zI_m) -\tilde Y_n(zI_m)\Vert =O(1/\sqrt{n}). \end{equation}
Now, (\ref{estimenracinen}) and (\ref{YmoinsYtilde}) obviously yield that, up to a $O(\frac{1}{n\sqrt{n}})$ correction term, one can replace any $R_n(zI_m)$  and  $Y_n(zI_m)$ by $\tilde Y_n(zI_m)$ in any term in the sum corresponding to the fourth cumulants. Now, using (\ref{lp}), (\ref{estimenracinen}) and (\ref{YmoinsYtilde}) also  yield that, up to a $O(\frac{1}{n\sqrt{n}})$ correction term, for any $p=1,\ldots,n$, one can replace  $(R_n(zI_m))$ and  $(Y_n(zI_m))$ by $\tilde Y_n(zI_m)$ in any diagonal term $(R_n(zI_m))_{pp}$ or $(R_n(zI_m){\cal B}(z)Y_n(zI_m))_{pp} $ in the sums corresponding to the third cumulants.\\
Finally, 
assume that   for $i=1,2,$ $Q^{(i)}= \tilde Y_n(zI_m) $ or $\tilde Y_n(zI_m){\cal B}(z)\tilde Y_n(zI_m)$. Let us consider any  $ Q^{(3)}=  \sum_{i,l=1}^nQ^{(3)}_{il}\otimes E_{il}$.
It is clear that if there exists some polynomial $Q$ with nonnegative coefficients such that for any $i,l \in \{1,\ldots,n\}^2$, $\Vert Q^{(3)}_{il}\Vert \leq \frac{Q(\vert \Im  z \vert^{-1})}{n}$, then 
$$\frac{1}{n^2 \sqrt{n}} \sum_{i,l=1}^n\Vert Q^{(1)}_{ii}\Vert \Vert Q^{(2)}_{ll}\Vert \Vert Q^{(3)}_{il}\Vert =O(1/n\sqrt{n}).$$
Now, if $\Vert Q^{(3)}\Vert =O(1/\sqrt{n})$,
we have $$\frac{1}{n^2 \sqrt{n}} \sum_{i,l=1}^n\Vert Q^{(1)}_{ii}\Vert \Vert Q^{(2)}_{ll}\Vert \Vert Q^{(3)}_{il}\Vert $$ $$\leq \frac{1}{n \sqrt{n}} \vert (\Im  z)^{-1}\vert^q\left(\sum_{i,l=1}^n\Vert Q^{(3)}_{il}\Vert^2 \right)^{1/2} \leq  \frac{\sqrt{m}}{n } \vert (\Im  z)^{-1}\vert^q\Vert Q^{(3)}\Vert =O(1/n\sqrt{n}) $$ for some $q\in \mathbb{N}\setminus{\{0\}},$ where we used \eqref{lp}.
\\
It is then clear that using Corollary \ref{estimenunsurn}, \eqref{Q} and (\ref{YmoinsYtilde}), 
up to a $0(\frac{1}{n\sqrt{n}})$ correction term, for any $(i,l)\in \{1,\ldots,n\}^2$, one can replace   $R_n(zI_m)$ and  $Y_n(zI_m)$ by $\tilde Y_n(zI_m)$ in any non-diagonal term $(R_n(zI_m))_{il}$, $(R_n(zI_m))_{li}$ or  $(R_n(zI_m){\cal B}Y_n)_{il}$ in the sums corresponding to the third cumulants.
Hence \eqref{estimdiffeqn} is proved for any $z\in \mathbb{C}$ such that $\Im  z >0$. \\Set $\alpha=(\alpha_1,\ldots \alpha_r)$, $\beta=(\beta_1,\ldots, \beta_t)$. Let us denote for a while $g_n=g_n^{\alpha,\beta,\gamma}$, $\tilde g_n=\tilde g_n^{\alpha,\beta,\gamma}$ and $\tilde E_n=\tilde E_n^{\alpha,\beta,\gamma}$. Note that we have similarly for  any $z\in \mathbb{C}$ such that $\Im  z >0$, 
\begin{equation} \label{estimdiffeqnmoins}
g_n^{-\alpha,-\beta,-\gamma}(z)-\tilde g^{-\alpha,-\beta,-\gamma}_n(z)+{\tilde{E}^{-\alpha,-\beta,-\gamma}_n(z)} = O(\frac{1}{n\sqrt{n}}).
\end{equation} Thus, since $g_n^{-\alpha,-\beta,-\gamma}(z)=-g_n^{\alpha,\beta,\gamma}(-z)$, $\tilde g_n^{-\alpha,-\beta,-\gamma}(z)=-\tilde g_n^{\alpha,\beta,\gamma}(-z)$ and  $\tilde E_n^{-\alpha,-\beta,-\gamma}(z)=-\tilde E_n^{\alpha,\beta,\gamma}(-z)$, it readily follows that  
\eqref{estimdiffeqn} is also valid for any $z \in \mathbb{C}$ such that $\Im  z <0$.
\end{proof}

\subsection{ From Stieltjes transform estimates to spectra.}
We start with the following key lemma.
\begin{lemme}\label{LSt}
For any fixed large n, $ \tilde E_n$ defined in Proposition \ref{estimdiff} is the Stieltjes transform of a compactly supported distribution $\nabla_n$ on $\mathbb{R}$ whose support is
included in the spectrum of $s_n=\gamma \otimes 1_{\cal A} + \sum_{v=1}^r\alpha_v \otimes x_v +\sum_{u=1}^t\beta_u \otimes a_n^{(u)}$ and such that $\nabla_n(1)=0$ .
\end{lemme}

The proof relies on the following characterization already used in \cite{Schultz05}. 

\begin{theoreme}\label{TS}\cite{Tillmann53}
\begin{itemize}
\item Let $\Lambda $ be a distribution on $\R$ with compact support. 
Define the Stieltjes transform of $\Lambda $, 
$ l:\C\setminus \R \rightarrow \C$ by 
$$l(z)=\Lambda \left( \frac{1}{z-x}\right) .$$
\noindent Then $l$ is analytic on $\C\setminus \R$
and has an analytic continuation to $\C\setminus {\rm supp}(\Lambda )$. 
Moreover
\begin{itemize}
\item[($c_1$)] $l(z)\rightarrow 0$ as $|z|\rightarrow \infty ,$
\item[($c_2$)] there exists a constant $C > 0$, 
an integer $q\in \N$ and a compact set $K\subset \R$ containing ${\rm supp}(\Lambda )$, 
such that for any $z\in \C\setminus \R$, 
$$|l(z)|\leq C\max \{ {\rm dist}(z,K)^{-q}, 1\} ,$$
\item[($c_3$)] for any $\phi \in \cal C^\infty (\R, \R)$ with compact support
$$\Lambda (\phi )=\frac{i}{2\pi }\lim _{y\rightarrow 0^+} \int _\R\phi (x)[l(x+iy)-l(x-iy)]dx.$$
\end{itemize}
\item Conversely, if $K$ is a compact subset of $\R$ 
and if $l:\C \setminus K\rightarrow \C$ is an analytic function
satisfying ($c_1$) and ($c_2$) above, 
then $l$ is the Stieltjes transform of a compactly supported distribution $\Lambda $ on $\R$. 
Moreover, ${\rm supp}(\Lambda )$ is exactly the set of singular points of $l$ in $K$. 
\end{itemize}
\end{theoreme}
\begin{lemme} The singular points of $\tilde{E}_N$ defined in \eqref{defentilde} are included in the 
spectrum of  $s_n=\gamma \otimes 1_{\cal A}+\sum_{v=1}^r \alpha_v \otimes x_v +\sum_{u=1}^t\beta_u \otimes a_n^{(u)}$. 
\end{lemme}
\begin{proof}
Let us start by noting that, as it follows from the definition \eqref{defentilde} of $\tilde{E}_N$,
it is enough to show that $\omega_n(zI_m)\otimes I_n-\sum_{u=1}^t\beta_u \otimes A_n^{(u)}
\in GL_{nm}(\mathbb C)$ (the group of invertible $nm\times nm$ complex matrices) for any $z$
in the domain of definition of $\mathbb C\ni z\mapsto\omega_n(zI_m)\in M_{m}(\mathbb C).$

Assume towards contradiction that $x_0\in\mathbb C$ is in the domain of $\omega_n(\cdot I_m)$, and 
yet $\omega_n(x_0I_m)\otimes I_n-\sum_{u=1}^t\beta_u \otimes A_n^{(u)}$ is not invertible. First,
observe that a point $x_0$ with this property must be isolated and real. Indeed, otherwise the zeros of
the analytic map $\mathbb C\ni z\mapsto\det\left(\omega_n(zI_m)\otimes I_n-
\sum_{u=1}^t\beta_u \otimes A_n^{(u)}\right)\in\mathbb C$ would have $x_0$ as a cluster point
in the interior of its domain (which coincides with the domain of $\omega_n(\cdot I_m)$), and thus
it would be identically equal to zero. However,  
$\omega_n(zI_m)\otimes I_n-\sum_{u=1}^t\beta_u \otimes A_n^{(u)}$ is invertible when $\Im  z\neq0$,
providing us with a contradiction. Consider now such an isolated $x_0$. Recall 
\begin{eqnarray*}
\tilde{g}_n(z) & = & (\tr_m\otimes\tau)\left(\left(
\omega_n(zI_m)\otimes1_\mathcal A-\sum_{u=1}^t\beta_u \otimes a_n^{(u)}\right)^{-1}\right)\\
& = & (\tr_m\otimes\tr_n)\left(\left(
\omega_n(zI_m)\otimes I_n-\sum_{u=1}^t\beta_u \otimes A_n^{(u)}\right)^{-1}\right),
\end{eqnarray*}
from before (the second equality is justified by the hypothesis that the distribution of $(A_n^{(1)},
\dots,A_n^{(t)})$ with respect to $\tr_n$ coincides with the distribution of $(a_n^{(1)},\dots,a_n^{(t)})$
with respect to $\tau$). We have seen that $\tilde{g}_n$ is defined exactly on the complement of the 
spectrum of $s_n$. Thus, it is enough to show that, given an analytic function 
$f\colon\mathbb C^+\to H^{-}(M_p(\mathbb C))=\{b\in M_p(\mathbb{C}), \Im  b <0\}$, and $x_0\in\mathbb R$ with the property that there
exists some $\epsilon>0$ such that $f$ extends analytically through $(x_0-\epsilon,x_0)\cup(x_0,
x_0+\epsilon)$ with self-adjoint values, then either both or none of $f$ and $\tr_p\circ f$ extend 
analytically to $x_0$. We shall then apply this to $p=mn$ and $f(z)=\left(\omega_n(zI_m)\otimes I_n-
\sum_{u=1}^t\beta_u \otimes A_n^{(u)}\right)^{-1}$ to conclude.
 
It is clear that if $f$ extends analytically through $x_0$, then so does $\tr_p\circ f$. Assume 
towards contradiction that $\tr_p\circ f$ extends analytically through $x_0$, but $f$ does not.
Consider an arbitrary system $\{e_1,\dots,e_p\}$ of minimal mutually orthogonal projections
in $M_p(\mathbb C)$. It is clear that $z\mapsto e_jf(z)e_j\in e_jM_p(\mathbb C)e_j\simeq\mathbb C$ is 
analytic wherever $f$ is. Moreover, $\Im  e_jf(z)e_j\leq0$ whenever $z\in\mathbb C^+$, and 
$e_jf(z)e_j\in\mathbb R$ whenever $f(z)$ is self-adjoint. If $z\mapsto e_jf(z)e_j$ extends analytically
through $x_0$ for any $e_j\in\{e_1,\dots,e_p\}$ and all systems $\{e_1,\dots,e_p\}$, then 
$z\mapsto\varphi(f(z))$ extends analytically through $x_0$ for all linear functionals $\varphi\colon
M_p(\mathbb C)\to\mathbb C$. This implies that $f$ itself is analytic around $x_0$. On the other hand,
if there exists a system $\{e_1,\dots,e_p\}$ of minimal mutually orthogonal projections
in $M_p(\mathbb C)$ which contains an $e_j$ such that $z\mapsto e_jf(z)e_j$ does not 
extend analytically through $x_0$, then, as $z\mapsto e_jf(z)e_j$ does extend analytically 
with real values through $(x_0-\epsilon,x_0)\cup(x_0,x_0+\epsilon)$ and maps $\mathbb C^+$
into $\mathbb C^-\cup\mathbb R$, it follows that $x_0$ is a simple pole of 
$z\mapsto e_jf(z)e_j$ and $\lim_{y\to0}iye_jf(x_0+iy)e_j\in(0,+\infty)$ by the Julia-Carath\'eodory
Theorem applied to $1/e_jf(z)e_j$ (see (2) Theorem 2.1 in \cite{Serban}). But then
$\tr_p(f(z))=\sum_{k=1}^pe_kf(z)e_k$, so that 
$$
\lim_{y\to0}iy\tr_p(f(x_0+iy))=\sum_{k=1}^p\lim_{y\to0}iye_kf(x_0+iy)e_k>0,
$$
which contradicts the assumption that $\tr_p\circ f$ extends analytically through $x_0$.

\end{proof}

Now, we are going to show that for any fixed large $n$, 
$\tilde{E}_n$ satisfies ($c_1$) and ($c_2$) of Theorem \ref{TS}.

\noindent First note that there exists a polynomial Q in two variables  with positive coefficients such that 
\begin{equation}\label{normedetildeEn} \vert \tilde {E}_n(z)\vert  \leq \Vert \tilde Y(zI_m))\Vert^4 Q(\Vert \tilde Y(zI_m))\Vert, \Vert r_n(zI_m)\Vert ).
\end{equation}

\noindent Let $C > 0$ be such that, for all  $n$, $\mbox{sp}(\sum_{u=1}^t\beta_u \otimes A_n^{(u)})\subset [-C;C]$ and
$\mbox{sp}(\gamma \otimes 1_{\cal A} +\sum_{v=1}^r \alpha_v \otimes x_v +\sum_{u=1}^t\beta_u \otimes a_n^{(u)})\subset [-C;C]$.

\noindent Let $d > C + \sqrt{r}\max_{v=1}^r\Vert \alpha_v \Vert $. 
For any $z\in \C$ such that $|z| > \Vert \gamma \Vert +d $, 
\begin{eqnarray*}\Vert \gamma +\sum_{v=1}^r \alpha_v \tilde G_n(zI_m) \alpha_v \Vert &\leq& \Vert \gamma \Vert + \frac{r \max_{v=1}^r\Vert \alpha_v \Vert ^2}{\vert z\vert -C}\\&\leq &\Vert \gamma \Vert + \frac{r \max_{v=1}^r\Vert \alpha_v \Vert ^2}{d-C} \\&
<& \Vert \gamma \Vert + \frac{(d -C)^2}{d -C} \\&=& \Vert \gamma \Vert + d-C\end{eqnarray*}
 Thus, $$\Vert  \gamma+\sum_{v=1}^r\alpha_v\tilde G_n(zI_m) \alpha_v +\sum_{u=1}^t \beta_u \otimes A_n^{(u)}\Vert \leq\Vert \gamma  \Vert +d $$
so that we get that for any $z\in \C$ such that $|z| > \Vert \gamma\Vert +d $, 
\begin{eqnarray*}\Vert \tilde Y_n(zI_m)\Vert &=&
\Vert ((zI_m- \gamma -\sum_{v=1}^r \alpha_v \tilde G_n(zI_m) \alpha_v)\otimes I_n-\sum_{u=1}^t\beta_u \otimes A_n^{(u)})^{-1}\Vert\\
& \leq &\frac{1}{\vert z\vert -\Vert \gamma \Vert -d }.
\end{eqnarray*}
We get readily from \eqref{normedetildeEn} that, for $|z| > \Vert \gamma\Vert +d $, 
\begin{equation}\label{bound} \vert \tilde{E}_n(z)\vert \leq \frac{1} {(\vert z\vert -\Vert \gamma \Vert -d )^4}Q\left(\frac{1} {(\vert z\vert -\Vert \gamma \Vert -d )},\frac{1}{(|z|-C)} \right) .\end{equation}
Then, it is clear than $\vert\tilde{E}_n(z)\vert \rightarrow 0$ 
when $|z|\rightarrow +\infty $ and ($c_1$) is satisfied.\\
\noindent 
Now we are going to prove ($c_2$) using  the approach of \cite{Schultz05}(Lemma 5.5). 
Denote by $\mathcal {E}_n$ the convex envelope of the spectrum of $s_n=\gamma \otimes 1_{\cal A} +\sum_{v=1}^r\alpha_v \otimes x_v +\sum_{u=1}^t\beta_u \otimes a_n^{(u)}$
and define $$K_n:=\left\{ x\in \R; {\rm dist}(x, \mathcal {E}_n)\leq 1\right\} $$
\noindent 
and $$D_n=\left\{ z\in \C; 0 < {\rm dist}(z, K_n)\leq 1\right\} .$$
\begin{itemize}
\item Let $z\in D_n\cap (\C\setminus \R)$ with $\Re  (z)\in K_n$. 
We have ${\rm dist}(z, K_n)=|\Im  z|\leq 1$. 
We have from \eqref{normedetildeEn},  (\ref{normeG}) and  \eqref{Y2} that 
$$\vert \tilde{E}_n(z)\vert \leq {\vert \Im  z \vert^{-4}} Q\left(\vert \Im  z \vert^{-1},\vert \Im  z \vert^{-1}\right).$$
Noticing that $1\leq {\vert \Im  z\vert^{-1}}$, 
we easily deduce that there exists some constant $C_0$ and some  number $q_0 \in \mathbb{N}\setminus \{0\}$ such that 
for any $z\in D_n\cap \C\setminus \R$ with $\Re  (z)\in K_n$, 
\begin{eqnarray*}
\vert \tilde{E}_n(z)\vert &\leq &C_0|\Im  z|^{-q_0}\\
&\leq &C_0{\rm dist}(z, K_n)^{-q_{0}}\\
&\leq &C_0\max ({\rm dist}(z, K_n)^{-q_{0}}; 1)
\end{eqnarray*}
\item Let $z\in D_n\cap (\C\setminus \R)$ with $\Re (z)\notin K_n$. 
Then ${\rm dist}(z,{\rm sp}(s_n))\geq 1$. 
Since $\tilde{E}_n$ is bounded on compact subsets of $\C\setminus {\rm sp}(\gamma \otimes 1_{\cal A} +\sum_{v=1}^r\alpha_v \otimes x_v +\sum_{u=1}^t\beta_u \otimes a_n^{(u)})$,
we easily deduce that there exists some constant $C_1(n)$ 
such that for any $z\in D_n$ with $\Re  (z)\notin K_n$, 
$$\vert \tilde{E}_n(z)\vert \leq C_1(n)\leq C_1(n)\max ({\rm dist}(z, K_n)^{-q_0}; 1).$$
\item Since $\vert \tilde{E}_n(z)\vert \rightarrow 0$ when $|z|\rightarrow +\infty $, 
$\tilde{E}_n$ is bounded on $\C\setminus \overline{D_n}$. 
Thus, there exists some constant $C_2(n)$ such that for any $z\in \C\setminus \overline{D_n}$, 
$$\vert \tilde{E}_n(z)\vert \leq C_2(n)=C_2(n)\max ({\rm dist}(z, K_n)^{-q_0}; 1).$$
\end{itemize}
Hence ($c_2$) is satisfied with $C(n)=\max (C_0, C_1(n), C_2(n))$ and $l=q_0$.  Thus, Theorem \ref{TS} implies that for any fixed large n, $ \tilde E_n$ defined in Proposition \ref{estimdiff} is the Stieltjes transform of a compactly supported distribution $\nabla_n$ on $\mathbb{R}$ whose support is
included in the spectrum of $s_n=\gamma \otimes 1_{\cal A} +\sum_{v=1}^r\alpha_v \otimes x_v +\sum_{u=1}^t\beta_u \otimes a_n^{(u)}$.  Following the proof of Lemma 5.6 in \cite{Schultz05} and using (\ref{bound}), 
one can show that $\nabla _n(1)=0$. 
The proof of Lemma \ref{LSt} is complete. $\Box$\\

\noindent Now,  \eqref{spectre3} can be deduced from \eqref{estimdiffeqno} by an approach inspired by \cite{HT} and \cite{Schultz05} as follows.  \\
Using the inverse Stieltjes tranform, we get respectively that, 
for any $\varphi _n $ in ${\cal C}^\infty (\R, \R)$ with compact support, 
$$\mathbb{E} [\tr_m \otimes \tr _n(\varphi _n(S_n))]-\tr_m\otimes \tau(\varphi_n(s_n))
+{\nabla _{n}(\varphi _n)}$$
$$=\frac{1}{\pi }\lim _{y\rightarrow 0^+}\Im  \int _\R\varphi _n(x)\epsilon_n(x+iy)dx,$$
where $\epsilon_n(z)=\tilde g_n(z)-g_n(z)-\tilde{E}_n(z)$ satisfies, 
according to Proposition \ref{estimdiff}, for any $z\in \C\setminus \R$, 
\begin{equation*}\label{estimgdif}
\vert \epsilon_n(z)\vert \leq \frac{1}{n\sqrt{n}}P(\vert \Im z \vert ^{-1}) .
\end{equation*} 
We refer the reader to the Appendix of \cite{CD07} 
where it is proved using the ideas of \cite{HT} that if  $h$ is an analytic function on $\C\setminus \R$ which satisfies
\begin{equation*}\label{nestimgdif}
\vert h(z)\vert \leq P(\vert \Im  z\vert ^{-1})
\end{equation*} 
\noindent for some polynomial $P$ with nonnegative coefficients and degree $k$, then
there exists a polynomial $Q$ such that
$$\limsup _{y\rightarrow 0^+}\vert \int _\R\varphi _n(x)h(x+iy)dx\vert $$
$$\leq \int _\R\int _0^{+\infty }\vert (1+D)^{k+1}\varphi _n(x)\vert Q(t)\exp(-t)dtdx$$
where $D$ stands for the derivative operator.
Hence, if there exists $K > 0$ such that, for all large $n$, 
the support of $\varphi _n$ is included in $[-K, K]$ and 
$\sup _n\sup _{x \in [-K, K]}\vert D^p\varphi _n(x)\vert =C_p < \infty$ for any $p\leq k+1$, 
dealing with $h(z) =n\sqrt{n}\epsilon_n(z)$, we deduce that there exists $C>0$ such that for all large $n$,
\begin{equation*} \label{majlimsup1} 
\limsup _{y\rightarrow 0^+}\vert \int _\R \varphi _n(x)\epsilon_n(x+iy)dx\vert \leq \frac{C}{n\sqrt{n}}
\end{equation*} 
and then 
\begin{equation}\label{StS} 
\mathbb{E} [\tr_m \otimes \tr _n(\varphi _n(S_n))]-\tr_m\otimes \tau(\varphi_n(s_n))
+{\nabla _{n}(\varphi _n)}=O(\frac{1}{n\sqrt{n}}).  
\end{equation}
Let $\rho \geq 0$ be in ${\cal C}^\infty (\R, \R)$ 
such that its support is included in $[-1;1]$ and $\int \rho (x)dx=1$. 
Let $0 < \epsilon < 1$. 
Define for any $x\in \mathbb{R}$, $$\rho _{\frac{\epsilon }{2}}(x)=\frac{2}{\epsilon }\rho(\frac{2x}{\epsilon }).$$
Set $$K_n(\epsilon )=\{ x, {\rm dist}(x, {\rm sp}(s_n))\leq \epsilon \}$$ 
and define for any $x\in \mathbb{R}$, $$f_n(\epsilon )(x)=\int _\mathbb{R} \1 _{K_n(\epsilon )}(y)\rho _{\frac{\epsilon }{2}}(x-y)dy.$$
The function $f_{n}(\epsilon )$ is in ${\cal C}^\infty (\mathbb{R}, \mathbb{R})$, 
$f_{n}(\epsilon )\equiv 1$ on $K_n(\frac{\epsilon }{2})$; 
its support is included in $K_n(2\epsilon )$. 
Since there exists $K$ such that, for all large $n$, the spectrum of $s_n$ 
is included in $[-K;K]$, for all large $n$ the support of $f_n(\epsilon )$ is included in $[-K-2;K+2]$ 
and for any $p > 0$, 
$$\sup _{x\in [-K-2;K+2]}\vert D^pf_n(\epsilon )(x)\vert \leq 
\sup _{x\in [-K-2;K+2]}\int _{-K-1}^{K+1} \vert D^p \rho _{\frac{\epsilon }{2}}(x-y)\vert dy \leq C_p(\epsilon ).$$
Thus, according to \eqref{StS}, 
\begin{equation} 
\mathbb{E} [\tr_m\otimes \tr _n(f_n(\epsilon )(S_n))]-\tr_m \otimes \tau f_n(\epsilon )(s_n)
+{\nabla _n(f_n(\epsilon ))}=O_{\epsilon }(\frac{1}{n\sqrt{n}})
\end{equation}
and 
\begin{equation}\label{prime} 
\mathbb{E} [\tr_m\otimes \tr _n((f_n'(\epsilon ))^2(S_n))]-\tr_m \otimes \tau (f_n'(\epsilon )(s_n))^2
+{\nabla _n((f_n'(\epsilon ))^2)}=O_{\epsilon }(\frac{1}{n\sqrt{n}}).
\end{equation}
According to Lemma \ref{LSt}, we have  $\nabla _n(1)=0$. 
Then, the function $\psi _n(\epsilon )\equiv 1-f_n(\epsilon )$ also satisfies
\begin{equation} 
\mathbb{E} [\tr_m\otimes \tr _n(\psi _n(\epsilon )(S_n))]-\tr_m \otimes \tau \left(\psi _n(\epsilon )(s_n)\right)
+\nabla _n(\psi _n(\epsilon ))=O_{\epsilon }(\frac{1}{n\sqrt{n}}). 
\end{equation}
Moreover, since $\psi _n'(\epsilon )=-f_n'(\epsilon )$, 
it comes readily from \eqref{prime} that 
$$\mathbb{E} [\tr _n((\psi _n'(\epsilon ))^2(S_n))]-\tr_m \otimes \tau (\psi _n'(\epsilon )(s_n))^2
+{\nabla_n((\psi _n'(\epsilon ))^2)}=O_{\epsilon }(\frac{1}{n\sqrt{n}}).$$
Now, since $\psi _n(\epsilon )\equiv 0$ on the spectrum of $s_n$, 
we deduce that 
\begin{equation}\label{psi2} 
\mathbb{E} [\tr_m\otimes \tr _n(\psi _n(\epsilon )(S_n))]=O_{\epsilon }\left(\frac{1}{n\sqrt{n}}\right)
\end{equation}
and 
\begin{equation} \label{psiprime} \mathbb{E} [\tr_m\otimes \tr _n((\psi _n'(\epsilon ))^2(S_n))]=O_{\epsilon }\left(\frac{1}{n\sqrt{n}}\right).
\end{equation}
By Lemma \ref{variance} (sticking to the proof of Proposition 4.7 in \cite{HT} with $\varphi=f_n(\epsilon)$), 
we have
$$\mathbf{V}{\left[ \tr_m \otimes \tr _n(\psi _n(\epsilon )(S_n))\right] }
\leq \frac{C }{n^2}\mathbb{E} \left[\tr_m \otimes  \tr _n\{ (\psi _n'(\epsilon )(S_n))^2\} \right] .$$
Hence, using \eqref{psiprime}, one can deduce that 
\begin{equation} \label{variancepsi}
\mathbf{V}{\left[ \tr_m \otimes \tr _n(\psi _n(\epsilon )(S_n))\right] }=O_{\epsilon }\left(\frac{1}{n^3\sqrt{n}}\right).
\end{equation}
Fix $0<\delta<\frac{1}{4}$.
Set $$Z_{n, \epsilon }:=\tr_m \otimes \tr _n(\psi _n(\epsilon )(S_n))$$ 
and $$\Omega _{n, \epsilon }=\{ \left| Z_{n, \epsilon }-\mathbb{E}\left(Z_{n, \epsilon }\right)\right| > n^{-(1+\delta)}\}.$$
Hence, using  \eqref{variancepsi}, we have $$\mathbb{P}(\Omega _{n, \epsilon })\leq n^{2+2\delta}\mathbf{V}\{ Z_{n, \epsilon }\}
=O_{\epsilon }\left(\frac{1}{n^{1+\frac{1}{2}-2\delta}}\right).$$
By Borel-Cantelli lemma, we deduce that, almost surely for all large $n$, \begin{equation}\label{concentrezn} \left|Z_{n, \epsilon }-
\mathbb{E}\left(Z_{n, \epsilon }\right)\right| \leq n^{-(1+\delta)}.\end{equation}
From \eqref{psi2} and \eqref{concentrezn}, we deduce that there exists some constant $C_\epsilon$ such that, almost surely for all large $n$,
$$\left| Z_{n, \epsilon }\right|\leq n^{-1}\left(n^{-\delta}+C_\epsilon n^{-1/2}\right).$$
Since $\psi_{n}( \epsilon )\geq \1 _{\mathbb{R}\setminus K_n({2\epsilon })}$, 
it readily follows that, almost surely for all large $n$, 
the number of eigenvalues of $S_n$ which are in $\R\setminus K_n({2\epsilon })$ 
is lower than $m\left(n^{-\delta}+C_\epsilon n^{-1/2}\right)$ and thus obviously,  almost surely for all large $n$, the number of eigenvalues of $S_n$ which are in $\R\setminus K_n({2\epsilon })$ 
 has 
to be equal to zero. Thus we have the following
\begin{theoreme} Let $\epsilon>0$.
Almost surely for all large $n$, the spectrum of $S_n$ is included in $K_n(\epsilon )=\{ x, {\rm dist}(x, {\rm sp}(s_n))\leq \epsilon \}$. 
\end{theoreme}
Since the  above theorem holds for any  $m\times m$ Hermitian matrices $\gamma$, $\{\alpha_v\}_{v=1,\ldots,r}$,  $\{\beta_u\}_{u=1,\ldots,t}$, the proof of Lemma \ref{inclu2} is complete.

\section{Proof of Theorem \ref{noeigenvalue}}\label{sectionnoeigenvalue}
\subsection{Linearization}\label{linearisation}

Linearization procedures are by no means unique, and no agreed upon definition of what a linearization
is exists in the literature. We use the procedure introduced in \cite[Proposition 3]{A}, which has several 
advantages, to be described below.

It is shown in \cite{A} that, given a polynomial $P\in\mathbb C\langle X_1,\dots,X_k\rangle$,
there exist $m\in\mathbb N$ and matrices $\zeta_1,\dots,\zeta_k,\gamma\in M_m(\mathbb C)$
such that $(z-P(X_1,\dots,X_k))^{-1}=\left[\left(z\hat{E}_{11}\otimes1-\gamma\otimes 1-\sum_{j=1}^k
\zeta_j\otimes X_j\right)^{-1}\right]_{11}$. Moreover, if $P=P^*$, then $\gamma$ and $\zeta_1,\dots,
\zeta_k$ can be chosen to be self-adjoint. We denote $L_P=\gamma\otimes1+\sum_{j=1}^k\zeta_j
\otimes X_j\in M_m(\mathbb C\langle X_1,\dots,X_k\rangle)$ and call it a {\em linearization} of $P$. 
The size $m$ and the matrix coefficients $\gamma,\zeta_1,\dots,\zeta_k$ aren't unique. Following \cite{BMS} (see also \cite{Mai}), we provide 
a very brief outline of a recursive construction for a linearization  $L_P$ such that $$L_P := \begin{pmatrix} 0 & u\\v & Q \end{pmatrix} \in M_m(\mathbb{C}) \otimes \mathbb{C} \langle X_1,\ldots, X_k \rangle$$
where
\begin{enumerate}
\item $ m \in \mathbb{N}$,
\item $ Q \in M_{m-1}(\mathbb{C})\otimes \mathbb{C} \langle X_1,\ldots, X_k \rangle$ is invertible,
\item 
  u is a row vector and v is a column vector, both of size $m-1$ with
entries in $\mathbb{C} \langle X_1,\ldots, X_k \rangle$,
\item  the polynomial entries in $Q, u$ and $v$ all have degree $\leq 1$,\\
\item
$${P=-uQ^{-1}v} ,$$
\item and moreover, if $P$ is self-adjoint, $L_P$ is self-adjoint.
\end{enumerate}
Thus, to linearize a monomial $P=X_{i_1}X_{i_2}X_{i_3}\cdots X_{i_{k-1}}X_{i_l}$, write
$$
L_P=-\begin{bmatrix}
0 & 0 & \cdots & 0 & 0 & X_{i_1}\\
0 & 0 & \cdots & 0 & X_{i_2} & -1\\
0 & 0 & \cdots & X_{i_3} & -1 & 0\\
\vdots&\vdots& \cdots&\vdots&\vdots&\vdots\\
0 & X_{i_{l-1}}&\cdots&0&0&0\\
X_{i_l}&-1&\cdots&0&0&0
\end{bmatrix},
$$
with the obvious adaptations if $l=1,2$. The $(l-1)\times(l-1)$ lower right corner of the 
above matrix is invertible in the algebra $M_{l-1}(\mathbb C\langle X_1,\dots,X_k\rangle)$ and 
its inverse has as entries polynomials of degree up to $l-1$ (again with the obvious modifications when 
$l\leq 2$). The constant term in the inverse's formula is simply the matrix having $-1$ on its second 
diagonal, and its spectrum included in $\{-1,1\}$.
 If the matrices $\begin{bmatrix}
0 & u_j\\
v_j & Q_j\end{bmatrix},$ 
with $u_j\in M_{1\times n_j}(\mathbb C\langle X_1,\dots,X_k\rangle),v_j\in M_{n_j\times1}
(\mathbb C\langle X_1,\dots,X_k\rangle),Q_j\in M_{n_j}(\mathbb C\langle X_1,\dots,X_k\rangle)$
are linearizations for $P_j,$ $j=1,2$, then 
$$
L_{P_1+P_2}=\begin{bmatrix}
0 & u_1 & u_2\\
v_1 & Q_1 & 0\\
v_2 & 0 &Q_2
\end{bmatrix}\in M_{n_1+n_2+1}(\mathbb C\langle X_1,\dots,X_k\rangle).
$$
In particular, we have again that $\begin{bmatrix} Q_1 & 0\\ 0 & Q_2\end{bmatrix}^{-1}\in
M_{n_1+n_2}(\mathbb C\langle X_1,\dots,X_k\rangle),$ with entries of degrees no more $\max\{n_1,n_2\}$. 
Again, its constant term has all its eigenvalues of absolute value equal to one.
The above construction does not necessarily provide a self-adjoint $L_P$, even if
$P=P^*$. However, any self-adjoint polynomial $P$ is written as a sum $P=P_0+P_0^*$ for some 
other polynomial $P_0$ (self-adjoint or not) of the same degree. Let $L_{P_0}=\begin{bmatrix}
0 & u_0\\
v_0 & Q_0
\end{bmatrix}.$  To insure that the linearization we obtain is self-adjoint, we write
$$
L_P=L_{P_0+P_0^*}=\begin{bmatrix}
0 & u_0 & v_0^*\\
u_0^* & 0 & Q_0^* \\
v_0 & Q_0 & 0
\end{bmatrix},
$$
which satisfies $L_P=L_P^*$ and linearizes $P$. Moreover, since $\begin{bmatrix} 0 & Q_0^* \\
Q_0 & 0
\end{bmatrix}^{-1}=\begin{bmatrix} 0 & Q_0^{-1} \\
(Q_0^*)^{-1} & 0
\end{bmatrix}$, the matrix $\begin{bmatrix} 0 & Q_0^* \\
Q_0 & 0
\end{bmatrix}^{-1}$ has entries which are polynomials in $X_1,\dots,X_k$, all of them of degree 
majorized by the degree of $P$, and its constant term is a complex matrix having 
spectrum included in the unit circle of the complex plane. These remarks, which the reader can find in
\cite{Mai}, will be most useful in our analysis below.

It follows from the above that if $X_1,\dots,X_k$ are elements in some complex algebra $\mathcal R$ 
with unit $1$, then $z1-P(X_1,\dots,X_k)$ is invertible in $\mathcal R$ if and only if $z\hat{E}_{11}-L_{P}
(X_1,\dots,X_k)$ is invertible in $M_m(\mathcal R)$ ($m$ being the size of the matrix $L_P$). Moreover, 
the construction above guarantees that the matrix $Q$ in the linearization $L_{P}(X_1,\dots,X_k)
=\begin{bmatrix}
0 & u\\
v & Q
\end{bmatrix}$ is invertible independently of the elements $X_1,\dots,X_k\in\mathcal R$.
We apply this to the case when $\mathcal R\subseteq\mathcal B(\mathcal H)$ is 
a unital ${\cal C}$${}^*$-algebra of bounded linear operators on the Hilbert space  $\mathcal H$, 
for various separable Hilbert spaces $\mathcal H$. We formalize this result in the following

\begin{lemme}\label{inversible}
Let  $P=P^*\in\mathbb{C}\langle X_1,\ldots,X_k\rangle$ and let
$L_P \in M_m(\mathbb{C}\langle X_1,\ldots,X_k\rangle)$
be a linearization of P with the properties outlined above. Let $y = (y_1,\ldots, y_k)$  be a k-tuple of self-adjoint operators in a ${\cal C}^*$-algebra ${\cal A}$. Then, for any $z\in \mathbb{C}$,  $z\hat E_{11}\otimes 1_{\cal A}-L_P(y)$
is invertible if and only if $z 1_{\cal A}-P(y)$ is invertible.
\end{lemme}

Beyond the property described above, we want also to compare the norms of the inverses of 
$z\hat E_{11}\otimes 1_{\cal A}-L_P(y)$ and $z 1_{\cal A}-P(y)$ when one (and hence the other) exists.

\begin{lemme}\label{distanceauspectre} Let  $P=P^*\in\mathbb{C}\langle X_1,\ldots,X_k\rangle$ and let
$L_P \in M_m(\mathbb{C}\langle X_1,\ldots,X_k\rangle)$
be a linearization of P constructed as above. Let $y_n = (y_1^{(n)},\ldots, y_k^{(n)})$  be a k-tuple of self-adjoint operators in a ${\cal C}^*$-algebra ${\cal A}$ such that $\sup_n \max_{i=1}^k \Vert y_n^{(i)}\Vert=C<+\infty$. Let $z_0 \in \mathbb{C}$ be such that, for all large $n$,  the distance from $z_0$
to $sp(P(y_n))$ is greater than $\delta$. Then, there exists a constant $\epsilon > 0$, depending only on $ \delta$, $L_P$ and $C$ such that the distance from $0$ to $sp(z_0\hat E_{11}\otimes 1_{\cal A}-L_P(y_n))$
is at least  $\epsilon$.
\end{lemme}
\begin{proof}
In this proof we only consider $P,u$ and $Q$ evaluated in $y_n$, so we will suppress $y_n$ from the 
notation without any risk of confusion. Let $L_P=\begin{bmatrix} 0 & u^*\\ u & Q\end{bmatrix}$,
so that $z_0\hat E_{11}\otimes 1_{\cal A}-L_P=\begin{bmatrix} z_0 &- u^*\\ -u & -Q\end{bmatrix}$, as 
above. We seek an $\epsilon>0$ such that 
$z_0\hat E_{11}\otimes 1_{\cal A}-L_P-z(I_m\otimes 1_{\cal A})=\begin{bmatrix} z_0-z &- u^*\\ -u & -Q-
z\end{bmatrix}$ is invertible for all $|z|<\epsilon$. We naturally require first that $-Q-z$ remains
invertible. As $Q=Q^*$, it follows by functional calculus that the spectrum of $-Q$ is at 
distance equal to $\|Q^{-1}\|^{-1}$ from zero. As noted before, $Q^{-1}\in M_{m-1}(\mathbb C
\langle y_1^{(n)},\dots,y_k^{(n)}\rangle)$, with entries depending only on $P$, so that we can 
majorize $\|Q^{-1}\|$ by a constant $\kappa>0$ depending only on $C$ and $L_P$ (and independent of
the particular $y_n$). Thus, our first condition on $\epsilon$ is $\epsilon\leq\kappa^{-1}/2$. 
Note that it follows that $\|(Q+z)^{-1}\|< 2 \kappa.$
Next,
we require that in addition $(z_0-z-u^*(-Q-z)^{-1}u)$ is invertible for all $|z|<\epsilon$. By the openness
of the resolvent set, we do know that such an $\epsilon>0$ exists. More precisely, by a geometric series 
argument, if $a$ is invertible, then $b=((b-a)a^{-1}+1)a$ is invertible whenever $\|b-a\|<\|a^{-1}\|^{-1}
$. We apply this to $a=z_0-P=z_0+u^*Q^{-1}u$ (so that $\|a^{-1}\|^{-1}>\delta$) and 
$b=z_0-z-u^*(-Q-z)^{-1}u$.
We have \\

\noindent $\|z_0+u^*Q^{-1}u-z_0+z+u^*(-Q-z)^{-1}u\|$ 
\begin{eqnarray*}
& \leq  &|z|+ \|u^*[(-Q-z)^{-1}+Q^{-1}]u\|\\
& \leq &|z|+ |z|\|u\|^2\|Q^{-1}\|\|(Q+z)^{-1}\|\\
&\leq & |z|(1+ 2\kappa^2 \|u\|^2).
\end{eqnarray*}
Since the norm of  $u$ is majorized in terms of $C$ and $L_P$ only by 
some  constant $\ell>0$, we deduce that 
 $\|z_0+u^*Q^{-1}u-z_0+z+u^*(-Q-z)^{-1}u\|\leq |z|(1+ 2\kappa^2 \ell^2).$
Thus, we require $|z|(1+ 2\kappa^2 \ell^2)<\delta$.  
 This yields that , if $|z|<\min\{\kappa^{-1}/2,\frac{\delta}{(1+2
\kappa^2\ell^2)}\}$, then $z_0\hat E_{11}\otimes 1_{\cal A}-L_P-z(I_m\otimes 1_{\cal A})$ is invertible. This
concludes the proof of our lemma.

\end{proof}

\subsection{From Lemma \ref{inclu2} to Theorem  \ref{noeigenvalue}}

Let $a_n=(a_n^{(1)},\ldots,a_n^{(t)})$ be a t-tuple of noncommutative self-adjoint  random variables which 
is free with the semicircular system $x=(x_1,\ldots,x_r)$ in $({\cal A},\tau)$, such that the distribution of  
$a_n$  in $({\cal A},\tau)$ coincides with the distribution of $(A_n^{(1)},\ldots, A_n^{(t)})$ in $({ M}_n(\mathbb{C}),\tr_n)$.
Let $P$ be a Hermitian  polynomial  in t+r noncommutative indeterminates.
Let
$L_P\in M_m\left(\mathbb{C}\langle X_1,\ldots,X_{t+r}\rangle\right)$ be a linearization of $P$ as constructed in Section \ref{linearisation}.  Fix $\delta>0$ and let $z\in \mathbb{R}$ be such that  for all large $n$, the distance from $z$ to the spectrum of  $\left(P\left(x_1,\ldots,x_r, a_n^{(1)},
\ldots,a_n^{(t)}\right)\right)$ is greater than $\delta$. According to Lemma  \ref{distanceauspectre} (and 
\eqref{normeAn}),  there exists a constant $\epsilon > 0$, depending only on $\delta,L_P$ and $\sup_{n}
\max_{1\le u\le t}\|A_n^{(u)}\|$ such that the distance from 0 to $sp(z\hat E_{11}\otimes 1_{\cal A}-
L_P(x_1,\ldots,x_r, a_n^{(1)},\ldots,a_n^{(t)}))$ is as least $\epsilon$.  Now, according to Lemma 
\ref{inclu2}, almost surely, for all large $n$, the distance from 0 to  the spectrum of $(z\hat E_{11}\otimes I_n-L_P(\frac{X_n^{(1)}}{\sqrt{n}}, \ldots, \frac{X_n^{(r)}}{\sqrt{n}}, A_n^{(1)},\ldots,A_n^{(t)}))$ is as least 
$\epsilon/2$. Hence, 
for any $z'\in ]z-\epsilon/4;z+\epsilon/4[$, 0 is not in the spectrum of $(z' \hat E_{11}\otimes I_n-
L_P(\frac{X_n^{(1)}}{\sqrt{n}}, \ldots, \frac{X_n^{(r)}}{\sqrt{n}}, A_n^{(1)},\ldots,A_n^{(t)}))$.
 Finally, according to Lemma \ref{inversible}, almost surely, for large $n$, $]z-\epsilon/4;z+\epsilon/4[$ lies outside the spectrum of $P(\frac{X_n^{(1)}}{\sqrt{n}}, \ldots, \frac{X_n^{(r)}}{\sqrt{n}}, A_n^{(1)},\ldots,A_n^{(t)})$.
A compacity argument readily yields Theorem \ref{noeigenvalue}.

\section{Proof of \eqref{safbis}}\label{strategie}

 Our approach is then very similar to that of \cite{HT} and \cite{Schultz05}. Therefore, we will recall the main steps.\\ First,  the almost sure  minoration
$$ \liminf_{n \vers +\infty} \left\| P\left(\frac{X_n^{(1)}}{\sqrt{n}}, \ldots, \frac{X_n^{(r)}}{\sqrt{n}},A_n^{(1)},\ldots,A_n^{(t)}\right) \right\| \geq  \left\| P(x_1, \ldots x_r,a_1,\ldots,a_t)\right\| \  $$
comes rather easily from \eqref{af}; this can be proved by closely following the proof of  Lemma 7.2 in \cite{HT}. So, the main
difficulty is the proof of the almost sure reverse inequality:
\begin{equation}\label{limsup}
\limsup_{n \vers +\infty} \left\| P\left(\frac{X_n^{(1)}}{\sqrt{n}}, \ldots, \frac{X_n^{(r)}}{\sqrt{n}},,A_n^{(1)},\ldots,A_n^{(t)}\right) \right\|\leq  \left\| P(x_1, \ldots x_r,,a_1,\ldots,a_t)\right\|_{\cal A}. \ 
\end{equation}
The proof of
(\ref{limsup}) consists in two steps.\\

  \noindent
{\bf Step 1: A linearization trick} (see Section 2 and the proof of Proposition 7.3 in  \cite{HT}) \\
In order to prove (\ref{limsup}), it is sufficient to prove:
\begin{lemme} \label{inclu} For all $m \in \N$, all self-adjoint matrices $\gamma, \alpha_1, \ldots,\alpha_r, \beta_1, \ldots, \beta_t$ of size $m\times m$ and
all $\epsilon >0$, almost surely for all large $n$,
$$ 
sp(\gamma \otimes I_n + \sum_{v=1}^r \alpha_v \otimes \frac{X_n^{(v)}}{\sqrt{n}}+  \sum_{u=1}^t \beta_u \otimes A_n^{(u)})$$ \begin{equation} \label{spectre}\subset
sp(\gamma \otimes 1_{\cal A} + \sum_{v=1}^r \alpha_v \otimes x_v + \sum_{u=1}^t \beta_u \otimes a_u) + ]-\epsilon, \epsilon[.
\end{equation}

\end{lemme}
\noindent {\bf Step 2: An intermediate inclusion}
Let $a_n=(a_n^{(1)},\ldots,a_n^{(t)})$ be a t-tuple of noncommutative self-adjoint  random variables which is free with the semicircular system $x=(x_1,\ldots,x_r)$ in $({\cal A},\tau$), such that the distribution of  $a_n$ coincides with the distribution of $(A_n^{(1)},\ldots, A_n^{(t)})$ in $({ M}_n(\mathbb{C}),\tr_n)$.
Male \cite{CamilleM} proved that if $(A_n^{(1)},\ldots, A_n^{(t)})$ converges strongly to $(a_1,\ldots, a_t)$, then, for any $\epsilon >0$, for all large $n$,\\

\noindent 
$
sp(\gamma \otimes 1_{\cal A} + \sum_{v=1}^r \alpha_v \otimes x_v + \sum_{u=1}^t \beta_u \otimes a_n^{(u)})$ \begin{equation} \label{spectre2} \subset
sp(\gamma \otimes 1_{\cal A} + \sum_{v=1}^r \alpha_v \otimes x_i + \sum_{u=1}^t \beta_u \otimes a_u) + ]-\epsilon, \epsilon[.
\end{equation}
Therefore,  Lemma \ref{inclu} can be deduced from Lemma \ref{inclu2}.

\section{Appendix}
\subsection{Basic identities and inequalities}
\begin{lemme}\label{majcarre}
For any matrix $M \in M_m(\mathbb{C})\otimes M_n(\mathbb{C})$ 
and for any fixed $k$, we have \begin{equation}\label{O}  \sum_{l =1}^n ||M_{lk} ||^2 \leq m ||M ||^2\end{equation}
(or equivalently \begin{equation}\label{l} \sum_{l =1}^n ||M_{kl} ||^2 \leq m ||M ||^2.)\end{equation}
Therefore, we have 
 \begin{equation}\label{lp}
 \frac{1}{n} \sum_{k,l =1}^n ||M_{kl} ||^2 \leq m ||M ||^2.
 \end{equation}
\end{lemme}
\begin{proof}
Note that \begin{eqnarray*} \sum_{l =1}^n ||M_{lk} ||^2 &\leq& \sum_{l =1}^n ||M_{lk} ||_2^2\\&=& \Tr M ( I_m\otimes E_{kk}) M^*\\&=&\Tr  ( I_m\otimes E_{kk}) M^*M ( I_m\otimes E_{kk})\\&\leq& 
 \Vert M \Vert^2 \Tr ( I_m\otimes E_{kk})=m \Vert M \Vert^2  .\end{eqnarray*}
Now, since  $$ \sum_{l =1}^n ||M_{kl} ||^2=\sum_{l =1}^n ||M_{kl}^* ||^2 = \sum_{l =1}^n ||(M^*)_{lk} ||^2  \; \mbox{and} \; \Vert M^* \Vert= \Vert M \Vert,$$
\eqref{O} and \eqref{l} can be deduced from each other thanks to conjugate transposition.
Finally \eqref{l} readily yields \eqref{lp}.
\end{proof}
\begin{lemme}
Let $k\geq 1$. Let $M^{(0)}, M^{(1)},\ldots, M^{(k)}, M^{(k+1)}$, be $nm\times nm$ matrices  depending on $\lambda \in \{\rho\in M_m(\mathbb{C}), \Im \rho>0\}$ such that 
$\forall w=0, \ldots, k+1, \; \left\| M^{(w)} \right\| =O(1).$ Assume that for any $ (i,l)\in \{1,\ldots,n\}^2$,
$z_{i,l}$ are complex numbers such that  $\sup_{i,l} \vert z_{i,l} \vert \leq C$ for some constant $C$ 
and $\{ i_w(i,l) \}_{w=1,\ldots,k+1}$ and 
 $\{j_w(i,l)\}_{w=0,\ldots,k}$ are equal to either $i$ or $l$.\\
 Then, 
 \begin{itemize} \item for any $(p,q)\in \{1,\ldots,n\}^2,$ \begin{equation}\label{Oden}\sum_{i,l=1}^n z_{i,l} M^{(0)}_{pj_0} M^{(1)}_{i_1 j_1}\cdots M^{(k)}_{i_k j_k} M^{(k+1)}_{i_{k+1}q}=O_{p,q}^{(u)}(n),\end{equation}
\item  if there exists $w_0\in \{1,\ldots,k\}$ such that $(i_{w_{0}}, j_{w_0})\in \{(i,l),(l,i)\}$, then for any $(p,q)\in \{1,\ldots,n\}^2,$ \begin{equation}\label{Oderacine}\sum_{i,l=1}^n z_{i,l} M^{(0)}_{pj_0} M^{(1)}_{i_1 j_1}\cdots M^{(k)}_{i_k j_k} M^{(k+1)}_{i_{k+1}q} =O_{p,q}^{(u)}(\sqrt{n}),\end{equation}
\item  if there exists $w_0\in \{1,\ldots,k\}$ such that $(i_{w_{0}}, j_{w_0})\in \{(i,l),(l,i)\}$, then
 \begin{equation}\label{Odenracinepas}\sum_{i,l=1}^n z_{i,l}  M^{(1)}_{i_1 j_1}\cdots M^{(k)}_{i_k j_k}  =O(n\sqrt{n}),\end{equation}
 \item  if there exists $(w_0, w_1)\in \{1,\ldots,k\}^2$, $w_0\neq w_1$,  such that 
 $\{(i_{w_{0}}, j_{w_0}), (i_{w_{1}}, j_{w_1})\}$ is a subset of $\{(i,l),(l,i)\}^2$
 then 
  \begin{equation}\label{Odenpas}\sum_{i,l=1}^n z_{i,l}  M^{(1)}_{i_1 j_1}\cdots M^{(k)}_{i_k j_k}  =O(n).\end{equation}
\end{itemize}
\end{lemme}
\begin{proof}
If $(j_0, i_{k+1}) \in \{(i,l),(l,i)\}$, noticing (using Lemma \ref{majcarre}) that \begin{eqnarray*}\sum_{i,l=1}^n  \left\|M^{(0)}_{pi} \right\| \left\| M^{(k+1)}_{lq}\right\|
&\leq& \left( \sum_{i=1}^n  \left\|M^{(0)}_{pi} \right\|^2 \right)^{1/2} \sqrt{n}  \left( \sum_{l=1}^n  \left\|M^{(k+1)}_{lq} \right\|^2 \right)^{1/2} \sqrt{n} \\&=&O_{p,q}^{(u)}(n),
\end{eqnarray*}
\eqref{Oden} follows. \\
Now, if $(j_0, i_{k+1}) \in \{(i,i),(l,l)\}$, noticing  (using Lemma \ref{majcarre}) that \begin{eqnarray*}\sum_{i,l=1}^n  \left\|M^{(0)}_{pi} \right\| \left\| M^{(k+1)}_{iq}\right\|
&\leq& n \left( \sum_{i=1}^n  \left\|M^{(0)}_{pi} \right\|^2 \right)^{1/2}  \left( \sum_{i=1}^n  \left\|M^{(k+1)}_{iq} \right\|^2 \right)^{1/2} 
 \\&=&O_{p,q}^{(u)}(n),
\end{eqnarray*}
\eqref{Oden} follows. The proof of \eqref{Oden} is complete.\\

\noindent Now, assume that  there exists $w_0\in \{1,\ldots,k\}$ such that $(i_{w_{0}}, j_{w_0})\in \{(i,l),(l,i)\}$. Let $\tilde  M^{(w_0)}$ be either $ M^{(w_0)}$ or $ (M^{(w_0)})^*$.\\
If $(j_0, i_{k+1}) \in \{(i,l),(l,i)\}$, noticing (using Lemma \ref{majcarre}) that
\\

$\sum_{i,l=1}^n  \left\|M^{(0)}_{pi} \right\|\left\| \tilde M^{(w_0)}_{il}\right\| \left\| M^{(k+1)}_{lq}\right\|$  \begin{eqnarray*}
&\leq&  \sum_{l=1}^n  \left\|M^{(k+1)}_{lq} \right\| \left( \sum_{i=1}^n  \left\|M^{(0)}_{pi} \right\|^2 \right)^{1/2}   \left( \sum_{i=1}^n \left\|\tilde M^{(w_0)}_{il} \right\|^2 \right)^{1/2} \\&\leq & \left( \sum_{i=1}^n  \left\|M^{(0)}_{pi} \right\|^2 \right)^{1/2} 
\left( \sum_{i=1}^n  \left\|M^{(k+1)}_{lq} \right\|^2 \right)^{1/2} \left( \sum_{i,l=1}^n \left\|\tilde M^{(w_0)}_{il} \right\|^2 \right)^{1/2}
 \\&=&O_{p,q}^{(u)}(\sqrt{n}),
\end{eqnarray*}
\eqref{Oderacine} follows. \\
Now, if $(j_0, i_{k+1}) \in \{(i,i),(l,l)\}$, noticing  (using Lemma \ref{majcarre}) that \\

\noindent $\sum_{i,l=1}^n  \left\|M^{(0)}_{pi} \right\| \left\| \tilde M^{(w_0)}_{il}\right\| \left\| M^{(k+1)}_{iq}\right\|$ \begin{eqnarray*}
&\leq& n  \left\| \tilde M^{(w_0)}\right\| \left( \sum_{i=1}^n  \left\|M^{(0)}_{pi} \right\|^2 \right)^{1/2}  \left( \sum_{i=1}^n   \left\|M^{(k+1)}_{iq} \right\|^2 \right)^{1/2} 
 \\&=&O_{p,q}^{(u)}(n),
\end{eqnarray*}
\eqref{Oderacine} follows. The proof of \eqref{Oderacine} is complete.\\

\noindent  Assume that  there exists $w_0\in \{1,\ldots,k\}$ such that $(i_{w_{0}}, j_{w_0})\in \{(i,l),(l,i)\}$; then  noticing (using Lemma \ref{majcarre}) that \begin{eqnarray*} 
\sum_{i,l=1}^n \left\| \tilde M^{(w_0)}_{il} \right\|& \leq& n  \left( \sum_{i,l=1}^n \left\|\tilde M^{(w_0)}_{il} \right\|^2 \right)^{1/2}\\&=& O(n\sqrt{n}),
\end{eqnarray*}
\eqref{Odenracinepas} follows.\\

\noindent Now, assume that there exists $(w_0, w_1)\in \{1,\ldots,k\}^2$, $w_0\neq w_1$,  such that 
 $\{(i_{w_{0}}, j_{w_0}), (i_{w_{1}}, j_{w_1})\}$ is a subset of $\{(i,l),(l,i)\}^2$.  Let for $h=0,1$, $\tilde  M^{(w_h)}$ be either $ M^{(w_h)}$ or $ (M^{(w_h)})^*$; then  noticing (using Lemma \ref{majcarre}) that \begin{eqnarray*} 
\sum_{i,l=1}^n \left\|  \tilde M^{(w_0)}_{il} \right\| \left\|  \tilde M^{(w_1)}_{il} \right\|& \leq&   \left( \sum_{i,l=1}^n \left\|\tilde M^{(w_0)}_{il} \right\|^2 \right)^{1/2} \left( \sum_{i,l=1}^n \left\| \tilde  M^{(w_1)}_{il} \right\|^2 \right)^{1/2}\\&=& O(n),
\end{eqnarray*}
\eqref{Odenpas} follows.\\
\end{proof}
 We end  by recalling some properties of  resolvents.
 First, one can easily see that for any $\lambda$ and $\lambda^{'}$ in $M_m(C)$ such that $\Im(\lambda)$ and $\Im(\lambda^{'})$ are positive
definite,
 \begin{equation}\label{difStieljes}(\lambda \otimes 1_{{\cal A}} - s)^{-1}-(\lambda^{'} \otimes 1_{{\cal A}} -
 s)^{-1}=(\lambda \otimes 1_{{\cal A}} - s)^{-1}(\lambda^{'} -
 \lambda)(\lambda^{'} \otimes 1_{{\cal A}} - s)^{-1}.
 \end{equation}
For a Hermitian matrix $M$, the derivative w.r.t $M$ of the resolvent $R(z) = (z-M)^{-1}$ satisfies:
 \begin{equation} \label{resolvente}
 R'_M(z) . H = R(z) H R(z)   \mbox{ for all  Hermitian matrix $H$}. \end{equation}
\begin{lemme} \label{G}
Let $\lambda$ in $M_m(C)$ such that $\Im(\lambda)$ is positive
definite and $h$ be a self-adjoint element in $M_m(\C)\otimes {\cal A}$ where ${\cal A}$ is a   ${\cal C}^*$-algebra endowed with some state $\tau$. Then
 \begin{equation}\label{normeG}
 \Vert (\lambda \otimes 1_{{\cal A}} - h)^{-1}\Vert \leq ||
 \Im(\lambda)^{-1}\Vert \mbox{~~and~~~}
  || G(\lambda ) || \leq || \Im(\lambda)^{-1}
  ||,\end{equation}
where  $G(\lambda) = ({\rm id}_m \otimes \tau) [ (\lambda \otimes 1_{{\cal A}} - h)^{-1}].$
\end{lemme}

 \begin{lemme} \label{lem2}
 Let $\lambda$ in $M_m(C)$ such that $\Im(\lambda)$ is positive definite, then for any $mn\times mn$ Hermitian matrix $H$
 \begin{equation}\label{norme}
  || (\lambda \otimes I_n - H)^{-1} || \leq || \Im(\lambda)^{-1}
  ||,\end{equation}
 $$  \forall 1 \leq k,l \leq n, || (\lambda \otimes I_n - H)^{-1}_{kl} || \leq || \Im(\lambda)^{-1} ||,$$
  and for $p \geq 2$,
 \begin{equation}\label{pplusgrand}
 \frac{1}{n} \sum_{k,l =1}^n ||(\lambda \otimes I_n - H)^{-1}_{kl} ||^p \leq m ||\Im(\lambda)^{-1} ||^p.
 \end{equation}

 \end{lemme}
 \subsection{Variance estimates}
We refer the reader to the book \cite{Tou}. 
A probability measure $\mu$ on $\mathbb{R}$ is said to satisfy the Poincar\' e inequality with constant $C_{PI}$ if
 for any
${\cal C}^1$ function $f: \R\rightarrow \C$  such that $f$ and
$f' $ are in $L^2(\mu)$,
$$\mathbf{V}(f)\leq C_{PI}\int  \vert f' \vert^2 d\mu ,$$
\noindent with $\mathbf{V}(f) = \int \vert
f-\int f d\mu \vert^2 d\mu$. \\

\begin{remarque}\label{multiple} If the law of a random variable $X$ satisfies the Poincar\'e inequality with constant $C_{PI}$ then, for any fixed $\alpha \neq 0$, the law of $\alpha X$ satisfies the Poincar\'e inequality with constant $\alpha^2 C_{PI}$.\\
Assume that  probability measures $\mu_1,\ldots,\mu_M$ on $\mathbb{R}$ satisfy the Poincar\'e inequality with constant $C_{PI}(1),\ldots,C_{PI}(M)$ respectively. Then the product measure $\mu_1\otimes \cdots \otimes \mu_M$ on $\mathbb{R}^M$ satisfies the Poincar\'e inequality with constant $\displaystyle{C_{PI}^*=\max_{i\in\{1,\ldots,M\}}C_{PI}(i)}$ in the sense that for any differentiable function $f$ such that $f$ and its gradient ${\rm grad} f$ are in $L^2(\mu_1\otimes \cdots \otimes \mu_M)$,
$$\mathbf{V}(f)\leq C_{PI}^* \int \Vert {\rm grad} f \Vert_2 ^2 d\mu_1\otimes \cdots \otimes \mu_M$$
\noindent with $\mathbf{V}(f) = \int \vert
f-\int f d\mu_1\otimes \cdots \otimes \mu_M \vert^2 d\mu_1\otimes \cdots \otimes \mu_M$ (see Theorem 2.5 in  \cite{GuZe03}) .
\end{remarque}


\begin{lemme}\label{zitt}[Theorem 1.2 in \cite{BGMZ}]
Assume that the distribution of  a random variable $X$ is  supported in  $[-C;C]$ for some constant $C>1$. Let $g$ be an independent standard real  Gaussian random variable. Then  $X+\delta g$ satisfies a Poincar\'e inequality with constant 
$C_{PI}\leq \delta^2 \exp \left( 4C^2/\delta^2\right)$.
\end{lemme}

Consider the linear isomorphism $\Psi_0$ between $M_n(\C)_{sa}$ and $ \mathbb{R}^{n^2}$ defined for any  $[b_{kl}]_{k,l=1}^{n} \in M_n(\C)_{sa}$  by $$\Psi_0([b_{kl}]_{k,l=1}^{n}) = ((b_{kk})_{1\leq k\leq n}, (\sqrt{2} \Re  (b_{kl}))_{1\leq k<l \leq n},  (\sqrt{2} \Im  (b_{kl}))_{1\leq k<l \leq n}).$$ Denote by $\Psi$ the natural extension of $\Psi_0$ to a linear isomorphism between $(M_n(\C)_{sa})^r$ and $ \mathbb{R}^{rn^2}$ defined for any $(B_1, \ldots, B_r)$ in $(M_n(\C)_{sa})^r$ by 
$$\Psi (B_1,\ldots, B_r)=(\Psi_0(B_1),\ldots, \Psi_0(B_r)).$$
 \begin{lemme}\label{variance}
Consider   r  independent $n\times n$ random  Hermitian matrices $X_n^{(v)} =    [X^{(v)}_{jk}]_{j,k=1}^n$, $v=1,\ldots,r$,   such that  the  random variables $ X^{(v)}_{ii}$,
$ \sqrt{2}Re(X^{(v)}_{ij})$, $\sqrt{2} Im(X^{(v)}_{ij}),{i<j}$, are independent and satisfies
a  Poincar\'e
inequality with common  constant $C_{PI}$. Let $\tilde f: \mathbb{R}^{rn^2}\rightarrow \mathbb{C}$ be 
a  ${\cal C}^1$ function such
that $\tilde f$ and ${\rm grad} \tilde f $ are  bounded. Consider the ${\cal C}^1$ function $f: (M_n(\C)_{sa})^r \rightarrow \mathbb{C}$ given by $f=\tilde f\circ \Psi$. Then 
\begin{equation}\label{Poincare}\mathbf{V}{f(\frac{X_n^{(1)}}{\sqrt{n}}, \ldots,
\frac{X_n^{(r)}}{\sqrt{n}})} \leq \frac{C_{PI}}{n}\mathbb{E}\{\Vert {\rm grad}  f(\frac{X_n^{(1)}}{\sqrt{n}},
\ldots, \frac{X_n^{(r)}}{\sqrt{n}}) \Vert_e^2 \}\end{equation}
where for any Hermitian matrices $ Z_1,\ldots,Z_r, W_1,\ldots, W_r$, $$\langle (Z_1,\ldots,Z_r), (W_1,\ldots,W_r)\rangle_e = \sum_{v=1}^r \Tr (Z_vW_v).$$
\end{lemme}
\begin{proof}
(\ref{Poincare}) follows from  Remark \ref{multiple}, using that $\Psi$ 
is a linear isometry from  $\left((M_n(\C)_{sa})^r, \langle,\rangle_e\right)$ to 
$ \mathbb{R}^{rn^2}$ with usual Euclidean inner product.
\end{proof}

\noindent Thanks to (\ref{Poincare}),  sticking to
  the proof (from ``(4.6)'' to ``(4.13)'') of Theorem 4.5 in \cite{HT}, we similarly
get the following  variance estimates.
\begin{lemme}\label{var}
Let $X_n^{(v)}, v=1,\ldots,r$ be as defined in Lemma \ref{variance} and $A_n^{(u)}$, $u=1,\ldots,t$, be $n\times n$ deterministic Hermitian matrices such that for any $u=1,\ldots,t$, $\sup_n \left\| A_n^{(u)} \right\| <\infty$. Let $R_N$ and $H_N$ be as defined in \eqref{rn} and \eqref{defSt}.
There exists some constant $C>0$ such that, for any $n\geq 1$ and  for any  $\lambda$  in $M_m(\mathbb{C})$ such that $\Im  \lambda$ is positive definite,
we have, 
   for any deterministic $nm\times nm$ matrices $F_n^{(1)}$ and $F_n^{(2)}$ such that $\Vert F_n^{(1)}\Vert \leq K$ and $ \Vert F_n^{(2)}\Vert \leq K$, for any $(p,q)\in \{1,\ldots,n\}^2$,\\
 
 \noindent 
$ \mathbb{E}\{\Vert{\rm id}_m\otimes tr_n( F_n^{(1)}R_n(\lambda)F_n^{(2)})-\mathbb{E}({\rm id}_m\otimes tr_n(F_n^{(1)}R_n(\lambda)F_n^{(2)}))\Vert^2\}$ \begin{equation}\label{varhn}\leq\frac{K^4C m^3}{n^2}
 \Vert (\Im (\lambda))^{-1}\Vert^4,\end{equation}
 $ \mathbb{E}\{\Vert (F_n^{(1)}R_n(\lambda)F_n^{(2)})_{pq}-\mathbb{E}((F_n^{(1)}R_n(\lambda)F_n^{(2)})_{pq})\Vert^2\}$ $$~~~~~~~~~~\leq\frac{K^4C m^3}{n}
 \Vert (\Im(\lambda))^{-1}\Vert^4.$$

 \end{lemme}

{\bf Acknowledgements.} The authors  wish to thank an anonymous referee for pertinent  comments
which led to an improvement of this paper.


\begin{thebibliography}{lll}
\bibitem{AK} N. Akhiezer. The Classical Moment Problem. Hafner, New York, 1965.
\bibitem{A} G.W. Anderson. Convergence of the largest singular value of a polynomial in independent Wigner matrices. {\em Ann. Probab.} 41, no. 3B  (2013), 2103-2181. 

\bibitem{AGZ}
G.~Anderson, A.~Guionnet, and O.~Zeitouni.
\newblock {\em An Introduction to Random Matrices}.
\newblock Cambridge University Press, 2009.

\bibitem{Tou} \textsc{C. An\'e, S. Blach\`ere, D. Chafa\"\i, P. Foug\`eres, I. Gentil,
F. Malrieu F, C. Roberto, G.  Scheffer}. {\it Sur les in\'egalit\'es de
Sobolev logarithmiques} (French) [Logarithmic Sobolev inequalities],
Panoramas et Synth\`eses [Panoramas and Syntheses] {\bf{10}}, S.M.F
Paris, 2000.

\bibitem{ABFN} M. Anshelevich, S.T. Belinschi, M F\'evrier, and A. Nica. 
\newblock Convolution powers in the operator-valued framework.
\newblock {\em Trans. Amer. Math. Soc.}, {\bf365}, no. 4 (2013), 2063--2097. 

\bibitem{BaiSil98} Z. Bai and J. Silverstein. No eigenvalues outside the support of the limiting spectral distribution
of large-dimensional sample covariance matrices. {\em  Ann. Probab.} 26 , no. 1 (1998),
316--345.


\bibitem{BaiSil06}
Z.~D. Bai and J.~W. Silverstein.
\newblock Spectral Analysis of of large-dimensional random matrices.
\newblock Mathematics Monograph Series 2, Science Press Beijing 2006.


\bibitem{BaiSil12}
Z. Bai and J. Silverstein. No eigenvalues outside the support of the limiting spectral distribution
of information-plus-noise type matrices. {\em Random Matrices Theory Appl.} 1,
no. 1 (2012), 1150004, 44 pp.

\bibitem{BaiYin}Z. D.
Bai,  and  Y. Q. Yin,
Necessary and sufficient conditions for almost sure convergence of the largest eigenvalue of a Wigner matrix.
Ann. Probab. 16 (1988), no. 4, 1729-1741.


\bibitem{BGMZ}
J.B. Bardet, N. Gozlan, F. Malrieu and  P.-A. Zitt. Functional inequalities for Gaussian convolutions of compactly supported measures: explicit bounds and dimension dependence.  {\em  Bernoulli Journal} Vol. 24, Number 1 (2018), 333-353.

\bibitem{Serban}S.T. Belinschi. A noncommutative version of the Julia-Wolff-Caratheodory Theorem. arxiv: 1510.05377 (2016).

\bibitem{BMS} S. T. Belinschi, T. Mai and R. Speicher, {\em Analytic subordination theory 
of operator-valued free additive convolution and the solution of a general random matrix problem}.
Journal für die reine und angewandte Mathematik (2015) DOI 10.1515/crelle-2014-0138


\bibitem{Bruce} B. Blackadar, Operator Algebras. Theory of C${}^*$-Algebras and von Neumann Algebras. 
Encyclopaedia of Mathematical Sciences, Volume 122. Springer-Verlag Berlin Heidelberg 2006.



\bibitem{CD07}
M.~Capitaine and C.~Donati-Martin.
\newblock Strong asymptotic freeness for {W}igner and {W}ishart matrices.
\newblock {\em Indiana Univ. Math. J.}, 56(2) (2007):767--803.


\bibitem{CDMFF} M. Capitaine, C. Donati-Martin, D. F\'eral and M. F\'evrier. Free convolution with a semicircular
distribution and eigenvalues of spiked deformations of Wigner matrices. Electron. J.
Probab. 16 no. 64  (2011), 1750--1792.

\bibitem{MCo} 
{B. Collins and C. Male}.
     {The strong asymptotic freeness of {H}aar and deterministic
              matrices},
    {\em Ann. Sci. \'Ec. Norm. Sup\'er.} (4)
 47 no. 1 (2014): {147--163}.

\bibitem{CSBD} 
  R. Couillet, Silverstein, J. W., Z. Bai and
             M.  Debbah.
     {Eigen-inference for energy estimation of multiple sources},
    {\em IEEE Trans. Inform. Theory},
{57},
{4} (2011){:2420--2439}.
\bibitem{D} K.  Dykema.
On certain free product factors via an extended matrix model. (English summary)
{\em J. Funct. Anal. 112}, no. 1 (1993), 31–60. 



\bibitem{GuZe03}
A.~Guionnet and B.~Zegarlinski.
\newblock Lectures on Logarithmic Sobolev inequalities.
\newblock In {\em S\'eminaire de {P}robabilit\'es, {XXXVI}}, volume 1801 of
  {\em Lecture Notes in Math.}. Springer, Berlin, 2003.


\bibitem{HT}
U.~Haagerup and S.~Thorbj{\o}rnsen.
\newblock A new application of random matrices: {${\rm Ext}(C^*_{\rm
  red}(F_2))$} is not a group.
\newblock {\em Ann. of Math.} (2) (2005), 162(2):711-775.

\bibitem{HRS} J. W. Helton, R. Rashidi Far and R. Speicher, {\em 
Operator-valued Semicircular Elements: Solving A Quadratic Matrix Equation with Positivity Constraints}.
International Mathematics Research Notices, Vol. 2007, Article ID rnm086, 15 pages.

\bibitem{KKP}
A. M. Khorunzhy, B.A. Khoruzhenko and L.A. Pastur. Asymptotic properties of large random matrices with independent entries. {\em  J. Math. Phys. } 37  no. 10 (1996), 5033–5060.



\bibitem{KVV} D. S. Kaliuzhnyi-Verbovetskyi and V. Vinnikov, \emph{Foundations of free 
noncommutative function theory}, Mathematical Surveys and Monographs, 199. American 
Mathematical Society, Providence, RI, 2014.


\bibitem{L}
M. Ledoux. The Concentration of Measure Phenomenon, Mathematical Surveys and Monographs,
Volume 89, A.M.S, 2001.







\bibitem{Mai} T. Mai, {\em } On the analytic theory of non-commutative distributions in free probability, PhD thesis Universität des Saarlandes, (2017), http://scidok.sulb.uni-saarland.de/volltexte/2017/6809.

\bibitem{CamilleM}
C. Male.
The norm of polynomials in large random and
deterministic matrices.
With an appendix by Dimitri Shlyakhtenko.
{\em Probab. Theory Related Fields} 154, no. 3-4 (2012), 477-532. 


\bibitem{NSS} A. Nica, D. Shlyakhtenko, and R. Speicher, {\em Operator-Valued Distributions. I. Characterizations of Freeness}. International Mathematics Research Notices, {\bf 29} (2002),
1509--1538.

\bibitem{Schultz05}
H.~Schultz.
\newblock Non-commutative polynomials of independent {G}aussian random
  matrices. {T}he real and symplectic cases.
\newblock {\em Probab. Theory Related Fields}, 131(2)) (2005):261-309.



\bibitem{SMem} R. Speicher, {\em Combinatorial theory of the free product with amalgamation and operator-valued free probability theory}. Mem. Amer. Math. Soc., vol. 132, no. 627 (1998), pp. x+88.

\bibitem{taylor} J.L.~Taylor, \emph{Functions of several noncommuting variables}, Bull. Amer. Math. Soc. 
\textbf{79} (1973), 1-34.

\bibitem{T} S. Thorbj{\o}rnsen. Mixed Moments of Voiculescu’s Gaussian Random Matrices. {\em  J. Funct. Anal.}
176 (2000), 213-246.


\bibitem{Tillmann53}
H.-G. Tillmann.
\newblock Randverteilungen analytischer {F}unktionen und {D}istributionen.
\newblock {\em Math. Z.}, 59 (1953):61-83.

\bibitem{V} D.V. Voiculescu. Limit laws for random matrices and free products. {\em Invent. Math.} 104 (1991),
201-220.

\bibitem{V1995}
D.V. Voiculescu. {Operations on certain non-commutative operator-valued
  random variables}.  Ast\'erisque (1995), no.~232, 243--275, Recent advances in
  operator algebras (Orl{\'e}ans, 1992). 
\bibitem{Voiculescu97} 
D. V. Voiculescu., A strengthened asymptotic freeness result for random matrices with applications to free entropy, 
IMRN Internat. Math. Res. Notices
(1998), no. 1, 41--63.
\bibitem{V2000}
D.V. Voiculescu. {The coalgebra of the free difference quotient and free
  probability}, {\em  Internat. Math. Res. Notices} no.~2 (2000), 79-106.


\bibitem{FreeMarkov} 
D.V. Voiculescu, \emph{Analytic subordination consequences of free Markovianity}, Indiana Univ. Math. J. {\bf 51} (2002), 1161--1166.

\bibitem{V1} Voiculescu, D.V.: {Free Analysis Questions I: Duality Transform for the Coalgebra of 
$\partial_{X:B}$}. {\em Int. Math. Res. Notices} No. 16 (2004), 793--822.


\bibitem{VDN} D.V. Voiculescu, K. Dykema, and A. Nica, Free random variables, CRM Monograph Series,
vol. 1, American Mathematical Society, Providence, RI, 1992, ISBN 0-8218-6999-X, A
noncommutative probability approach to free products with applications to random matrices,
operator algebras and harmonic analysis on free groups.


\bibitem{W} J. D. Williams.
 Analytic function theory for operator-valued free probability. {\em Jouran fur die reine und angewandte Mathematik (Crelles Journal)}, 2017(729), pp. 119--149.

\end{thebibliography}
\end{document}